%% file: diss.tex
\documentclass[12pt,twoside]{memoir}

\usepackage[top=1.5in, bottom=1in, textwidth=6in, hmarginratio=1:1]{geometry}
\usepackage{mystyle}

\input{packages}

\input{macros}

\chapterstyle{hline}

\newcommand{\UnderlinedTocSection}[1]{%
  \addtocontents{toc}{\protect\addvspace{1pc}%
         \protect\contentsline {chapter}%
        {\protect\Large\scshape\mdseries #1}{}{}%
        \hfill\par\protect\addvspace{-18pt}%
        \noindent\hrulefill\hspace*{24pc}\par}}

\renewcommand*{\backref}[1]{}
\renewcommand*{\backrefalt}[4]{%
    \ifcase #1 (Not cited.)%
    \or        (Cited on page~#2.)%
    \else      (Cited on pages~#2.)%
    \fi}
    
\hypersetup{
  pdfauthor=Sean Patrick Mandrick,
  pdftitle=Dissertation-Mandrick,
  }

\begin{document}

\thispagestyle{empty}

\pretitle{ \begin{center}}
\title{\vspace{-6pc} \hrulefill \\ {\scshape \LARGE On Inversion Graphs of Permutations \\[6pt]}}
\author{{\small by} \\[4pc] Sean Patrick Mandrick}
\predate{\vfill\begin{center}\normalsize}

\date{\vspace{4pc} \emph{A lightly edited version of}\vspace{5mm} \\ A Dissertation Presented to the Graduate School of the University
of Florida in Partial Fulfillment of the Requirements for the Degree of \\
Doctor of Philosophy \\[1pc]
University of Florida \\[1pc]
2025}
\maketitle
\thispagestyle{empty}

\cleardoublepage 

\thispagestyle{empty}

\begin{center} \itshape
  \vspace*{.4\textheight}
  \input{dedication}
\end{center}

\cleardoublepage

\frontmatter

\cleardoublepage

\chapter*{Acknowledgements}
\addcontentsline{toc}{chapter}{Acknowledgements}
\input{acknowledgements.tex}

\cleardoublepage

\tableofcontents

\cleardoublepage


\cleardoublepage

\listoffigures

\newpage
\cleardoublepage

\chapter*{Abstract}
\addcontentsline{toc}{chapter}{Abstract}
  \input{abstract.tex}

\UnderlinedTocSection{Chapters}

\mainmatter

\cleardoublepage
\typeout{******************}
\typeout{**  Chapter 1   **}
\typeout{******************}
\input{chap-intro}

\cleardoublepage
\typeout{******************}
\typeout{**  Chapter 2   **}
\typeout{******************}
\input{chap-uniqueness}

\cleardoublepage
\typeout{******************}
\typeout{**  Chapter 3   **}
\typeout{******************}
\input{chap-lettericity}

\cleardoublepage
\typeout{******************}
\typeout{**  Chapter 4   **}
\typeout{******************}
\input{chap-reflections}






 
 

 


\cleardoublepage
\backmatter

\input{diss.bbl}
\chapter*{Biography}
\addcontentsline{toc}{chapter}{Biography}
  \input{bio}


\end{document}

%% file: packages.tex
\usepackage{enumerate}
\usepackage{amsthm, amsmath, amssymb}
\usepackage{tabulary}

\usepackage{float}
\usepackage{bm} 
\usepackage[nocompress]{cite}

\input{pp-doodles}

\usepackage{vdraw}
\usepackage{nicefrac}
\usepackage{mathrsfs}
\usepackage[T1]{fontenc}
\usepackage[all]{nowidow}

\usepackage{url}
\usepackage{latexsym}
\usepackage{array}
\usepackage{enumitem}
\usepackage{booktabs}
\usepackage{multirow}
\usepackage{multicol}
\usepackage{rotating}

\usepackage[figure,ruled]{algorithm2e}

\usepackage{filecontents}

\usepackage{comment}
\usepackage{lipsum} 
\usepackage{kantlipsum}

\usepackage{tikz}
\usepackage{tikz-qtree}
  \usetikzlibrary{shapes,snakes,plotmarks,matrix, patterns, arrows, decorations.pathreplacing}
  \tikzstyle{ball}=[circle, draw, inner sep=1pt, color=black]
  \tikzstyle{open_y}=[ball,fill=white]
  \tikzstyle{open}=[ball,fill=white]
  \tikzstyle{open2}=[ball,fill=white,minimum size=.15cm]
  \tikzstyle{closed}=[ball,fill=black]
  \tikzstyle{zball}=[ball,rectangle,minimum size=.1cm]

%% file: pp-doodles.tex
\newcommand\mybullet{\raisebox{-5pt}{\normalsize \ensuremath{\bullet}}}
\newcommand\mycirc{\raisebox{-5pt}{\normalsize \ensuremath{\circ}}}

\makeatletter
\def\absdot{\@ifnextchar[{\@absdotlabel}{\@absdotnolabel}}
	\def\@absdotlabel[#1]#2{%
		\node at #2 {\normalsize \mybullet};
		\node at #2 [below=2pt] {\ensuremath{#1}};
	}
	\def\@absdotnolabel#1{%
		\node at #1 {\normalsize \mybullet};
	}
\def\absdothollow{\@ifnextchar[{\@absdothollowlabel}{\@absdothollownolabel}}
	\def\@absdothollowlabel[#1]#2{%
		\node at #2 {\normalsize \textcolor{white}{\mybullet}};
		\node at #2 {\normalsize \mycirc};
		\node at #2 [below=2pt] {\ensuremath{#1}};
	}
	\def\@absdothollownolabel#1{%
		\node at #1 {\normalsize \textcolor{white}{\mybullet}};
		\node at #1 {\normalsize \mycirc};
	}
\makeatother

\newcommand{\plotperm}[1]{
	\foreach \j [count=\i] in {#1} {
		\absdot{(\i,\j)};
	};
}

\newcommand{\plotpermbox}[4]{
	\draw [darkgray, thick, rounded corners=0.01, line cap=round]
		({#1-0.5}, {#2-0.5}) rectangle ({#3+0.5}, {#4+0.5});
}

\newcommand{\plotpermgraph}[1]{
	\foreach \j [count=\i] in {#1} {
		\foreach \b [count=\a] in {#1} {
			\ifthenelse{\a<\i \AND \b>\j}{\draw (\a,\b)--(\i,\j);}{}
		};
	};
	\plotperm{#1};
}

\newcommand{\plotpermdyckpath}[1]{
	\draw [ultra thick, rounded corners=0.01, line cap=round] (0.5,0.5)
	\foreach \step in {#1} {
		\ifnum\step=1
			-- ++(0,1)
		\else
			-- ++(1,0)
		\fi
	};
}

\newcommand{\plotrotheperm}[1]{
	\foreach \j [count=\i] in {#1} {
		\absdot{(\j, -\i)};
	};
  \pgfmathsetmacro\len{dim({#1})}; 
	\foreach \x in {0,1,...,\len} {
		\draw[darkgray, thin, line cap=round] (0.5,{-\x-0.5})--({\len+0.5},{-\x-0.5});
		\draw[darkgray, thin, line cap=round] ({\x+0.5},-0.5)--({\x+0.5},{-\len-0.5});
	};
}

\newcommand{\plotrotheperminversions}[1]{
	\foreach \pii [count=\i] in {#1} {
		\foreach \pij [count=\j] in {#1} {
			\ifthenelse{\i<\j \AND \pii>\pij}{\absdothollow{(\pij, -\i)}}{}
		};
	};
	\plotrotheperm{#1}
}

\newcommand{\matrixpermwithzeros}[2]{
	\foreach \y [count=\x] in {#2} {
		\foreach \j in {1, 2, ..., #1} {
			\ifthenelse{\j=\y}{
				\node at (\x,\j) {\tiny $\mathbf{1}$};
			}{
				\node at (\x,\j) {\textcolor{darkgray}{\tiny $0$}};
			}
		}
	}
}

\newcommand{\plotdyckpath}[1]{
	\draw[ultra thick, line cap=round] (0.5,0)
	\foreach \step in {#1} {
		\ifnum\step=1
			-- ++(1,1)
		\else
			-- ++(1,-1)
		\fi
	};
}

\newcommand{\arcskinnyplain}[2]{
	\draw (#1,0) arc (180:0:{(#2-#1)/2});
}

\newcommand{\matchsmall}[1]{
	\begin{tikzpicture}[scale=.1, anchor=base]
		\def\h{0};
		\def\maxh{0};
		\foreach \i/\j in {#1} {
			\pgfmathparse{\j-\i};
			\let\h\pgfmathresult;
			\pgfmathifthenelse{\h>\maxh}{\h}{\maxh};
			\global\let\maxh\pgfmathresult;
		};
		\pgftransformyscale{{4.5/\maxh}};
		\foreach \i/\j in {#1} {
			\arcskinnyplain{\i}{\j};
		};
	\end{tikzpicture}
}

\newcommand{\matchpermsmall}[1]{
	\begin{tikzpicture}[scale=.1, anchor=base]
    \pgfmathsetmacro\len{dim({#1})}; 
		\def\h{0};
		\def\maxh{0};
		\foreach \j [count=\i] in {#1} {
			\pgfmathparse{2*\len+1-\j-\i};
			\let\h\pgfmathresult;
			\pgfmathifthenelse{\h>\maxh}{\h}{\maxh};
			\global\let\maxh\pgfmathresult;
		};
		\pgftransformyscale{{4.5/\maxh}};
		\foreach \j [count=\i] in {#1} {
			\arcskinnyplain{\i}{{2*\len+1-\j}};
		};
	\end{tikzpicture}
}

\newcommand{\plotpinsequence}[1]{
	\absdot{(0,0)}{};
	\edef\n{0}
	\edef\s{0}
	\edef\e{0}
	\edef\w{0}
	\edef\x{0}
	\edef\y{0}
	\foreach \pin [remember=\pin as \oldpin (initially 1), count=\i] in {#1} {
		\ifthenelse{\pin=1 \OR \pin=2}{
			\ifthenelse{\oldpin=3}{
				\xdef\x{\number\numexpr\e-1}
			}{
				\xdef\x{\number\numexpr\w+1}
			}
			\ifnum\i=1 
				\pgfmathparse{\e+1}
 				\xdef\e{\pgfmathresult}
			\fi	
		}{ 
			\ifthenelse{\oldpin=1}{
				\xdef\y{\number\numexpr\n-1}
			}{
				\xdef\y{\number\numexpr\s+1}
			}
			\ifnum\i=1 
				\pgfmathparse{\s-1}
 				\xdef\s{\pgfmathresult}
			\fi	
		}
		\ifnum\pin=1 
			\pgfmathparse{\n+2}
 			\xdef\n{\pgfmathresult}		
			\absdot{(\x,\n)}{};
			\ifnum\i>1
				\draw (\x,\n) -- (\x,\y-0.5);
			\else
			\fi
		\fi
		\ifnum\pin=2 
			\pgfmathparse{\s-2}
 			\xdef\s{\pgfmathresult}
			\absdot{(\x,\s)}{};
			\ifnum\i>1
				\draw (\x,\s) -- (\x,\y+0.5);
			\else
			\fi
		\fi
		\ifnum\pin=3 
			\pgfmathparse{\e+2}
 			\xdef\e{\pgfmathresult}
			\absdot{(\e,\y)}{};
			\ifnum\i>1
				\draw (\e,\y) -- (\x-0.5,\y);
			\else
			\fi
		\fi
		\ifnum\pin=4 
			\pgfmathparse{\w-2}
 			\xdef\w{\pgfmathresult}
			\absdot{(\w,\y)}{};
			\ifnum\i>1
				\draw (\w,\y) -- (\x+0.5,\y);
			\else
			\fi
		\fi		
	};
}

\newcommand{\gridsmallhoriz}[1]{
	\begin{tikzpicture}[scale=1, anchor=base]
    \pgfmathsetmacro\gridwidth{dim({#1})}; 
	  \pgftransformxscale{{0.225/2}};
		\pgftransformyscale{{0.225}};
    \foreach \dir [count=\i] in {#1} {
      \ifthenelse{\dir>0}{
		    \draw [semithick, line cap=round] ({\i-1}, 0)--(\i, 1);
      }{
        \draw [semithick, line cap=round] ({\i-1}, 1)--(\i, 0);
      };
    };
  \end{tikzpicture}
}

\newcommand{\gridsmallhorizfn}[1]{
	\begin{tikzpicture}[scale=1, anchor=base]
    \pgfmathsetmacro\gridwidth{dim({#1})}; 
	  \pgftransformxscale{{0.155/\gridwidth}};
		\pgftransformyscale{{0.155}};
    \foreach \dir [count=\i] in {#1} {
      \ifthenelse{\dir>0}{
		    \draw [semithick, line cap=round] ({\i-1}, 0)--(\i, 1);
      }{
        \draw [semithick, line cap=round] ({\i-1}, 1)--(\i, 0);
      };
    };
  \end{tikzpicture}
}

\newcommand{\gridhoriz}[1]{
	\begin{tikzpicture}[scale=1, anchor=base]
	  \pgftransformxscale{0.225};
		\pgftransformyscale{0.225};
    \foreach \dir [count=\i] in {#1} {
      \ifthenelse{\dir>0}{
		    \draw [semithick, line cap=round] ({\i-1}, 0)--(\i, 1);
      }{
        \draw [semithick, line cap=round] ({\i-1}, 1)--(\i, 0);
      };
    };
  \end{tikzpicture}
}

\newcommand{\gridsmallvert}[1]{
	\begin{tikzpicture}[scale=1, anchor=base]
    \pgfmathsetmacro\gridheight{dim({#1})}; 
	  \pgftransformxscale{0.225};
		\pgftransformyscale{0.225/2};
    \foreach \dir [count=\i] in {#1} {
      \ifthenelse{\dir>0}{
		    \draw [semithick, line cap=round] (0, {\i-1})--(1, \i);
      }{
        \draw [semithick, line cap=round] (0, \i)--(1, {\i-1});
      };
    };
  \end{tikzpicture}
}

\newcommand{\gridsmalltwobyvert}[2]{
	\begin{tikzpicture}[scale=1, anchor=base]
    \def\gridwidth{2};
    \pgfmathsetmacro\gridheight{dim({#1})}; 
	  \pgftransformxscale{{0.225/\gridwidth}};
		\pgftransformyscale{{0.225/\gridheight}};
    \foreach \dir [count=\i] in {#1} {
			\ifthenelse{\dir=1 \OR \dir=-1}{
				\ifthenelse{\dir>0}{
					\draw [semithick, line cap=round] (0, {\i-1})--(1, \i);
				}{
					\draw [semithick, line cap=round] (0, \i)--(1, {\i-1});
				};
			}{};
    };
    \foreach \dir [count=\i] in {#2} {
			\ifthenelse{\dir=1 \OR \dir=-1}{
				\ifthenelse{\dir>0}{
					\draw [semithick, line cap=round] (1, {\i-1})--(2, \i);
				}{
					\draw [semithick, line cap=round] (1, \i)--(2, {\i-1});
				};
			}{};
    };
  \end{tikzpicture}
}

\newcommand{\gridverysmallvert}[1]{
	\begin{tikzpicture}[scale=1, anchor=base]
    \def\gridwidth{1};
    \pgfmathsetmacro\gridheight{dim({#1})}; 
	  \pgftransformxscale{{0.175/\gridwidth}};
		\pgftransformyscale{{0.175/\gridheight}};
    \foreach \dir [count=\i] in {#1} {
      \ifthenelse{\dir>0}{
		    \draw [semithick, line cap=round] (0, {\i-1})--(1, \i);
      }{
        \draw [semithick, line cap=round] (0, \i)--(1, {\i-1});
      };
    };
  \end{tikzpicture}
}



%% file: macros.tex
\newcommand{\inv}{\textnormal{inv}}
\newcommand{\perm}{\textnormal{perm}}
\newcommand{\Geom}{\textnormal{Geom}}

\newcommand{\eval}[2][\right]{\relax\ifx#1\right\relax \left.\fi#2#1\rvert}

\newcommand{\Ex}[1]{\mathbb{E}\left[ #1 \right]}

\renewcommand{\Pr}[1]{\mathbb{P}\left[ #1 \right]}

  \definecolor{lightgray}{rgb}{.8,.8,.8}

  \definecolor{light-gray}{gray}{0.65}
  \definecolor{dark-gray}{gray}{0.35}

  \theoremstyle{plain} 
  \newtheorem{theorem}{Theorem}[section]
  \newtheorem{lemma}[theorem]{Lemma}
  \newtheorem{corollary}[theorem]{Corollary}

  \newtheorem{proposition}[theorem]{Proposition}

  \newtheorem{conjecture}[theorem]{Conjecture}
  \newtheorem{problem}[theorem]{Problem}


%% file: dedication.tex
To Mathilde

%% file: acknowledgements.tex
With deep gratitude, I acknowledge the many individuals who have made this dissertation possible. First and foremost, my family's unwavering support gave me the confidence to embark on this academic journey and pursue my intellectual passions. My path in combinatorics began under the guidance of my undergraduate advisor, Rob Davis, who greatly inspired both my interest in the field and self-belief. I am indebted to Zachary Hamaker, whose mentorship and insightful counsel as my faculty mentor proved invaluable. The challenging and enriching courses taught by Miklós Bóna transformed my understanding of combinatorics and shaped me into the mathematician I am today. Above all, I extend my heartfelt thanks to my advisor, Vince Vatter, who generously shared his time, expertise, and the fascinating problems that form the foundation of this work.

%% file: abstract.tex
In this dissertation, we explore the structure of inversion graphs of permutations---a class of graphs that naturally arises by representing each permutation as a graph, where vertices correspond to entries and edges encode inversions. Advantageously, this class also arises from several other definitions that relate these graphs to geometry and partially ordered sets. By leveraging these different perspectives, we are equipped to gain many insights into their combinatorial intricacies.

First among the contributions of this work is a concise proof of a result of Gallai pertaining to modular decomposition and partial orders. Because of their relationship to partial orders, this result has important consequences for inversion graphs and the permutations they represent. After giving the proof, we see how its method arises from a geometric perspective on inversion graphs, and then discuss these consequences.

Next, we consider the notion of geometric grid classes of permutations and its intimate connection, via inversion graphs, to letter graphs from structural graph theory. These geometric grid classes are indeed closely related to a geometric perspective on inversion graphs, and the connection to letter graphs enables us to show that inversion graphs have very well-behaved structure in comparison to almost all graphs as measured by a related parameter called lettericity.

Last, we observe how applying transpositions to a permutation can affect its inversion graph, thus enabling us to generalize transpositions to an operation on arbitrary simple graphs. In studying this operation, we make further connections between the subjects of permutation theory and graph theory.

%% file: chap-intro.tex
\chapter{Introduction}
\label{chap:intro}

Permutations are among the most fundamental objects in mathematics. From pure mathematical theory to applications across sciences, these objects function as essential tools for understanding and describing patterns. Similarly fundamental are graphs, which serve as the mathematical language for representing interconnected structures. In addition to being used to model complex systems across diverse scientific domains, graphs give rise to many intriguing abstract mathematical problems.

This dissertation explores the inversion graphs of permutations---constructions that capture the structure of permutations in the language of graph theory. In this opening chapter, we introduce the necessary concepts from permutation theory and graph theory that enable us to define inversion graphs and list some of their elementary properties. We then observe the connections between inversion graphs, geometry, and partially ordered sets, providing multiple perspectives from which to study these graphs. By leveraging these perspectives, we explore the connections between permutations and graphs in the following chapters, leading to insights in both domains.

\section{Permutations and Inversions}

\subsection{Permutations}

A \emph{permutation}~\cite{bona:combinatorics-o:} is a bijection from a finite linearly ordered set to itself. The labels of the elements upon which a permutation acts are immaterial, so we need only consider permutations of the set $[n] = \{ 1, \dots, n\}$ for some positive integer $n$. The simplest way to represent a permutation $\pi$ of $[n]$ is to write it in \emph{one-line notation} as  $\pi = \pi(1) \pi(2) \dots \pi(n)$. Note that for each $i \in [n]$, we call the pair~$(i,\pi(i))$ an \emph{entry} of $\pi$, and we refer to $i$ as its \emph{index} and $\pi(i)$ as its \emph{value}.

We often refer to the set of permutations of $[n]$, which we denote by $S_n$. There are~$n!$ permutations of $[n]$. For example, there are $3!=6$ permutations in $S_3$: $123$, $213$, $132$, $231$, $312$, and $321$. 

\subsection{Inversions}

Central to the study of permutations is the notion of an \emph{inversion}~\cite[Chp. 2]{bona:combinatorics-o:}. An inversion in a permutation~$\pi$ is any pair $(i,j)$ of indices such that $i<j$ and  $\pi(i) > \pi(j)$. Inversions are easy to spot from the one-line notation of a permutation. That is, one can simply read the permutation from left to right and pick out the pairs of indices in which the larger value is first. For example, the inversions of the permutation $2413$ are $(1,3)$, $(2,3)$ and $(2,4)$.

We let $\inv(\pi)$ denote the number of inversions in the permutation $\pi$. It is clear that $0 \leq \inv( \pi) \leq \binom{n}{2}$ for all permutations $\pi \in S_n$, and that the extremal permutations are the identity and reverse identity. More explicitly, $\inv(e) = 0$ and $\inv(e^{\textnormal{r}}) = \binom{n}{2}$, where~$e = 12 \dots n$ and~$e^{\textnormal{r}} = n \dots 21$.

We see next that we can uniquely represent a permutation in terms of its inversions. Given a permutation $\pi$ of $[n]$, we define its \emph{Lehmer code}~\cite{laisant:sur-la-numerati:, lehmer:teaching-combin:}, or simply \emph{code}, to be the~$n$-tuple $\mathbf{c}(\pi) =(c_1, c_2, \dots, c_n)$ such that $c_i$ denotes the number of indices $j$ for which~$(i,j)$ is an inversion of~$\pi$. In other words, the number of entries to the right of $\pi(i)$ in $\pi(1) \pi(2) \dots \pi(n)$ that are smaller than $\pi(i)$. For example, we have 
\[
\mathbf{c} (37168254) = (2,5,0,3,3,0,1,0).
\]

It is not difficult to see that given a permutation $\pi$ of $[n]$ that $\mathbf{c}(\pi) \in \mathfrak{C}_n$ where
\[
\mathfrak{C}_n = [0,n-1] \times [0,n-2] \times \dots \times [0,1] \times [0,0]. 
\]
As hinted at by the fact that $|\mathfrak{C}_n|$ is $n!$, the following confirms that codes uniquely identify their permutations.

\begin{proposition}
	For all integers $n \geq 1$, the map $\mathbf{c}: S_n \to \mathfrak{C}_n$ is a bijection.
\end{proposition}

\begin{proof}
	To prove the result we find the inverse of this map. That is, let $(c_1, c_2, \dots, c_n) \in \mathfrak{C}_n$ and we construct the permutation $\tau$ that has this tuple as its code. Given that every other entry is to the right of the first entry in one-line notation, we must define $\tau(1) = c_1 + 1$. Next, we set $\tau(2)$ to be the element of $[n] \setminus \{ c_1 + 1\}$ such that it contains exactly $c_2$ elements smaller than it. We continuing defining $\tau$ from left to right in this way, setting $\tau(i)$ to be the element of $[n] \setminus \{\tau(1), \tau(2), \dots, \tau(i-1) \}$ such that it contains exactly $c_i$ elements smaller than it. This gives the result.
\end{proof}

Given this bijection, it is simple to find the generating function that enumerates permutations of $[n]$ by their number of inversions.

\begin{theorem}[Rodrigues~\cite{rodrigues:note-sur-les-in:}]
	For all integers $n \geq 2$, we have
	\[
	\sum_{\pi \in S_n} q^{\inv(\pi)} = (1+q) (1+q+q^2) \dots (1+q+ \dots +q^{n-1}).
	\]
\end{theorem}

\begin{proof}
	For any permutation $\pi \in S_n$, given its code $\mathbf{c} (\pi) = (c_1, c_2, \dots, c_n)$ it is clear that $\inv(\pi) = c_1 + c_2 + \dots + c_n$. Therefore, we have
	\begin{align*}
		\sum_{\pi \in S_n} q^{\inv(\pi)} & = \sum_{c_1 = 0}^{n-1} \sum_{c_2 = 0}^{n-2} \dots \sum_{c_n = 0}^{0} q^{c_1 + c_2 + \dots + c_n} \\[2ex]
		& = \left( \sum_{c_1 = 0}^{n-1} q^{c_1} \right)\left( \sum_{c_2 = 0}^{n-2} q^{c_2} \right) \dots \left( \sum_{c_n = 0}^{0} q^{c_n} \right).
	\end{align*}
	This gives the result.
\end{proof}

Denoting this polynomial by $I_n(q)$, we let $i(n,k)$ denote the coefficient of $q^k$ in $I_n(q)$, i.e., the number of permutations of $[n]$ with $k$ inversions. It follows from this factorization of~$I_n(q)$ that the sequence $i(n,0), i(n,1), \dots, i(n, \binom{n}{2})$ has a very desirable property. We say that a sequence $a_1, a_2, \dots ,a_n$ of positive real numbers is \emph{log-concave}~\cite{stanley:log-concave-and:} if $a_{k-1} a_{k+1} \leq a_k^2$ for all indices $k \in \{2, \dots, n-1 \}$. Further, we say that a polynomial is log-concave if its coefficients, (ordered by degree), form a log-concave sequence. It is a standard exercise to show that the product of two log-concave polynomials is also log-concave, and thus we have the following.

\begin{corollary}
	For all integers $n \geq 3$, the sequence $i(n,0), i(n,1), \dots, i(n, \binom{n}{2})$ of coefficients of~$I_n(q)$ is log-concave.
\end{corollary}



We have now seen how to uniquely capture the structure of a permutation in terms of its inversions with its Lehmer code. In the rest of this dissertation, we study inversion graphs, constructions that also capture the structure of permutations in terms inversions.

\section{Inversion Graphs}

\subsection{Graphs and a First Definition}

To capture the structure of a permutation and its inversions, we can use a graph~\cite{chartrand:graphs--digraph:7}. For completeness, a \emph{graph} $G$ is an ordered pair $(V,E)$ made up of
\begin{itemize}
	\item a set $V$, whose elements are called \emph{vertices}, \emph{nodes} or \emph{points}, and
	\item a set $E \subseteq \{ \{ u,v\} : u \neq v \}$ of pairs of vertices, whose elements are called \emph{edges}.
\end{itemize}
We mention that an edge $\{ u,v \}$ is often denoted more concisely as $uv$ or $vu$, and that when not explicitly defined, the vertex and edge sets of a graph $G$ are denoted by $V(G)$ and $E(G)$, respectively. For a vertex $v$ of a graph $G$, we let $N_G(v) = \{ v : u \sim v \} \subseteq V(G)$ denote the \emph{neighborhood} of $v$, and similarly, $N_G[v] = N_G(v) \cup \{ v\}$ the \emph{closed neighborhood} of $v$. We may omit the subscript from these notations if it is clear which graph we are referring to. 

The \emph{complement} of a graph $G$ is the graph $\overline{G}$ on the same set of vertices and with edge set $E(\overline{G}) = \{ uv : u \neq v \text{ and } uv \notin E(G) \}$. A \emph{subgraph} of a graph $G$ is any graph $H$ such that $V(H) \subseteq V(G)$ and $E(H) \subseteq E(G)$. Furthermore, a subgraph $H$ of a graph $G$ is an \emph{induced subgraph} if $E(H) = V(H)^2 \cap E(G)$. That is, two vertices are adjacent in $H$ if and only if they are adjacent in $G$. Thus, an induced subgraph~$H$ of a graph $G$ depends only on the subset $S$ of vertices such that $V(H) = S$, and thus we may denote $H$ by $G[S]$. If $S = \{ v_1, \dots, v_k\}$, we may also slightly abuse notation and denote $H$ by $G[v_1, \dots, v_k]$. 

With most of our graphical notation now outlined, we direct the reader to~\cite{chartrand:graphs--digraph:7} (or any other standard text on graph theory) for the many further common definitions used throughout the subject. We make use of the following conventional terms but omit their definitions: degree, walk, path, connected, component, tree, forest, cycle, complete, bipartite, clique, independent set (i.e., anticlique), disjoint union, join, diameter, girth, isomorphism, and automorphism.

The objects of study in this dissertation are inversion graphs of permutations, previously referred to as permutation graphs~\cite[Chp. 7]{golumbic:algorithmic-gra:} in the literature. Precisely, the \emph{inversion graph} $G_{\pi}$ of the permutation $\pi$ is the graph whose vertices are the entries of $\pi$, with an edge between two entries if and only if they form an inversion. It is often convenient to label the vertices by only the index or the value of the corresponding entry in the permutation. Throughout our discussions, it will be most useful to label vertices by value, and thus, if~$\pi \in S_n$ we write $V(G_{\pi}) = [n]$ and $E(G_{\pi}) = \{ \pi(i) \pi(j) : i<j \text{ and } \pi(i)>\pi(j) \}$.

We list here some elementary facts about inversion graphs.

\begin{proposition}\label{prop:isomorphic_to_inverse}
	For all permutations $\pi$, we have $G_{\pi} \cong G_{\pi^{-1}}$.
\end{proposition}
	
That is, we can associate the entry $(i, \pi(i))$ in $\pi$ with the entry $(\pi(i), i)$ in $\pi^{-1}$, and we see that the entries $(i, \pi(i))$ and $(j, \pi(j))$ form an inversion in $\pi$ if and only if $(\pi(i),i)$ and $(\pi(j), j)$ form an inversion in $\pi^{-1}$. We will further discuss the problem of determining which permutations have isomorphic inversion graphs in Chapter~\ref{chap:uniqueness}.

Letting $\pi^{\textnormal{r}}$ denote the reverse of the permutation $\pi$, we see next that the complement of an inversion graph is also an inversion graph.

\begin{proposition}
	For all permutations $\pi$, we have $\overline{G}_{\pi} \cong G_{\pi^{\textnormal{r}}}$.
\end{proposition}
	
That is, if $i<j$, then the value $\pi(i)$ appears to the left of~$\pi(j)$ in $\pi$, whereas the value $\pi(i)$ is to the right of~$\pi(j)$ in $\pi^{\textnormal{r}}$ since they appear at the indices $n+1-i$ and $n+1-j$, respectively. Thus, each pair of values forms an inversion in exactly one of $\pi$ and $\pi^{\textnormal{r}}$, and therefore $\overline{G}_{\pi} = G_{\pi^{\textnormal{r}}}$ in keeping with our convention of labeling each vertex by the value of its corresponding entry.

In the rest of this section, we observe several definitions that also give rise to the class of inversion graphs. Given these definitions, we see that inversion graphs are related to both geometry and partially ordered sets, giving us diverse perspectives from which to study these graphs.

\subsection{Geometric Definitions}

For our first geometric definition giving rise to the class of inversion graphs, we see that one can obtain the inversion graph of a permutation $\pi$ by plotting the vertex of the entry~$(i, \pi(i))$ at the point $(i, \pi(i))$ in the plane for each $i \in [n]$, and then connecting each pair of vertices such that the unique line through their points has negative slope. This geometric perspective is the most useful for our study of inversion graphs, and we often identify permutations with their plots in the plane. We see an example of an inversion graph plotted in this way in Figure~\ref{fig:sm_inversion_graph}. Note that we omit grid lines and axes in the figures that follow. 

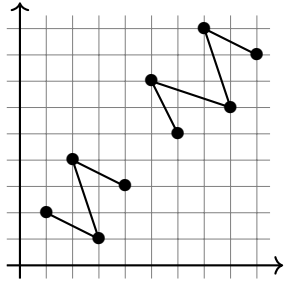
\begin{figure}[h]
\begin{center}
	\begin{tikzpicture}[scale=0.35]
	
		\draw[step=1cm,gray,very thin] (-0.5,-0.5) grid (9.5,9.5);
		\draw[->] (-0.5,0) to (10,0);
		\draw[->] (0,-0.5) to (0,10);

		\plotpermgraph{2,4,1,3,7,5,9,6,8};
				
	\end{tikzpicture}
\end{center}
\caption{The inversion graph of $241375968$ plotted in the plane.}
\label{fig:sm_inversion_graph}
\end{figure}

This perspective enables us to give an equivalent geometric definition for the class of inversion graphs with the following result. It isn't difficult to understand both directions from the above discussion.

\begin{theorem}\label{thm:plotting_def}
	A graph is the inversion graph of a permutation if and only if it can be obtained by plotting a finite set of points in the plane such that no two points share an $x$- or $y$-coordinate, and connecting any two points that determine a line with negative slope.
\end{theorem}

Furthermore, this plotting perspective enables us to observe another simple proof of Proposition~\ref{prop:isomorphic_to_inverse}. That is, one can obtain the plot of the permutation $\pi^{-1}$ by reflecting the plot of~$\pi$ over the line $y = x$. Noting that this reflection preserves the sign of the slope of any line with nonzero, finite slope, the result follows. 

This geometric perspective on inversion graphs also gives us an opportunity to intuitively define the well-known sum operations on permutations. The \emph{direct sum} of the permutations $\sigma$ of $[n]$ and $\tau$ of $[m]$ is the permutation~$\sigma \oplus \tau$  given by
\[
(\sigma \oplus \tau)(i) = \begin{cases}
	\sigma(i) & \text{for } 1 \leq i \leq n, \\
	\tau(i-n)+n & \text{for } n+1 \leq i \leq m+n.
\end{cases}
\]
Similarly, their \emph{skew sum} is the permutation $\sigma \ominus \tau$ given by
\[
(\sigma \ominus \tau)(i) = \begin{cases}
	\sigma(i) + m & \text{for } 1 \leq i \leq n, \\
	\tau(i - n) & \text{for } n+1 \leq i \leq m+n.
\end{cases}
\]

More intuitively, we see in Figure~\ref{fig:direct_skew_sums} that the plot of $\sigma \oplus \tau$ is obtained by appending the plot of $\tau$ above and to the right of the plot of $\sigma$, and $\sigma \ominus \tau$ is obtained by appending $\tau$ below and to the right of~$\sigma$.

\begin{figure}[h]
\begin{center}
	\begin{tikzpicture}[scale=0.2]
		\draw[line width = 1pt, color = gray] (0.5,0.5) rectangle (5.5,5.5);
		
		\plotperm{2,4,1,5,3};
		
		\node [] at (3,-1) {$\sigma$};
		
	\end{tikzpicture}
	\hspace{10mm}
	\begin{tikzpicture}[scale=0.2]
		\draw[line width = 1pt, color = gray] (0.5,0.5) rectangle (4.5,4.5);
		
		\plotperm{2,3,4,1};
		
		\node [] at (2.5,-1) {$\tau$};
		
	\end{tikzpicture}
	\hspace{10mm}
	\begin{tikzpicture}[scale=0.2]
		\draw[line width = 1pt, color = gray] (0.5,0.5) rectangle (5.5,5.5);
		\draw[line width = 1pt, color = gray] (0.5+5,0.5+5) rectangle (4.5+5,4.5+5);

		\plotperm{2,4,1,5,3,2+5,3+5,4+5,1+5};
		
		\node [] at (5,-1) {$\sigma \oplus \tau $};
		
	\end{tikzpicture}
	\hspace{10mm}
	\begin{tikzpicture}[scale=0.2]
		\draw[line width = 1pt, color = gray] (0.5,0.5+4) rectangle (5.5,5.5+4);
		\draw[line width = 1pt, color = gray] (0.5+5,0.5) rectangle (4.5+5,4.5);

		\plotperm{2+4,4+4,1+4,5+4,3+4,2,3,4,1};
		
		\node [] at (5,-1) {$\sigma \ominus \tau $};
		
	\end{tikzpicture}
\end{center}
\caption{The direct sum and skew sum operations on permutations.}
\label{fig:direct_skew_sums}
\end{figure}
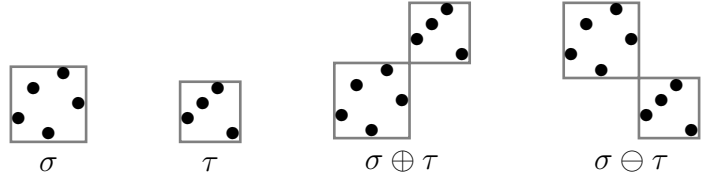

It follows easily from this figure that, given two permutations $\sigma$ and $\tau$, we have 
\[
G_{\sigma \oplus \tau} \cong G_{\sigma} \cup G_{\tau} \textnormal{ and } G_{\sigma \ominus \tau} \cong G_{\sigma} \vee G_{\tau},
\]
with $\cup$ and $\vee$ denoting the well-known graph operations of disjoint union and join, respectively.

Perhaps the most common definition for the class of inversion graphs in the literature is given as follows.

\begin{theorem}[Even, Pneuli, and Lempel~\cite{pnueli:transitive-orie:, even:permutation-gra:}]\label{thm:parallel_lines}
	A graph is the inversion graph of a permutation if and only if it is the intersection graph of line segments whose endpoints lie on two parallel lines.
\end{theorem}

We see an example of such a construction and the resulting intersection graph in Figure~\ref{fig:intersection_graph}. On the left, we have five line segments with their endpoints lying on two horizontal lines, each corresponding to a vertex in the resulting graph on the right. The vertices in the resulting graph are connected if and only if their corresponding line segments cross at any point.

\def\nodesize{4}
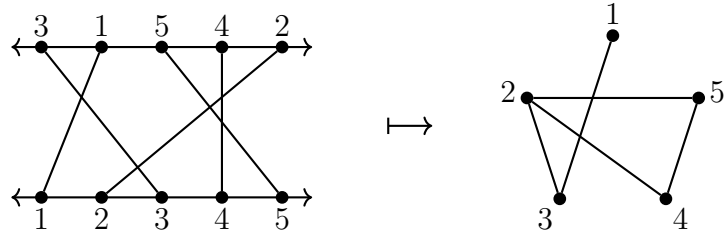
\begin{figure}[h]
\begin{center}
	\begin{tikzpicture}[scale=2]
		\draw [<->, line cap=round] (0,0)--(2,0);
		\draw [<->, line cap=round] (0,1)--(2,1);
		
		\node [circle, fill, minimum size = \nodesize pt, draw] (1) at (0.2,0) {};
		\node [circle, fill, minimum size = \nodesize pt, draw] (2) at (0.6,0) {};
		\node [circle, fill, minimum size = \nodesize pt, draw] (3) at (1,0) {};
		\node [circle, fill, minimum size = \nodesize pt, draw] (4) at (1.4,0) {};
		\node [circle, fill, minimum size = \nodesize pt, draw] (5) at (1.8,0) {};
		
		\labelnode[-90]{1}{$1$}
		\labelnode[-90]{2}{$2$}
		\labelnode[-90]{3}{$3$}
		\labelnode[-90]{4}{$4$}
		\labelnode[-90]{5}{$5$}
		
		\node [circle, fill, minimum size = \nodesize pt, draw] (p3) at (0.2,1) {};
		\node [circle, fill, minimum size = \nodesize pt, draw] (p1) at (0.6,1) {};
		\node [circle, fill, minimum size = \nodesize pt, draw] (p5) at (1,1) {};
		\node [circle, fill, minimum size = \nodesize pt, draw] (p4) at (1.4,1) {};
		\node [circle, fill, minimum size = \nodesize pt, draw] (p2) at (1.8,1) {};
		
		\labelnode[90]{p1}{$1$}
		\labelnode[90]{p2}{$2$}
		\labelnode[90]{p3}{$3$}
		\labelnode[90]{p4}{$4$}
		\labelnode[90]{p5}{$5$}
		
		\draw [] (1) to (p1);
		\draw [] (2) to (p2);
		\draw [] (3) to (p3);
		\draw [] (4) to (p4);
		\draw [] (5) to (p5);
		
	\end{tikzpicture}
	\hspace{8mm}
	\begin{tikzpicture}[scale=2]
	
	\draw[|->] (-1.5, 0) to (-1.2,0);
	
	\makenodescircle{v}{5}{0.6}[90][circle, fill, minimum size = \nodesize pt, draw]
	\labelnode[36/2+36*6]{v3}{$3$}
	\labelnode[36/2+36*4]{v2}{$2$}
	\labelnode[36/2]{v5}{$5$}
	\labelnode[90]{v1}{$1$}
	\labelnode[36/2+36*8]{v4}{$4$}
	
	\drawedges{v1/v3, v2/v3, v2/v4, v2/v5, v4/v5};
	\end{tikzpicture}
\end{center}
\caption{Five line segments with their endpoints lying on two parallel lines and the resulting intersection graph.}
\label{fig:intersection_graph}
\end{figure}

We see in this figure that we labeled the line segments by the order in which their endpoints land on the lower horizontal line. Then, reading the labels of the line segments as they appear on the upper horizontal line in the same direction, we see that the labels have been permuted. Indeed, we read the permutation $31542$ on the upper line, and it is clear that two lines intersect if and only if their labels are values that form an in version in this permutation. That is, the resulting graph is isomorphic to $G_{31542}$.

The other direction is also simple. That is, given a permutation $\pi$ of $[n]$, we can simply choose~$n$ points on a line and label them going in one direction with $1$, $2$, \dots, $n$, and then on a distinct parallel pick $n$ other points and label them going in the same direction with $\pi(1)$, $\pi(2)$, \dots, $\pi(n)$. It follows that by drawing a line segment between each pair of points with the same label, the resulting intersection graph is isomorphic to $G_{\pi}$.

The next geometric definition is essentially the same as this last one, but we include it as it makes clear that inversion graphs are contained in the class of circle graphs.

\begin{theorem}[Golumbic~\cite{golumbic:algorithmic-gra:}]\label{thm:circ_graph_def}
	A graph is the inversion graph of a permutation if and only if it is a circle graph that admits an equator.
\end{theorem}

A circle graph is also the intersection graph of line segments from a geometric drawing. In this case, the line segments are chords of a circle, and the stipulation that it must admit an equator means that it must be possible to draw another chord that intersects the rest of the chords. Such a drawing can be seen on the left in Figure~\ref{fig:circle_graph_with_equator}, with the dashed gray line playing the role of such an equator.

\begin{figure}[h]
\begin{center}
	\begin{tikzpicture}[scale=1.2]
	
	\draw (0,0) circle (1);
	\draw [dashed, gray] (-1.1, 0) to (1.1,0);
	\makenodescircle{v}{10}{1}[90][circle, fill, minimum size = \nodesize pt, draw]
	
	\labelnode[36/2+36*4]{v3}{$3$}
	\labelnode[36/2+36*3]{v2}{$1$}
	\labelnode[36/2+36*2]{v1}{$5$}
	\labelnode[36/2+36]{v10}{$4$}
	\labelnode[36/2]{v9}{$2$}
	
	\labelnode[36/2+36*5]{v4}{$1$}
	\labelnode[36/2+36*6]{v5}{$2$}
	\labelnode[36/2+36*7]{v6}{$3$}
	\labelnode[36/2+36*8]{v7}{$4$}
	\labelnode[36/2+36*9]{v8}{$5$}
	
	\drawedges{v3/v6, v2/v4, v1/v8, v10/v7, v9/v5};

	\end{tikzpicture}
	\hspace{6mm}
	\begin{tikzpicture}[scale=2]
		\draw[<->] (-0.8, 0.5) to (-0.4, 0.5);
	
		\draw [<->, line cap=round] (0,0)--(2,0);
		\draw [<->, line cap=round] (0,1)--(2,1);
		
		\node [circle, fill, minimum size = \nodesize pt, draw] (1) at (0.2,0) {};
		\node [circle, fill, minimum size = \nodesize pt, draw] (2) at (0.6,0) {};
		\node [circle, fill, minimum size = \nodesize pt, draw] (3) at (1,0) {};
		\node [circle, fill, minimum size = \nodesize pt, draw] (4) at (1.4,0) {};
		\node [circle, fill, minimum size = \nodesize pt, draw] (5) at (1.8,0) {};
		
		\labelnode[-90]{1}{$1$}
		\labelnode[-90]{2}{$2$}
		\labelnode[-90]{3}{$3$}
		\labelnode[-90]{4}{$4$}
		\labelnode[-90]{5}{$5$}
		
		\node [circle, fill, minimum size = \nodesize pt, draw] (p3) at (0.2,1) {};
		\node [circle, fill, minimum size = \nodesize pt, draw] (p1) at (0.6,1) {};
		\node [circle, fill, minimum size = \nodesize pt, draw] (p5) at (1,1) {};
		\node [circle, fill, minimum size = \nodesize pt, draw] (p4) at (1.4,1) {};
		\node [circle, fill, minimum size = \nodesize pt, draw] (p2) at (1.8,1) {};
		
		\labelnode[90]{p1}{$1$}
		\labelnode[90]{p2}{$2$}
		\labelnode[90]{p3}{$3$}
		\labelnode[90]{p4}{$4$}
		\labelnode[90]{p5}{$5$}
		
		\draw [] (1) to (p1);
		\draw [] (2) to (p2);
		\draw [] (3) to (p3);
		\draw [] (4) to (p4);
		\draw [] (5) to (p5);
		
	\end{tikzpicture}
\end{center}
\caption{Constructions from Theorems~\ref{thm:parallel_lines} and~\ref{thm:circ_graph_def} with the same intersection graph.}
\label{fig:circle_graph_with_equator}
\end{figure}
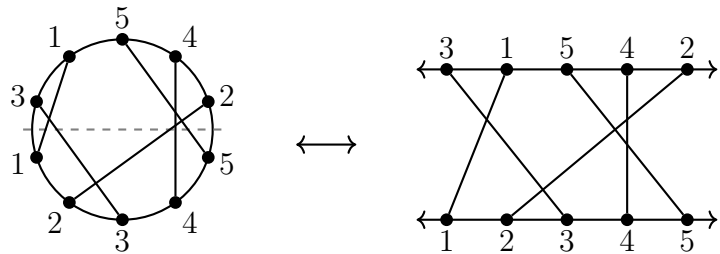

On the right of Figure~\ref{fig:circle_graph_with_equator}, we include the drawing from Figure~\ref{fig:intersection_graph} that also induces the inversion graph of the permutation~$31542$ as its intersection graph. Indeed, it should be clear how to transform both drawings in this figure into the other without changing the intersections, and further, how to go between the two types of drawings in general.

We mention that not all circle graphs are inversion graphs. For example, we see in Figure~\ref{fig:circle_graph_cycle} that the cycle $C_5$ is a circle graph, but the given circle drawing does not admit an equator. We see in Chapter~\ref{chap:uniqueness} that $C_5$ is not the inversion graph of any permutation, and thus there is no such circle drawing inducing $C_5$ that admits an equator.

\begin{figure}[h]
\begin{center}
	\begin{tikzpicture}[scale=1.2]
	
	\draw (0,0) circle (1);
	\makenodescircle{v}{10}{1}[90][circle, fill, minimum size = \nodesize pt, draw]
	
	\labelnode[36/2+36*2]{v1}{$1$}
	\labelnode[36/2+36*5]{v4}{$1$}
	
	\labelnode[36/2+36*4]{v3}{$2$}	
	\labelnode[36/2+36*7]{v6}{$2$}

	\labelnode[36/2+36*6]{v5}{$3$}
	\labelnode[36/2+36*9]{v8}{$3$}
	
	\labelnode[36/2+36*8]{v7}{$4$}
	\labelnode[36/2+36]{v10}{$4$}
	
	\labelnode[36/2+36*3]{v2}{$5$}
	\labelnode[36/2]{v9}{$5$}
	
	\drawedges{v1/v4, v3/v6, v5/v8, v7/v10, v9/v2};

	\end{tikzpicture}
	\hspace{6mm}
	\begin{tikzpicture}[scale=1.2]
	
	\draw[|->] (-2.3, 0) to (-1.8,0);
	
	\makenodescircle{v}{5}{1}[90][circle, fill, minimum size = \nodesize pt, draw]
	\labelnode[36/2+36*6]{v3}{$3$}
	\labelnode[36/2+36*4]{v2}{$2$}
	\labelnode[36/2]{v5}{$5$}
	\labelnode[90]{v1}{$1$}
	\labelnode[36/2+36*8]{v4}{$4$}
	
	\drawedges{v1/v2, v2/v3, v3/v4, v4/v5, v1/v5};
	\end{tikzpicture}
\end{center}
\caption{Chords drawn on a circle such that the induced intersection graph is the cycle~$C_5$.}
\label{fig:circle_graph_cycle}
\end{figure}
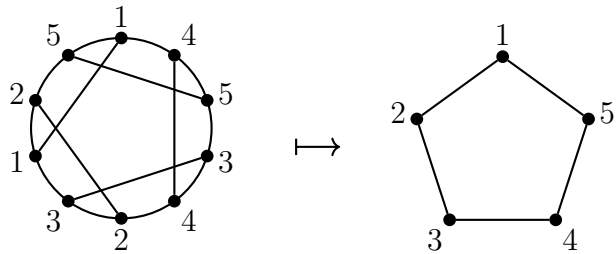 

A final geometric definition is as follows.

\begin{theorem}[Dushnik and Miller~\cite{dushnik:partially-order:}]
	A graph is the inversion graph of a permutation if and only if it is a containment graph of intervals.
\end{theorem}


More specifically, by containment graph of intervals we mean a graph that is realized by a set of intervals $\{ [l_i, r_i] \}_{i=1}^n$ of the real line such that each interval corresponds to a vertex and two vertices are adjacent if and only if one of their corresponding intervals is contained in the other. It is easy to find a set of intervals that realize the inversion graph of a given permutation $\pi$. If $\pi$ is a permutation of $[n]$, we simply define an interval $[l_i, r_i]$ with label~$i$ for each $i \in [n]$ such that 
\[
l_1 < l_2 < \dots < l_n < r_{\pi(1)} < r_{\pi(2)} < \dots < r_{\pi(n)}. 
\]
Such a construction for the permutation $\pi = 31542$ can be seen in Figure~\ref{fig:containment_of_intervals_graph}, and makes clear that choosing intervals in this way works for all permutations.

\def\vstretch{0.5}
\begin{figure}[h]
\begin{center}
	\begin{tikzpicture}[scale=0.8]
		
	\node [circle, fill, minimum size = \nodesize pt, draw] (l1) at (1,0) {};	
	\node [circle, fill, minimum size = \nodesize pt, draw] (r1) at (7,0) {};
	
	\node [circle, fill, minimum size = \nodesize pt, draw] (l2) at (2,\vstretch) {};	
	\node [circle, fill, minimum size = \nodesize pt, draw] (r2) at (10,\vstretch) {};
	
	\node [circle, fill, minimum size = \nodesize pt, draw] (l3) at (3,\vstretch*2) {};	
	\node [circle, fill, minimum size = \nodesize pt, draw] (r3) at (6,\vstretch*2) {};
	
	\node [circle, fill, minimum size = \nodesize pt, draw] (l4) at (4,\vstretch*3) {};	
	\node [circle, fill, minimum size = \nodesize pt, draw] (r4) at (9,\vstretch*3) {};
	
	\node [circle, fill, minimum size = \nodesize pt, draw] (l5) at (5,\vstretch*4) {};	
	\node [circle, fill, minimum size = \nodesize pt, draw] (r5) at (8,\vstretch*4) {};
	
	\labelnode[180]{l1}{$1$}
	\labelnode[0]{r1}{$1$}
	
	\labelnode[180]{l2}{$2$}
	\labelnode[0]{r2}{$2$}
	
	\labelnode[180]{l3}{$3$}
	\labelnode[0]{r3}{$3$}
	
	\labelnode[180]{l4}{$4$}
	\labelnode[0]{r4}{$4$}
	
	\labelnode[180]{l5}{$5$}
	\labelnode[0]{r5}{$5$}

	\drawedges{r1/l1, r2/l2, r3/l3, r4/l4, r5/l5};

	\end{tikzpicture}
\end{center}
\caption{A set of intervals that realizes $G_{31542}$ as its the containment graph.}
\label{fig:containment_of_intervals_graph}
\end{figure}
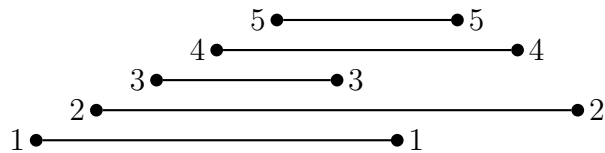 

Of course, many collections of intervals don't look like this, and thus showing that the containment graph of any interval system realizes an inversion graph requires slight attention. So, let~$\{[l_i, r_i] \}_{i=1}^{n}$ be any set of intervals of the real line, and we will further assume that all the left-endpoints and right-endpoints are unique in order to simplify our discussion. We can therefore assume without loss of generality that $l_1 < l_2 < \dots < l_n$, and also, that there exists a permutation $\tau$ of $[n]$ such that $r_{\tau(1)} < r_{\tau(2)} < \dots < r_{\tau(n)}$. We argue that the resulting containment graph is isomorphic to $G_{\tau}$. So, let $i < j$ and suppose $\tau^{-1}(j)< \tau^{-1}(i)$, i.e., $(\tau^{-1}(j), \tau^{-1}(i))$ is an inversion of $\tau$. It follows that $l_i < l_j < r_j < r_i$, or equivalently, $[l_j, r_j] \subseteq [l_i, r_i]$. Furthermore, if $\tau^{-1}(i) < \tau^{-1}(j)$, then the observations that $l_i < l_j$ and $r_i < r_j$ show that the intervals $[l_i, r_i]$ and $[l_j, r_j]$ are incomparable. This shows that the resulting graph is indeed $G_{\tau}$.

In summary, we label a collection of intervals with positive integers $1, 2, \dots$ by the order in which their left-endpoints appear on the real line from left to right, and then collect the permutation whose inversion graph is realized by the system by reading the labels of the intervals as their right-endpoints appear from left to right. For example, in Figure~\ref{fig:intervals_to_permutation} we read the permutation $14352$ from the labels that are in bold on the right-endpoint of each interval.

\def\vstretch{0.5}
\begin{figure}[h]
\begin{center}
	\begin{tikzpicture}[scale=0.8]
		
	\node [circle, fill, minimum size = \nodesize pt, draw] (l1) at (3,0) {};	
	\node [circle, fill, minimum size = \nodesize pt, draw] (r1) at (8,0) {};
	
	\node [circle, fill, minimum size = \nodesize pt, draw] (l2) at (1,\vstretch) {};	
	\node [circle, fill, minimum size = \nodesize pt, draw] (r2) at (4,\vstretch) {};
	
	\node [circle, fill, minimum size = \nodesize pt, draw] (l3) at (5,\vstretch*2) {};	
	\node [circle, fill, minimum size = \nodesize pt, draw] (r3) at (6,\vstretch*2) {};
	
	\node [circle, fill, minimum size = \nodesize pt, draw] (l4) at (2,\vstretch*3) {};	
	\node [circle, fill, minimum size = \nodesize pt, draw] (r4) at (10,\vstretch*3) {};
	
	\node [circle, fill, minimum size = \nodesize pt, draw] (l5) at (7,\vstretch*4) {};	
	\node [circle, fill, minimum size = \nodesize pt, draw] (r5) at (9,\vstretch*4) {};
	
	\labelnode[180]{l1}{$3$}
	\labelnode[0]{r1}{$\mathbf{3}$}
	
	\labelnode[180]{l2}{$1$}
	\labelnode[0]{r2}{$\mathbf{1}$}
	
	\labelnode[180]{l3}{$4$}
	\labelnode[0]{r3}{$\mathbf{4}$}
	
	\labelnode[180]{l4}{$2$}
	\labelnode[0]{r4}{$\mathbf{2}$}
	
	\labelnode[180]{l5}{$5$}
	\labelnode[0]{r5}{$\mathbf{5}$}

	\drawedges{r1/l1, r2/l2, r3/l3, r4/l4, r5/l5};

	\end{tikzpicture}
\end{center}
\caption{Identifying a permutation whose inversion graph is the containment graph of a collection of intervals.}
\label{fig:intervals_to_permutation}
\end{figure}
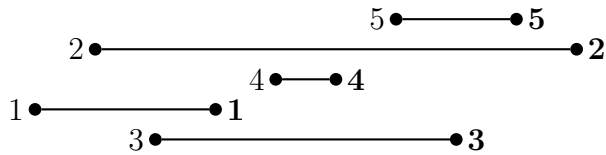 

\subsection{Partial Order Definitions}

As we have just seen, the class of inversion graphs arises from many geometric definitions. However, much of our understanding of this class comes from its relationship to the theory of partially ordered sets. To understand this connection, we need several definitions.

A \emph{partially ordered set}~\cite[Chp. 3]{stanley:enumerative-com:1}, or \emph{poset} for short, is a pair $(P, \leq)$ consisting of a set~$P$ and a binary relation $\leq$ on $P$ that is
\begin{enumerate}
	\item \emph{reflexive}, ($t \leq t$ for all $t \in P$),
	\item \emph{antisymmetric}, (if $s \leq t$ and $t \leq s$, then $s=t$), and
	\item \emph{transitive}, (if $s \leq t$ and $t \leq u$, then $s \leq u$). 
\end{enumerate}
We refer to a partially ordered set $(P, \leq )$ simply as $P$ when there is no confusion as to what the underlying relation can be. We write $s < t$ for the situation in which $s \leq t$ and $s \neq t$. Further, we say that two elements $s, t \in P$ are \emph{comparable} if $s \leq t$ or $t \leq s$. Otherwise, we say that $s$ and $t$ are \emph{incomparable}. A partial order is called a \emph{total order} if all of its elements are comparable.

A graph $G$ is called \emph{transitively orientable} if its edges can be directed in such a way that whenever~$a \rightarrow b$ and $b \rightarrow c$ are arcs (i.e., directed edges), then $a \rightarrow c$ is also an arc. These graphs are also called \emph{comparability graphs} as they correspond with the graphs obtained by taking the elements of a partially ordered set to be vertices, and connecting the vertices of two unique elements if and only if they are comparable.

The \emph{order dimension}~\cite{trotter:combinatorics-a:}, or simply \emph{dimension}, of a partially ordered set is the least number of total orders whose intersection yields the partial order. (Every partial order does arise from the intersection of some collection of total orders, and therefore dimension is indeed defined for all posets.) For example, a partially ordered set has order dimension 1 if and only if it is a total order. Further, the poset given in Figure~\ref{fig:2dimensional_poset} has order dimension 2 since it is not a total order and is given by 
\[
\{ a<b<c<d<e\} \cap \{ b<d<a<e<c\}.
\]
We note that the drawing of this poset is its \emph{Hasse diagram}. This graphical representation of a poset represents each element as a vertex, and draws the vertex $y$ above $x$ and connects them with an edge if $x < y$ is a \emph{cover relation}: there are no elements $t$ such that $x < t < y$. The rest of the relation can be extrapolated from the diagram by noting that it is the transitive closure of the cover relations.

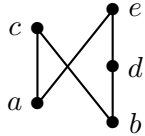
\begin{figure}[h]
\begin{footnotesize}
\begin{center}
	\begin{tikzpicture}[scale=1]
		\draw [] (0,0) -- (0,1) -- (1,-0.25) -- (1,0.5) -- (1,1.25) -- (0,0);
		
		\absdot{(0,0)};
		\absdot{(0,1)};
		\absdot{(1,-0.25)};
		\absdot{(1,0.5)};
		\absdot{(1, 1.25)}
		
		\node [] at (-0.3, 0) {\small$a$};
		\node [] at (-0.3, 1) {\small$c$};
		\node [] at (1.3, -0.25) {\small$b$};
		\node [] at (1.3, 0.5) {\small$d$};
		\node [] at (1.3, 1.25) {\small$e$};
		
	\end{tikzpicture}
\end{center}
\end{footnotesize}
\caption{The Hasse diagram of a poset with order dimension 2.}
\label{fig:2dimensional_poset}
\end{figure}

This next definition shows that inversion graphs are well-behaved comparability graphs.

\begin{theorem}[Baker, Fishburn, and Roberts~\cite{baker:partial-orders-:}]\label{thm:dimension_def}
	A graph is the inversion graph of a permutation if and only if it is the comparability graph of a partially ordered set of order dimension at most 2.
\end{theorem}

For a permutation $\pi$ of $[n]$, we consider the partial order given by
\[
\{ 1<2<\dots<n\} \cap \{ \pi(n) < \dots < \pi(2) < \pi(1) \}.
\]
Indeed, we have $i<j$ in this poset if and only if $(\pi^{-1}(j), \pi^{-1}(i))$ is an inversion of $\pi$. Perhaps more intuitively, this poset can be obtained by plotting the inversion graph of $\pi$, just as we did for the permutation $241375968$ in Figure~\ref{fig:sm_inversion_graph}, directing each edge so that it is an arc oriented towards the northwest, and then setting $x < y$ if and only if $x \rightarrow y$ is an arc. For example, in Figure~\ref{fig:PosetFromPerm} we direct all of the edges of the plotted inversion graph of $35142$ towards the northwest and obtain the poset from Figure~\ref{fig:2dimensional_poset}.

\begin{figure}[h]
\begin{center}
	\begin{tikzpicture}[scale=0.6]
		\draw [->-, line cap=round] (3,1)--(1,3);
		\draw [-->-, line cap=round] (5,2)--(1,3);
		\draw [-->-, line cap=round] (3,1)--(2,5);
		\draw [->-, line cap=round] (5,2)--(4,4);
		\draw [->-, line cap=round] (4,4)--(2,5);
		\draw [->-, line cap=round, out = 135, in = -90] (5,2)--(2,5);

		\plotperm{3,5,1,4,2};
		
	\end{tikzpicture}

\end{center}
\caption{By directing all edges of an inversion graph towards the northwest, we obtain a poset of order dimension 2.}
\label{fig:PosetFromPerm}
\end{figure}
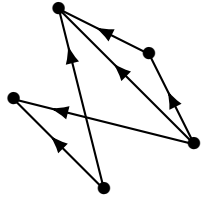

In the other direction, if a poset $P$ is obtained from the intersection of two total orders $T_1$ and $T_2$, then we can assume without loss of generality that $T_1 = \{ 1 < 2 <  \dots < n \}$, and also define a permutation $\sigma$ such that $T_2 = \{ \sigma(n) < \dots < \sigma(2) < \sigma(1) \}$. Thus, $P$ is the poset obtained from~$\sigma$ via the above construction.

The alert reader might have noticed that the term `dimension' is borrowed from geometry. In fact, this discussion intimates that this partial order definition for the class of inversion graphs is a thinly disguised rephrasing of the geometric definition given in Theorem~\ref{thm:plotting_def}. (That is, notice that we showed how to obtain the inversion graph of a permutation as the comparability graph of a poset with dimension at most 2 by looking at its plot in 2 dimensions.) In contrast, our final definition for this class of graphs makes no reference to geometry.

\begin{theorem}[Dushnik and Miller~\cite{dushnik:partially-order:}]\label{thm:comp_cocomp_def}
	A graph $G$ is the inversion graph of a permutation if and only if $G$ and $\overline{G}$ are comparability graphs.
\end{theorem}

Theorem~\ref{thm:dimension_def} tells us that all inversion graphs are comparability graphs, and thus it is clear that given a permutation $\pi$ both $G_{\pi}$ and $\overline{G}_{\pi} = G_{\pi^{\textnormal{r}}}$ are comparability graphs.

For the other direction, suppose that $G$ and $\overline{G}$ are comparability graphs. Then there exists a transitive orientation $\overrightarrow{E}(G)$ of the edges of $G$, and also a transitive orientation $\overrightarrow{E}(\overline{G})$ of the edges of $\overline{G}$. It is not too hard to see that the graph $(V(G), \overrightarrow{E}(G) \cup \overrightarrow{E}(\overline{G}))$ is an acyclic orientation of a complete graph, which thus yields a total order $T_1$ on the vertices of $G$ by translating each arc~$x \rightarrow y$ to mean $x < y$. Likewise, if we reverse each arc in $\overrightarrow{E}(\overline{G})$ we obtain another transitive orientation $\overleftarrow{E}(\overline{G})$, and it follows that the graph $(V(G), \overrightarrow{E}(G) \cup \overleftarrow{E}(\overline{G}))$ is another acyclic orientation of a complete graph, yielding another total order $T_2$ on the vertices of $G$. It follows that the comparability graph of the partial order $T_1 \cap T_2$ is $G$, and thus $G$ is the inversion graph of a permutation by Theorem~\ref{thm:dimension_def}.

\section{Content Preview}

In Chapter~\ref{chap:uniqueness}, we begin by discussing modular decomposition in graphs and its relationship to transitively orientable graphs. Next, we give a concise and novel proof of a result at the foundation of modular decomposition in graphs, which we contrast with the previously known proof. We then show how the proof method arises from the theory of permutations, and afterwards, relate the result back to permutations by discussing its structural and algebraic implications for inversion graphs.

In Chapter~\ref{chap:lettericity}, we see how the notion of geometric grid classes of permutations and the concept of letter graphs from structural graph theory are intimately related via inversion graphs. We then answer extremal questions about arbitrary graphs and inversion graphs in regard to a related parameter called lettericity, revealing that, as measured by this parameter, inversion graphs have very well-behaved structure. We then introduce a generalization of the letter graph construction, giving rise to further interesting combinatorial problems.

In Chapter~\ref{chap:reflections}, we observe how applying transpositions to a permutation can affect its inversion graph, thus enabling us to generalize transpositions to an operation on arbitrary graphs. In studying this operation, we make further connections between the subjects of permutation theory and graph theory.

%% file: chap-uniqueness.tex
\chapter{Modular Decomposition and Permutations}
\label{chap:uniqueness}

\newcommand{\rect}{\textnormal{rect}}
\newcommand{\Aut}{\textnormal{Aut}}
\newcommand{\TO}{\mathfrak{to}}
\newcommand{\rc}{\textnormal{rc}}

The modular decomposition of a graph is a hierarchical decomposition of its vertices that is key in designing efficient algorithms for finding transitive orientations, solving optimization problems, and graph drawing. Analogous to the factorization of integers into primes, modular decomposition canonically factors graphs into graphs that cannot be decomposed further. Accordingly, such indecomposable graphs are called prime. In this chapter, we study these prime graphs and their relationship to transitive orientability, and then describe some implications for the inversion graphs of permutations. As prime graphs are the building blocks of all graphs, these results have significance for all graphs.

\section{Preliminaries}

\subsection{Modular Decomposition and Edge Classes}

In his seminal 1967 paper \emph{Transitiv Orientierbare Graphen}~\cite{gallai:transitiv-orien:}, Gallai classifies the structure of all transitively orientable graphs, and in doing so, establishes modular decomposition in graphs. Originally in German, the developments in this paper paved the way for much of structural graph theory, and thus an English translation~\cite{gallai:a-translation-o:} has been added to the literature. Among the many consequences of this work is Theorem~\ref{thm:gallai}, the topic of this chapter. His proof of this fact is not explicitly stated, and does not give much further intuition about the structure of graphs beyond the result. We include here a concise and novel proof of this result, and then discuss its relationship to the theory of permutations. We begin with the definitions required to understand this result.

Given a graph $G$, a \emph{module} or \emph{homogeneous set} is a set $M \subseteq V(G)$ such that for all vertices $u,v \in M$, we have $N(u) \setminus M = N(v) \setminus M$. That is, each vertex not in $M$ is either adjacent or nonadjacent to every vertex in $M$. Every subset $M \subseteq V(G)$ of size $0$, $1$ or $|V(G)|$ is a module. We call these modules \emph{trivial}, and define a graph to be \emph{prime} if it contains no nontrivial modules. It is prime graphs that are the subject of Theorem~\ref{thm:gallai}.

Gallai's motivation for considering modules was to study the structure of comparability graphs and their possible transitive orientations. If we take an undirected graph and seek to give it a transitive orientation, we see that the orientation of one edge may imply the orientation of many other edges. For example, if $ab$ and~$bc$ are edges of $G$ and $ac$ is not, then in any transitive orientation of $G$, the only two possible orientations of these edges are
\[
	a \leftarrow b \rightarrow c
	\quad\text{and}\quad
	a \rightarrow b \leftarrow c.
\]
This means that the orientation of $ab$ determines the orientation of $bc$, and vice versa.

Following Gallai, we define a relation $\wedge$ on the edges of $G$ in which $ab \wedge ab$ for all edges $ab$ and
\[
	ab \wedge bc
\]
if $ab$ and $bc$ are edges and $ac$ is not. Thus, if $e\wedge f$, then in any transitive orientation of $G$, the orientation of one of these edges determines the orientation of the other.

Because $\wedge$ is reflexive and symmetric by definition, its transitive closure is an equivalence relation. We define the \emph{edge classes} of $G$ to be the equivalence classes of $E(G)$ with respect to this relation. That is, two edges $e$ and $f$ lie in the same edge class of $G$ if and only if there is a sequence of edges
\[
	e=e_0\wedge e_1\wedge\cdots\wedge e_t=f.
\]
Thus, clearly, if $e$ and $f$ lie in the same edge class of $G$, then in any transitive orientation of $G$, the orientation of one of these edges determines the orientation of the other. Furthermore, if a graph~$G$ is transitively orientable, then the edges of each edge class have exactly two possible transitive orientations. We obtain one of these orientations by directing one of the edges, hence deciding the rest of the directions of the edges in the edge class, and the other orientation by reversing all of the arcs in the first orientation. That is, the dual of the digraph induced by the arcs of the first orientation. We can now state the result of interest.

\begin{theorem}[Gallai~\cite{gallai:transitiv-orien:}]\label{thm:gallai}
	If the graph $G$ is prime, then all of its edges belong to the same edge class. That is, prime graphs have 0 or 2 transitive orientations.
\end{theorem}

\subsection{Chains}

In 2006, Brignall, Huczynska, and Vatter~\cite{brignall:decomposing-sim:} introduced the geometric notion of the proper pin sequence to obtain enumerative results about permutations. Indeed, the permutations they studied are those that have prime inversion graphs. Proper pin sequences have very nice graph theoretic properties when their entries are considered as the vertices of the inversion graphs of the permutations in which they lie. We will discuss proper pin sequences in some depth in Section~\ref{sec:proper_pin_sequences}.

In 2015, Chudnovsky, Kim, Oum, and Seymour~\cite{chudnovsky:unavoidable-ind:} defined chains for arbitrary graphs to study the induced subgraphs that must exist in large prime graphs. (Note that these chains are different from the ``chains'' of the English translation of Gallai~\cite{gallai:a-translation-o:}, or the ``Ketten'' of Gallai's original German version~\cite{gallai:transitiv-orien:}.) These are the graphical analogs of proper pin sequences. That is, a chain is any sequence of vertices in a graph satisfying the same nice graph theoretic properties as a proper pin sequence. More explicitly, a sequence $p_1$, $p_2$, $\dots$,~$p_m$ of distinct vertices of a graph $G$ is called a \emph{chain} if for all~$i \geq 3$, we have either
\begin{enumerate}
	\item $p_i \sim p_{i-1}$ and $p_i \not\sim p_0, p_1, \dots, p_{i-2}$, or
	\item $p_i \not\sim p_{i-1}$ and $p_i \sim p_0, p_1, \dots, p_{i-2}$.
\end{enumerate}

We see from this definition that chains are generalizations of paths. Furthermore, just as there is a path between any two vertices in a graph if and only if it is connected, the following result shows that chains play the analogous role in relation to primeness.

\begin{theorem}[Chudnovsky, Kim, Oum, and Seymour~\text{\cite[Proposition 2.1]{chudnovsky:unavoidable-ind:}}]\label{thm:CKOS}
	The graph~$G$ is prime if and only if there exists a chain from each pair $u$, $v$ of vertices to every other vertex $w$.
\end{theorem}

\begin{proof}
First, suppose that the graph $G$ is not prime with a nontrivial module $M$. Since $M$ is nontrivial, we can choose distinct vertices $u,v \in M$ and another vertex $w \notin M$. We then suppose for a contradiction that there is a chain $p_1, p_2, \dots, p_m$ such that $p_1 = u$, $p_2 = v$ and $p_m = w$. Let~$i \in \{ 3, \dots , m \}$ be the smallest index such that $p_i \notin M$, which must exist since $p_m \notin M$. This implies that $p_i$ does not agree on the vertices $\{ p_1, \dots, p_{i-1}\} \subseteq M$, and thus $M$ cannot be a module, a contradiction. 

Now suppose that $G$ is prime and let $u$ and $v$ be any two distinct vertices. Further, let $Z$ denote the set of vertices of $G$ that can be reached by a chain starting with $u$ and $v$. If $Z = V(G) \setminus \{u,v \}$, then we are done, so we suppose not and let $w$ be any vertex not in $Z \cup \{ u,v\}$. Since $w \notin Z$, it follows that $u,v,w$ cannot be a chain, and thus $w$ agrees on $\{ u,v \}$. Furthermore, let~$z$ be any element in $Z$ and $p_1, p_2, \dots, p_m$ any chain such that $p_1 = u$, $p_2 = v$ and $p_m = z$. If $w$ disagrees on the set~$\{ u,v,z \}$, then there exists a smallest index $i \in \{ 3, \dots , m \}$ such that $w$ disagrees on $\{ p_1, p_2, \dots, p_i \}$, and it follows that $p_1, p_2, \dots, p_i, w$ is a chain, contradicting the fact that $w \notin Z$. Therefore, $w$ does agree on~$\{ u,v,z \}$. Repeating this argument for each $z \in Z$, it follows that $w$ agrees on $Z \cup \{u,v \}$. We can also repeat this argument for each $w \notin Z \cup \{ u,v \}$, revealing that each vertex not in $Z \cup \{ u,v \}$ agrees on every vertex in this set. That is, $Z \cup \{ u,v \}$ is a nontrivial module, contradicting the fact that $G$ is prime. This gives the result.
\end{proof}

This theorem and the following proposition are the only background results we need for our proof.

\begin{proposition}\label{prop:disconnected_chains}
For any chain $p_1, p_2, \dots, p_m$ in a graph $G$, either $G[p_1, \dots, p_m]$ is connected or~${p_1\not\sim p_2}$ and $G[p_1, \dots, p_m]\cong P_1\cup P_{m-1}$.
\end{proposition}

\begin{proof}
	Let  $p_1, p_2, \dots, p_m$ be any chain. If for some $j \in \{ 2, \dots, m-1\}$ the prefix $p_1, \dots, p_j$ of the chain is such that $G[p_1, \dots, p_j]$ is connected, then $G[p_1, \dots, p_j, p_{j+1}]$ is connected since $p_{j+1}$ is adjacent to at least one of $p_1, \dots, p_j$. Thus, we see that $G[p_1, \dots, p_m]$ is connected by iterating this argument. Hence, if $p_1 \sim p_2$ we have that $G[p_1, p_2]$ is connected, and we are done. Furthermore, if for some~$j \geq 4$ we have $p_j \not\sim p_{j-1}$ and $p_j \sim p_1, \dots, p_{j-2}$, then $G[p_1, \dots, p_j]$ is connected, and we are also done. The result then follows by analyzing the resulting induced graph in the case that both~$p_1 \not\sim p_2$, and $p_i \sim p_{i-1}$ with $p_i \not\sim p_1, \dots, p_{i-2}$ for each $i \geq 4$.
\end{proof}

In the next section, we give a concise proof of Theorem~\ref{thm:gallai} using chains, and also include Gallai's original proof for the convenience of the reader. Then, we return to proper pin sequences in Section~\ref{sec:proper_pin_sequences}, where we will give their definition and observe how these considerations in the general graph setting regarding chains can be understood geometrically in the permutation setting using proper pin sequences. We will then also discuss two important corollaries to Theorem~\ref{thm:gallai} for the theory of permutations in Section~\ref{sec:permutations}.

\section{Proofs That Prime Graphs Have One Edge Class}\label{sec:the_proof}

\subsection{Proof by Chains}

We break the proof up into simple lemmas.

\begin{lemma}\label{lemma:p1p2_pspm_same_edge_class}
Suppose that $p_1, p_2, \dots, p_m$ is a chain in a graph $G$ such that $p_1\sim p_2$ and~$s={\max\{i : p_i\sim p_{i+1}\}}$. Then, $p_s$ is adjacent to $p_m$, and the edge $p_s p_m$ lies in the same edge class as the edge $p_1 p_2$.
\end{lemma}

\begin{proof}
We proceed by induction on $m$. If $m = 2$, then the result is trivial. Now suppose that~$p_1, \dots, p_m, p_{m+1}$ is such a chain and that the result holds for all chains of length $m \geq 2$. That is, considering the chain $p_1, \dots, p_m$, the edge $p_s p_m$ is in the same edge class as $p_1 p_2$ where $s$ is the greatest index less than $m$ such that $p_s \sim p_{s+1}$. We have two cases, based on which type of chain element $p_{m+1}$ is.

First consider the case where $p_{m+1}\sim p_m$ and $p_{m+1}\not\sim p_1,p_2,\dots,p_{m-1}$. Then, $m={\max\{i:p_i\sim p_{i+1}\}}$, so we seek to show that $p_m p_{m+1}$ lies in the same edge class as $p_1 p_2$. We know by induction that $p_s p_m$ lies in the same edge class as $p_1 p_2$, and by our hypothesis that~$p_m p_{m+1}\wedge p_s p_m$, so it follows that $p_m p_{m+1}$ lies in the same edge class as $p_1 p_2$, as desired.

Next suppose that $p_{m+1}\not\sim p_m$ and $p_{m+1}\sim p_1,p_2,\dots,p_{m-1}$. Then, $s={\max\{i:p_i\sim p_{i+1}\}}$, so we seek to show that $p_s p_{m+1}$ lies in the same edge class as $p_1 p_2$. We know by induction that~$p_s p_m$ lies in the same edge class as $p_1 p_2$, and by our hypothesis that $p_s p_m \wedge p_s p_{m+1}$, so it follows that~$p_s p_{m+1}$ lies in the same edge class as $p_1 p_2$, as desired.
\end{proof}

This lemma makes the next one easy to prove.

\begin{lemma}\label{lemma:final}
Suppose that $p_1, p_2, \dots, p_m$ is a chain in a graph $G$ such that $G[p_1,p_2,p_m]$ is isomorphic to $K_3$. Then, at least one of $p_1 p_m$ or $p_2 p_m$ is in the same edge class as $p_1 p_2$.
\end{lemma}

\begin{proof}
Define $s=\max\{i:p_i\sim p_{i+1}\}$. By the previous result, $p_1 p_2$ and $p_s p_m$ lie in the same edge class, and by our hypotheses, $p_m\sim p_1, p_2$. Thus we must have that $m\ge 4$ and $p_m \not\sim p_{m-1}$ with $p_m \sim p_1, \dots, p_{m-2}$. Thus, we must have $s \leq m-2$. Given that the sequence of vertices $p_1$, $p_2$, $\dots$, $p_{m-2}$ comprises a chain in~$G$, it also comprises a chain in $\overline{G}$. By Proposition~\ref{prop:disconnected_chains}, we then know that $p_s$ is in the same component of $\overline{G}[p_1, \dots, p_{m-2}]$ as at least one of $p_1$ and $p_2$. If for $k=1$ or $k=2$, $p_k$ lies in the same component as $p_s$, then there exists a path $p_s, v_1, \dots, v_r, p_k$ in~$\overline{G}[p_1, \dots, p_{m-2}]$, and therefore $p_s p_m \wedge v_1 p_m \wedge \dots \wedge v_r p_m \wedge p_k p_m$. That is, $p_k p_m$ lies in the same edge class as $p_1 p_2$ since they are both in the same edge class as $p_s p_m$. This gives the result.
\end{proof}

Now, let $G$ be any prime graph. If two incident edges $e_1$, $e_2$ induce a $P_3$, then we are done since $e_1 \wedge e_2$. Now we deal with the case where $G[x,y,z] \cong K_3$. By Theorem~\ref{thm:CKOS}, there exists a chain $p_1 = x$, $p_2 = y$, $\dots$, $p_m=z$. By Lemma~\ref{lemma:final}, at least one of $xz$ and $yz$ is in the same edge class as $xy$. If they are both, then we are done.  Otherwise, suppose $yz$ is the one that is not necessarily in the same edge class as $xy$. Then, again using Theorem~\ref{thm:CKOS}, there exists a chain $p_1' = y$, $p_2' = z$, $\dots$, $p_{m'}' = x$, then we must have that $yz$ is in the same edge class as one of $xy$ or $xz$. In either case, we are done by transitivity. That is, all incident edges in $G$ are in the same edge class, and since $G$ must be connected, we have proved Theorem~\ref{thm:gallai}.

\subsection{Gallai's Proof}\label{section:gallai's_proof}

Theorem~\ref{thm:gallai} is not explicitly stated by Gallai in~\cite{gallai:transitiv-orien:}, and is instead a corollary to a combination of his results. We consolidate those results here for the convenience of the reader, as well as to emphasize the need for a proof like the one given above. To aid in our discussions, given a subset $E$ of edges of a graph $G$, we let $V(E) \subseteq V(G)$ denote the vertices that are incident to at least one edge in $E$. We note that the citations for the following propositions are directed to the corresponding results in the English translation~\cite{gallai:a-translation-o:} of~\cite{gallai:transitiv-orien:}.

\begin{proposition}[Gallai~\text{\cite[3.2.1]{gallai:a-translation-o:}}]\label{prop:shortest_ab-path}
	For any $a, b \in V(G)$, all edges of the shortest $ab$-path $W = (a = x_0, x_1, \dots, x_t = b)$ of $G$ belong to the same edge class of $G$.
\end{proposition}

\begin{proof}
	It is trivial for $t=1$. For $t \geq 2$, we have that $x_{i-1} x_i \wedge x_i x_{i+1}$ for each $i \in \{ 1, \dots, t-1 \}$. That is, if $x_{i-1} x_{i+1}$ is an edge for some $i$, then $W$ is not a shortest $ab$-path.
\end{proof}

\begin{proposition}[Gallai~\text{\cite[3.2.2]{gallai:a-translation-o:}}]\label{prop:XY-edges_one_class}
	Let $X$ and $Y$ be two disjoint nonempty subsets of~$V(G)$ that are completely adjacent in $G$, and such that $\overline{G}[X]$ and $\overline{G}[Y]$ are both connected. Then all $XY$-edges belong to the same edge class.
\end{proposition}

\begin{proof}
	Let $x$ be any vertex in $X$ and $y$, $y'$ any two vertices in $Y$. Then since $\overline{G}[Y]$ is connected, there exists a path $y = y_0, y_1, \dots, y_s = y'$ through $\overline{G}[Y]$. That is, $xy_0 \wedge xy_1 \wedge \dots \wedge xy_s$, and therefore $xy$ and $xy'$ are in the same edge class. This shows that all $xY$-edges belong to the same edge class, and the same argument can be used to show all $yX$-edges belong to the same edge class for any $y \in Y$. Finally, if $xy$ and $x'y'$ are two nonincident $XY$-edges, then they both belong to the same edge class as $xy'$, and thus the same edge class as each other.
\end{proof}

The following is the key result to Gallai's proof, and at the core of modular decomposition in graphs.

\begin{proposition}[Gallai~\text{\cite[Lemma 3.2.3]{gallai:a-translation-o:}}]\label{prop:unique_spanning_edge_class}
	If $G$ and $\overline{G}$ are both connected and have at least two vertices, then $G$ has at least four vertices and possesses exactly one edge class $E$ such that~$V(E) = V(G)$.
\end{proposition}

\begin{proof}
The fact that $G$ must have four vertices is trivial after examining all graphs on two and three vertices. Since $G$ and $\overline{G}$ are connected, neither of them are complete. Given~$\overline{G}$ is not complete, there exists a proper subset $X \subsetneq V$ such that $\overline{G}[V(G) - X]$ is not connected. In this case, we call $X$ a \emph{cutset} of~$\overline{G}$. We let $L$ be a minimal cutset of $\overline{G}$ and let $K_1, \dots, K_q$, $(q \geq 2)$, be the vertex sets of each of the components in $\overline{G}[V(G) - L]$. As $L$ is minimal, every vertex of $L$ is adjacent to at least one vertex in each $K_i$ in~$\overline{G}$. That is, if a vertex $v$ in $L$ is nonadjacent to every vertex in some $K_i$ in $\overline{G}$, then since~$\overline{G}$ is connected we must have $|L| > 1$, and further that $L - \{ v\}$ is also a cutset, contradicting our choice of $L$. 
	
For each pair $i \neq j$, the vertex sets $K_i$ and $K_j$ are completely adjacent in $G$, and furthermore, all~$K_i K_j$-edges belong to the same edge class of $G$ by Proposition~\ref{prop:XY-edges_one_class}. Next, since $G$ is connected, there exists an edge between $L$ and some $K_j$ in $G$, say $K_1$. We let $a_1 x_1$ be an edge of~$G$ such that $a_1 \in K_1$ and $x_1 \in L$. By the previous paragraph, we can choose a vertex $b_i \in K_i$ for each $i \in \{ 2, \dots, q \}$ such that $x_1 b_i$ is an edge in $\overline{G}$. Therefore $a_1 x_1 \wedge a_1 b_i$ for each $i \in \{ 2, \dots, q \}$, and we obtain that all $K_1 K_i$ edges are in the same edge class $E$ with $a_1 x_1$. Since this is true for any choice of $a_1 x_1$, we also have that all~$K_1 L$-edges are in $E$. Thus $V(G)-L \subseteq V(E)$.
	
To see $V(E) = V(G)$, we will show that $L \subseteq V(E)$. So let $x$ be any vertex in $L$ and $a_1$ any vertex in~$K_1$. Let $W = (x=x_0, x_1, \dots, x_t=a_1)$ be a shortest $x a_1$-path in $G$. From above, we saw that all edges between $K_1$ and $V(G)-K_1$ are in $E$, and since $x$ is in $V(G)-K_1$ and $a_1 \in K_1$, the path $W$ must contain an edge in $E$ between $K_1$ and $V(G)-K_1$. That is, by Proposition~\ref{prop:shortest_ab-path}, each edge in $W$ is in $E$, including the edge incident with $x$. This shows $L \subseteq V(E)$.

Finally, to see that $E$ is the only edge class such that $V(E) = V(G)$, we again use the fact that all $K_1(V(G) - K_1)$-edges are in $E$. That is, if $E'$ is another edge class such that $V(E') = V(G)$, then it must have edges within $K_1$ and within $V(G)-K_1$, but none between these two sets since $E \cap E' = \varnothing$. However, this would contradict the fact that the graph induced by the edges in the edge class $E'$, i.e., $(V(E'), E')$, must be connected. This gives the result.
\end{proof}

\begin{proposition}[Gallai~\text{\cite[Theorem 3.1.5]{gallai:a-translation-o:}}]\label{prop:edge_classes_induce_modules}
	If $E$ is an edge class of a graph $G$, then $V(E)$ is a module in $G$.
\end{proposition}

\begin{proof}
	Let $E$ be an arbitrary edge class of $G$. The result is trivial if $V(E) = V(G)$, so we may assume that $V(E)$ is a proper subset of $V(G)$. If $V(E)$ is not a module, there exists a vertex~$z$ in $V(G) - V(E)$ such that both sets $X = V(E) \cap N(Z)$ and $Y = V(E) - N(Z)$ are nonempty. Since the graph $(V(E), E)$ is connected, there must exist an edge $xy$ of $G$ such that $x \in X$ and $y \in Y$. That is, $xy \wedge xz$ and therefore $z \in V(E)$, a contradiction.
\end{proof}

To see how these results imply Theorem~\ref{thm:gallai}, we let $G$ be a prime graph. If $G$ has one or two vertices then the result is trivial, and there are no prime graphs on three vertices. So we suppose that $G$ has at least four vertices, and note that both $G$ and $\overline{G}$ must be connected. Then by Proposition~\ref{prop:unique_spanning_edge_class}, there exists exactly one edge class $E$ of $G$ such that $V(E) = V(G)$. Now, suppose that $E'$ is any edge class of $G$. Since $E'$ contains at least one edge, we have that $V(E') \geq 2$. Since $V(E')$ is a module by Propostion~\ref{prop:edge_classes_induce_modules}, it must be that $V(E') = V(G)$ because $G$ does not contain any non-trivial modules. That is, $E'$ must be $E$.

\section{The Geometric Origin of Chains}\label{sec:proper_pin_sequences}

\subsection{Modular Decomposition in Inversion Graphs}

As previously mentioned, chains have a geometric precursor in proper pin sequences. Indeed, chains are generalizations of proper pin sequences as they are sequences of entries of a permutation that correspond to chains in the induced inversion graph. Before we can define proper pin sequences, we require several definitions.

For a permutation $\pi$ of $[n]$, an \emph{interval} is a set of contiguous indices $I = [a,b]$ such that the set of values $\pi(I) = \{\pi(i): i \in  I\}$ also forms an interval of natural numbers. It is clear that every subset of $[n]$ of size $0$, $1$ or $n$ forms an interval of $\pi$. We call these intervals \emph{trivial}, and define a \emph{simple permutation} to be a permutation that has no nontrivial intervals. 

\begin{figure}[h]
\begin{center}
	\begin{tikzpicture}[scale=0.32]
		\node [] at (5.5,-1) {(a) A nonsimple permutation};
		
		\draw [color = gray, line width = 0.6pt, dotted, thick, line cap=round] (0.5,4)--(3,4);
		\draw [color = gray , line width = 0.6pt, dotted, thick, line cap=round] (6,4)--(10.5,4);
		\draw [color = gray , line width = 0.6pt, dotted, thick, line cap=round] (0.5,7)--(3,7);
		\draw [color = gray , line width = 0.6pt, dotted, thick, line cap=round] (6,7)--(10.5,7);
		\draw [color = gray , line width = 0.6pt, dotted, thick, line cap=round] (3,0.5)--(3,4);
		\draw [color = gray , line width = 0.6pt, dotted, thick, line cap=round] (3,7)--(3,10.5);
		\draw [color = gray , line width = 0.6pt, dotted, thick, line cap=round] (6,0.5)--(6,4);
		\draw [color = gray , line width = 0.6pt, dotted, thick, line cap=round] (6,7)--(6,10.5);
		\draw [line width = 0.6pt] (3,4) rectangle (6,7);

		\plotperm{3,1,6,4,7,5,9,2,10,8};
		
		\draw [line cap=round] (1,3) -- (2,1);
		\draw [line cap=round] (1,3) -- (8,2);
		
		\draw [line cap=round] (3,6) -- (8,2);
		\draw [line cap=round] (4,4) -- (8,2);
		\draw [line cap=round] (6,5) -- (8,2);
		\draw [in = 110, out = -30] (5,7) to (8,2);
		\draw [line cap=round] (7,9) -- (8,2);
		\draw [line cap=round] (7,9) -- (10,8);
		\draw [line cap=round] (9,10) -- (10,8);
		
		\draw [line cap=round] (3,6) -- (4,4);
		\draw [line cap=round] (3,6) -- (6,5);
		\draw [line cap=round] (5,7) -- (6,5);
		
	\end{tikzpicture}
	\hspace{15mm}
	\begin{tikzpicture}[scale=0.32]
		\node [] at (5.5,-1) {(b) A simple permutation};

		\plotpermgraph{2,4,1,9,3,5,8,6,10,7};
		\draw [in = 135, out = 0] (4,9) to (10,7);
		
	\end{tikzpicture}
\end{center}
\caption{The inversion graphs of two permutations plotted in the plane.}
\label{fig:NonSimpleAndSimplePermutations}
\end{figure}
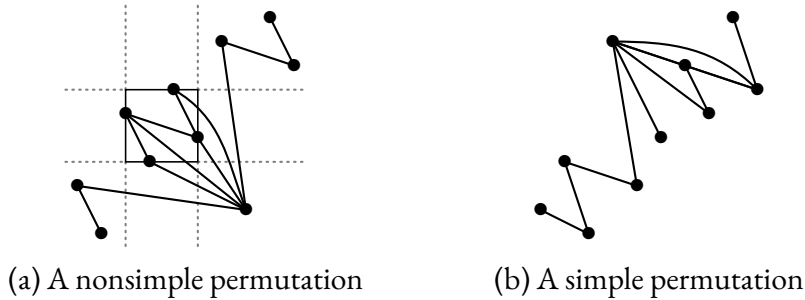

Figure~\ref{fig:NonSimpleAndSimplePermutations} contains the inversion graphs of two permutations drawn on their plots in the plane. The permutation in Figure~\ref{fig:NonSimpleAndSimplePermutations}(a) can be seen to have a nontrivial interval given by the four vertices in the black square, and thus is not a simple permutation. That is, the permutation maps $[3,6]$ to $[4,7]$. 

We let $\rect(S)$ denote the smallest axis-parallel rectangle containing the set $S$ of points in the plane. For example, the black square in Figure~\ref{fig:NonSimpleAndSimplePermutations}(a) is $\rect([4,7])$, (recall that we are labeling vertices by the values of their corresponding entries). As emphasized by the gray dotted lines in the figure, since the values $[4,7]$ constitute an interval, there are no points	 lying directly left, right, above or below $\rect([4,7])$. Furthermore, since the permutation plotted in Figure~\ref{fig:NonSimpleAndSimplePermutations}(b) is simple, for any subset $K$ of its vertices such that $|K| > 1$ and $\rect(K) \neq [1,10] \times [1,10]$, there is at least one entry lying directly left, right, above or below $\rect(K)$.

Now, given a permutation $\pi$, it is straightforward to see that every interval corresponds to a module of $G_{\pi}$. For example, observe the black square containing the entries of the interval of the nonsimple permutation in Figure~\ref{fig:NonSimpleAndSimplePermutations}(a). As previously remarked, no entries of the permutation can lie directly left, right, above or below this square. Hence, every entry not in the interval either forms an inversion with each entry in the interval (that is, the vertices above and to the left or below and to the right), or does not form an inversion with any entry of the interval (that is, the vertices above and to the right or below and to the left). Thus, every vertex not in the interval agrees on every vertex in the interval. The following is the correspondence that one might hope to be true.

\begin{proposition}\label{prop:simple_is_prime}
	The permutation $\pi$ is simple if and only if its inversion graph~$G_{\pi}$ is prime.
\end{proposition}

\begin{proof}
	As we have already shown, if the inversion graph $G_{\pi}$ is prime, then $\pi$ cannot have a nontrivial interval, and therefore $\pi$ is simple. We now show that if $G_{\pi}$ isn't prime, then~$\pi$ must have a nontrivial interval.
	
	So suppose that $G_{\pi}$ has a nontrivial module $M$. With the entries of $\pi$ plotted in the usual way, we see that there are no entries lying directly left, right, above or below $\rect(M)$. That is, an entry directly below or above $\rect(M)$ would disagree on the elements of $M$ lying on the left and right sides of $\rect(M)$, and similarly, an entry directly to the right or left would disagree on the elements of $M$ lying on the top and bottom sides. 

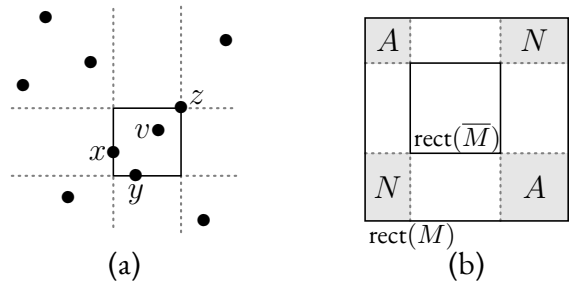
\begin{figure}[h]
\begin{center}
	\begin{tikzpicture}[scale=0.3]
		\node [] at (5.5,-1) {(a)};
		
		\draw [color = gray ,line width = 0.6pt, dotted, thick, line cap=round] (0.5,3)--(5,3)--(5,0.5);
		\draw [color = gray ,line width = 0.6pt, dotted, thick, line cap=round] (8,0.5)--(8,3)--(10.5,3);
		\draw [color = gray ,line width = 0.6pt, dotted, thick, line cap=round] (0.5,6)--(5,6)--(5,10.5);
		\draw [color = gray ,line width = 0.6pt, dotted, thick, line cap=round] (8,10.5)--(8,6)--(10.5,6);
		\draw [line width = 0.6pt] (5,3) rectangle (8,6);

		\plotperm{7,10,2,8,4,3,5,6,1,9};
		
		\node [] at (4.3, 4) {$x$};
		\node [] at (6, 2.2) {$y$};
		\node [] at (8.7, 6.5) {$z$};
		
		\node [] at (6.3, 5) {$v$};
	\end{tikzpicture}
	\hspace{15mm}
	\begin{tikzpicture}[scale=0.3]
		\node [] at (5.5,-1) {(b)};
		
		\fill [color = gray!20] (1,1) rectangle (3,4);
		\node [] at (2,2.5) {$N$};
		\fill [color = gray!20] (1,8) rectangle (3,10);
		\node [] at (2,9) {$A$};
		\fill [color = gray!20] (7,1) rectangle (10,4);
		\node [] at (8.5,2.5) {$A$};
		\fill [color = gray!20] (7,8) rectangle (10,10);
		\node [] at (8.5,9) {$N$};
		
		\draw [color = gray, line width = 0.6pt, dotted, thick, line cap=round] (1,4)--(3,4);
		\draw [color = gray, line width = 0.6pt, dotted, thick, line cap=round] (1,8)--(3,8);
	
		\draw [color = gray, line width = 0.6pt, dotted, thick, line cap=round] (3,1)--(3,4);
		\draw [color = gray, line width = 0.6pt, dotted, thick, line cap=round] (7,1)--(7,4);
		\draw [color = gray, line width = 0.6pt, dotted, thick, line cap=round] (7,4)--(10,4);
		\draw [color = gray, line width = 0.6pt, dotted, thick, line cap=round] (7,8)--(10,8);
		\draw [color = gray, line width = 0.6pt, dotted, thick, line cap=round] (7,8)--(7,10);
		\draw [color = gray, line width = 0.6pt, dotted, thick, line cap=round] (3,8)--(3,10);
		
		\draw [line width = 0.6pt] (1,1) rectangle (10,10);
		\node [anchor = west] at (1, 0.3) {\footnotesize$\rect(M)$};
		
		
		\draw [line width = 0.6pt] (3,4) rectangle (7,8);
		\node [anchor = west] at (3, 4.7) {\footnotesize$\rect(\overline{M})$};
		
	\end{tikzpicture}
\end{center}
\vspace{-4mm}
\caption{Cases in the proof of Proposition~\ref{prop:simple_is_prime}.}
\label{fig:simple_iff_prime}
\end{figure}

	If $\rect(M)$ does not contain all of the entries of $\pi$, then it is clear that all of the entries inside and along the border of $\rect(M)$ together form an nontrivial interval. For example, in Figure~\ref{fig:simple_iff_prime}(a) the vertices~$\{ x,y,z \}$ comprise a module in the inversion graph, but we must include the vertex $v$ lying inside $\rect(x,y,z )$ to obtain an interval in the permutation.
	
	Otherwise, $\rect(M)$ does contain all of the entries of $\pi$. Let $\overline{M}$ denote the points not in $M$, which is nonempty since $M$ is nontrivial. We note $\rect(\overline{M})$ must lie strictly inside $\rect(M)$, and that all entries not inside $\rect(\overline{M})$ all lie inside either the gray regions labeled~$N$ in Figure~\ref{fig:simple_iff_prime}(b), or the gray regions labeled $A$. Crucially, these entries cannot lie in both. It is then clear that~$\pi$ contains a nontrivial interval.
\end{proof}

\subsection{Proper Pin Sequences}

We are now ready to discuss pin sequences. Taking two points $p_1$ and $p_2$ in the plot of some permutation $\pi$, if the points contained in $\rect(p_1, p_2)$ do not form an interval, then there exists at least one point that `slices' it. That is, there is a point that lies above or below $\rect(p_1, p_2)$ that slices it vertically, or a point to the  left or right that slices it horizontally. We call such a point a \emph{pin}. We begin a sequence by selecting such a pin and labeling it $p_3$. If now $\rect(p_1, p_2, p_3)$ does not constitute an interval, we can find another pin $p_4$ outside $\rect(p_1, p_2, p_3)$ that slices it. Continuing in this way, we obtain a sequence $p_1, p_2, p_3, \dots$ called a \emph{pin sequence}. See Figure~\ref{fig:PinSequenceExample} for an example.

\begin{figure}[h]
\begin{center}
	\begin{tikzpicture}[scale=0.4]
		\draw[line width = 0.6pt] (3,1) rectangle (4,4);
		
		\node [] at (2.6, 0.5) {\small$p_1$};
		\node [] at (3.4, 4.5) {\small$p_2$};
		\node [] at (0.32, 3) {\small$p_3$};
		
		\draw [gray, line cap=round] (1,3) to (3.5,3);
		
		\plotperm{3,6,1,4,2,5};
		
	\end{tikzpicture}
	\hspace{12mm}
	\begin{tikzpicture}[scale=0.4]
		\draw[line width = 0.6pt] (1,1) rectangle (4,4);
		
		\node [] at (2.6, 0.5) {\small$p_1$};
		\node [] at (3.4, 4.5) {\small$p_2$};
		\node [] at (0.32, 3) {\small$p_3$};
		\node [] at (5.8, 2) {\small$p_4$};
		
		\draw [gray, line cap=round] (1,3) to (3.5,3);
		\draw [gray, line cap=round] (5,2) to (2.5,2);
		
		\plotperm{3,6,1,4,2,5};
		
	\end{tikzpicture}
	\hspace{12mm}
	\begin{tikzpicture}[scale=0.4]
		\draw[line width = 0.6pt] (1,1) rectangle (5,4);
		
		\node [] at (2.6, 0.5) {\small$p_1$};
		\node [] at (3.4, 4.5) {\small$p_2$};
		\node [] at (0.32, 3) {\small$p_3$};
		\node [] at (5.8, 2) {\small$p_4$};
		\node [] at (1.7, 6.7) {\small$p_5$};
		
		\draw [gray, line cap=round] (1,3) to (3.5,3);
		\draw [gray, line cap=round] (5,2) to (2.5,2);
		\draw [gray, line cap=round] (2,6) to (2,2.5);
		
		\plotperm{3,6,1,4,2,5};
		
	\end{tikzpicture}
	\hspace{12mm}
	\begin{tikzpicture}[scale=0.4]
		\draw[line width = 0.6pt] (1,1) rectangle (5,6);
		
		\node [] at (2.6, 0.5) {\small$p_1$};
		\node [] at (3.4, 4.5) {\small$p_2$};
		\node [] at (0.32, 3) {\small$p_3$};
		\node [] at (5.8, 2) {\small$p_4$};
		\node [] at (1.7, 6.7) {\small$p_5$};
		\node [] at (6.8, 5) {\small$p_6$};
		
		\draw [gray, line cap=round] (1,3) to (3.5,3);
		\draw [gray, line cap=round] (5,2) to (2.5,2);
		\draw [gray, line cap=round] (2,6) to (2,2.5);
		\draw [gray, line cap=round] (6,5) to (1.5,5);
		
		\plotperm{3,6,1,4,2,5};
		
	\end{tikzpicture}
\end{center}
\caption{A pin sequence.}
\label{fig:PinSequenceExample}
\end{figure}

A \emph{proper pin sequence} is simply a pin sequence $p_1, p_2, \dots$ that satisfies the \emph{separation condition}: for each $i \geq 2$, the pin $p_{i+1}$ must separate $p_i$ from $\rect (p_1, \dots, p_{i-1})$. That is, there is a vertical or horizontal line through $p_{i+1}$ that lies between $\rect (p_1, \dots, p_{i-1})$ and $p_i$. (Note that some authors~\cite{brignall:decomposing-sim:} have previously included the maximality condition in their definition of proper pin sequence, but this condition is not necessary for this work. Other authors~\cite{brignall:simple-permutat:decide:} have included the externality condition, but this condition is implicit from the definition of pin sequences.) For example, we see in the pin sequence in Figure~\ref{fig:PinSequenceExample} that $p_6$ separates $p_5$ (above $p_6$) from $p_1$, $p_2$, $p_3$ and $p_4$ (all below $p_6$). However, $p_5$ doesn't separate $p_4$ from $\rect (p_1, p_2,p_3)$, and $p_4$ does not separate $p_3$ from $\rect(p_1,p_2)$, and thus the sequence is not a proper pin sequence. It is not too difficult to see that proper pin sequences are chains.

\begin{proposition}\label{prop:pins_adjacencies}
	In the inversion graph of a permutation containing a proper pin sequence $p_1, p_2, \dots$, the pin $p_{i+1}$ satisfies either $p_{i+1} \sim p_1, \dots, p_{i-1}$ and $p_{i+1} \not\sim p_{i}$, or $p_{i+1} \not\sim p_1, \dots, p_{i-1}$ and $p_{i+1} \sim p_{i}$, (for all~$i \geq 2$).
\end{proposition}

\begin{proof}
	Suppose without loss of generality that $p_{i+1}$ is a right pin as the proof will make the result clear for pins in all directions. From the definition of pin sequences, $p_{i+1}$ is to the right of each of $p_1, \dots, p_{i}$. By the separation condition, one of the pin sets $\{ p_1, \dots, p_{i-1}\}$ or $\{ p_{i} \}$ is above the horizontal line through $p_{i+1}$, and the other set is below it. Then $p_{i+1}$ is adjacent to the pin(s) that lie above this line, and nonadjacent to the pin(s) that lie below. This gives the result.
\end{proof}

The following is the analogous result to Theorem~\ref{thm:CKOS}, understood instead via some geometric reasoning.

\begin{proposition}[Brignall, Huczynska, and Vatter~\text{\cite[Lemma 3.6]{brignall:decomposing-sim:}}]\label{prop:BHV_k-reaching}
	Let $\pi$ be a simple permutation and~$x$ and $y$ any two vertices of $G_{\pi}$. Then for any point $z$ not lying inside $\rect(x,y)$ in the plot of $\pi$, there is a proper $z$-reaching pin sequence beginning with $x$ and $y$.
\end{proposition}

\begin{proof}
	Letting $p_1$ and $p_2$ be the points $x$ and $y$, it is clear that we can find a pin sequence $p_1, p_2, \dots, p_m$ in the plot of $\pi$ such that $\rect(p_1, \dots, p_m) = [1,n] \times [1,n]$ where $[n]$ is the set that~$\pi$ permutes. Letting $p_{i_1}$ denote the point $z$, let $i_2$ be the least index such that $p_1, \dots, p_{i_2}, p_{i_1}$ is a valid pin sequence. We must have that $p_{i_1}$ separates $p_{i_2}$ from $\rect(p_1, \dots, p_{i_2-1})$ since $p_1, \dots, p_{i_2-1}, p_{i_1}$ is not a valid pin sequence. 
	
	Next, find the smallest index $i_3$ such that $p_1, \dots, p_{i_3}, p_{i_2}$ is a valid pin sequence. Again, we have that~$p_{i_2}$ separates $p_{i_3}$ from $\rect(p_1, \dots, p_{i_3-1})$. We continue in this way until we reach $p_{i_{m+1}} = p_2$, and thus we obtain the proper pin sequence \[ p_1 (= x), p_2 (= y), p_{i_m}, \dots, p_{i_2}, p_{i_1} (= z).\] This gives the result. 
\end{proof}

The following is the analog of Proposition~\ref{prop:disconnected_chains}.

\begin{proposition}\label{prop:disconnected_pin_sequences}
	For any proper pin sequence $p_1, p_2, \dots, p_m$ in a permutation $\pi$, either $G_{\pi}[p_1, \dots, p_m]$ is connected, or $p_1 \not\sim p_2$ and for every  $i \in \{3, \dots, m-1 \}$ we have $p_{i+1} \sim p_{i}$ and $p_{i+1} \not\sim p_1, \dots, p_{i-1}$. In this latter case, there exists $k \in \{1,2 \}$ such that $p_k$ is isolated and the pins $\{p_1, \dots, p_m \} \setminus \{ p_k \}$ comprise a component in $G_{\pi}[p_1, \dots, p_m]$
\end{proposition}

We see a pin sequence that induces a disconnected inversion graph in Figure~\ref{fig:disconnected_pin_sequences}, with the entire sequence alternating between down and left pins. The other proper pin sequences inducing disconnected subgraphs are essentially the same, with one of $p_1$ or $p_2$ being isolated and the rest of the pins alternating between down and left pins, or up and right pins, resulting in a path. 

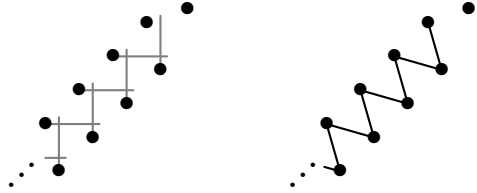
\begin{figure}[h]
\begin{footnotesize}
\begin{center}
	\begin{tikzpicture}[scale=0.9]

		\absdot{(3,2.8)}; 
		\absdot{(2.4,2.6)}; 
		
		\draw [gray, line cap=round] (2.6,1.9)--(2.6,2.7);
		\absdot{(2.6, 1.9)};
		
		\draw [gray, line cap=round] (1.9,2.1)--(2.7,2.1);
		\absdot{(1.9,2.1)};
		
		\draw [gray, line cap=round] (2.1,1.4)--(2.1,2.2);
		\absdot{(2.1, 1.4)};
		
		\draw [gray, line cap=round] (1.4,1.6)--(2.2,1.6);
		\absdot{(1.4,1.6)};
		
		\draw [gray, line cap=round] (1.6,0.9)--(1.6,1.7);
		\absdot{(1.6,0.9)};
		
		\draw [gray, line cap=round] (0.9,1.1)--(1.7,1.1);
		\absdot{(0.9, 1.1)};
		
		\draw [gray, line cap=round] (1.1,0.4)--(1.1,1.2);
		\absdot{(1.1, 0.4)};
		
		\draw [gray, line cap=round] (0.9,0.6)--(1.2,0.6);
		
		\node (x) at (0.7, 0.5) {\Large$.$};
		\node (x) at (0.55, 0.35) {\Large$.$};
		\node (x) at (0.4, 0.2) {\Large$.$};
		
	\end{tikzpicture}
	\hspace{10mm}
	\begin{tikzpicture}[scale=0.9]

		\absdot{(3,2.8)}; 
		\absdot{(2.4,2.6)}; 
		
		\absdot{(2.6, 1.9)};
		
		\absdot{(1.9,2.1)};
		
		\absdot{(2.1, 1.4)};
		
		\absdot{(1.4,1.6)};
		
		\absdot{(1.6,0.9)};
		
		\absdot{(0.9, 1.1)};
		
		\absdot{(1.1, 0.4)};
		
		\draw [thick, line cap=round] (2.4,2.6)--(2.6,1.9)--(1.9,2.1)--(2.1,1.4)--(1.4,1.6)--(1.6,0.9)--(0.9,1.1)--(1.1,0.4) --(0.8667,0.4667) ;
		
		\node (x) at (0.7, 0.5) {\Large$.$};
		\node (x) at (0.55, 0.35) {\Large$.$};
		\node (x) at (0.4, 0.2) {\Large$.$};
		
	\end{tikzpicture}

\end{center}
\end{footnotesize}
\caption{A proper pin sequence that induces a disconnected inversion graph.}
\label{fig:disconnected_pin_sequences}
\end{figure}

We can now see how our proof of Theorem~\ref{thm:gallai} utilizing chains plays out in the case of inversion graphs with proper pin sequences. The following is essentially the same as Lemma~\ref{lemma:p1p2_pspm_same_edge_class}, and thus we omit its proof, except we can now visualize the argument in the plane with Figure~\ref{fig:cases_p_s}. This figure portrays the inductive step in the event of $p_m$ being a right pin, and thus we consider the cases in which $p_{m+1}$ is a down pin and an up pin. The other cases for $p_m$ can be visualized similarly.

\begin{lemma}\label{lemma:p1p2_pspm_same_edge_class_pins}
	Suppose $p_1, p_2, \dots, p_m$ is a proper pin sequence in a permutation $\pi$ such that in~$G_{\pi}$ we have $p_1 \sim p_2$ and $s$ is the greatest index such that $p_s \sim p_{s+1}$. Then the edges $p_1 p_2$ and $p_s p_m$ are in the same edge class in $G_{\pi}$.
\end{lemma}

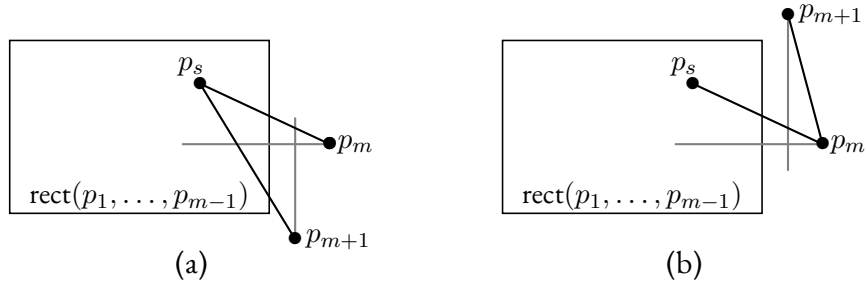
\begin{figure}[h]
\begin{center}
	\begin{tikzpicture}[scale=1.15]
		\draw[line width = 0.6pt] (-1,0) rectangle (2,2);
		\node [] at (0.5, 0.2) {\small$\rect(p_1, \dots, p_{m-1})$};
		
		\node [] at (1.1, 1.7) {\small$p_s$};
		\node [] at (3, 0.8) {\small$p_{m}$};
		\node [] at (2.8, -0.3) {\small$p_{m+1}$};
		
		\absdot{(2.7,0.8)};
		\draw [gray, line cap = round] (1,0.8) to (2.7,0.8);
		\absdot{(2.7,0.8)};
		\draw [gray, line cap = round] (2.3,-0.3) to (2.3,1.1);
		\absdot{(2.3,-0.3)};
		
		\absdot{(1.2,1.5)};
		\draw [] (2.7,0.8) -- (1.2,1.5) -- (2.3,-0.3);
		
		\node [rotate=0] (y) at (1.1, -0.6) {(a)};
		
	\end{tikzpicture}
	\hspace{15mm}
	\begin{tikzpicture}[scale=1.15]
		\draw[line width = 0.6pt] (-1,0) rectangle (2,2);
		\node [] at (0.5, 0.2) {\small$\rect(p_1, \dots, p_{m-1})$};
		
		\node [] at (1.1, 1.7) {\small$p_s$};
		\node [] at (3, 0.8) {\small$p_{m}$};
		\node [] at (2.8, 2.3) {\small$p_{m+1}$};
		
		\draw [gray, line cap = round] (1,0.8) to (2.7,0.8);
		\absdot{(2.7,0.8)};
		\draw [gray, line cap = round] (2.3,0.5) to (2.3,2.3);
		\absdot{(2.3,2.3)};
		
		\absdot{(1.2,1.5)};
		\draw [] (1.2,1.5) -- (2.7,0.8) -- (2.3,2.3);
		
		\node [rotate=0] (y) at (1.1, -0.6) {(b)};
		
	\end{tikzpicture}
\end{center}
\caption{Cases in the omitted proof of Lemma~\ref{lemma:p1p2_pspm_same_edge_class_pins}.}
\label{fig:cases_p_s}
\end{figure}

The following is the analog to Lemma~\ref{lemma:p1p2_pspm_same_edge_class}, except now we can be more specific about which edges must be in the same edge class.

\begin{lemma}\label{lemma:final_pins}
	Suppose $p_1, p_2, \dots, p_m$ is a proper pin sequence in a permutation $\pi$ such that $i<j<k$ are the indices of the pins $p_1$, $p_2$ and $p_m$, respectively, and $\pi(i) > \pi(j) > \pi(k)$. Then the edges $p_1 p_2$ and $p_1 p_m$ are in the same edge class in $G_{\pi}$. 
\end{lemma}

\begin{proof}
	By Lemma~\ref{lemma:p1p2_pspm_same_edge_class_pins}, $p_1 p_2$ and $p_s p_m$ are in the same edge class where $s$ is the greatest index such that $p_s \sim p_{s+1}$. Since $p_{m} \sim p_1, p_2$ by our hypotheses, Proposition~\ref{prop:pins_adjacencies} tells us that $p_m \sim p_1, \dots, p_{m-2}$ and $p_m \not\sim p_{m-1}$. Thus $s \neq {m-1}$. By reversing the permutation induced by the pins $p_1, \dots, p_{m-2}$, or simply flipping that pin sequence over a vertical axis to obtain a new pin sequence $p_1', p_2', \dots, p_{m-2}'$, we obtain a permutation with an inversion graph isomorphic to $\overline{G}_{\pi}[p_1, \dots, p_{m-2}]$, the complement of the the subgraph induced by the pins $p_1, \dots, p_{m-2}$.
	
	If this graph is connected, then there exists a path $p_s, p_{q_1}, \dots, p_{q_r}, p_1$ through $\{ p_1, \dots, p_{m-2} \} \subseteq N_{G_{\pi}}(p_m)$ in the complement of $G_{\pi}$. That is, we have $p_s p_m \wedge p_{q_1} p_m \wedge \dots \wedge p_{q_r} p_m \wedge p_1 p_m$, and thus~$p_1 p_m$ is in the same edge class as $p_1 p_2$ as they are both in the same edge class as $p_s p_m$.
	
	It remains to analyze the cases in which $\overline{G}_{\pi}[p_1, \dots, p_{m-2}]$ is not connected. By Proposition~\ref{prop:disconnected_pin_sequences}, the pin sequence $p_1', p_2', \dots, p_{m-2}'$ is one of the pin sequences described in Figure~\ref{fig:disconnected_pin_sequences}. Hence, the pin sequence $p_1, p_2, \dots, p_m$ must be one of the cases described in Figure~\ref{fig:disconnected_cases}.
	
	\begin{figure}[h]
\begin{footnotesize}
\begin{center}
	\begin{tikzpicture}[scale=1]
	
		
		\node [rotate=0] (x) at (0.1, 2.4) {$p_1$};
		\node [rotate=0] (y) at (1, 2.65) {$p_2$};
		\node [rotate=0] (y) at (2.9, 0.65) {$p_m$};
		\node [rotate=0] (y) at (2.65, 1.17) {$p_{m-2}$};
		\absdot{(0.2,2.6)}; 
		\absdot{(0.8,2.4)}; 
		
		\draw [gray, line cap=round] (0.6,1.7)--(0.6,2.5);
		\absdot{(0.6,1.7)}; 
		
		\draw [gray, line cap=round] (0.5,1.9)--(1.3,1.9);
		\absdot{(1.3,1.9)}; 

		\draw [gray, line cap=round] (1.1,2)--(1.1,1.2);
		\absdot{(1.1,1.2)}; 
		
		\draw [gray, line cap=round] (1,1.4)--(1.8,1.4);
		\absdot{(1.8,1.4)}; 
		
		\draw [gray, line cap=round] (1.6,0.7)--(1.6,1.5);
		\absdot{(1.6,0.7)}; 
		
		\draw [gray, line cap=round] (1.5,0.9)--(2.3,0.9);
		\absdot{(2.3,0.9)}; 
		
		\draw [gray, line cap=round] (2.1,1)--(2.1,0.2);
		\absdot{(2.1,0.2)}; 
		
		\draw [gray, line cap=round] (2,0.4)--(2.8,0.4);
		\absdot{(2.8,0.4)}; 
		
		\node [rotate=0] (y) at (1.5, -0.3) {(a)};
	\end{tikzpicture}
	\hspace{15mm}
	\begin{tikzpicture}[scale=1]
	
		
		\node [rotate=0] (x) at (1.3, 0.9) {$p_1$};
		\node [rotate=0] (y) at (1.9, 0.6) {$p_2$};
		\node [rotate=0] (y) at (2.75, 0.4) {$p_m$};
		\node [rotate=0] (y) at (0.6, 3.05) {$p_{m-2}$};
		\absdot{(1.4,1.1)}; 
		\absdot{(2,0.8)}; 
		
		\draw [gray, line cap=round] (1.6,1)--(1.6,1.8);
		\absdot{(1.6,1.8)}; 
		
		\draw [gray, line cap=round] (0.9,1.6)--(1.7,1.6);
		\absdot{(0.9,1.6)}; 
		
		\draw [gray, line cap=round] (1.1,1.5)--(1.1,2.3);
		\absdot{(1.1,2.3)}; 
		
		\draw [gray, line cap=round] (1.2,2.1)--(0.4,2.1);
		\absdot{(0.4,2.1)}; 
		
		\draw [gray, line cap=round] (0.6,2)--(0.6,2.8);
		\absdot{(0.6,2.8)}; 
		
		\draw [gray, line cap=round] (0.5,2.6)--(2.6,2.6);
		\absdot{(2.6,2.6)}; 
		
		\draw [gray, line cap=round] (2.4,2.7)--(2.4,0.4);
		\absdot{(2.4,0.4)}; 
		
		\node [rotate=0] (y) at (1.5,-0.1) {(b)};
	\end{tikzpicture}

\end{center}
\end{footnotesize}
\caption{Cases in the proof of Lemma~\ref{lemma:final_pins}.}
\label{fig:disconnected_cases}
\end{figure}
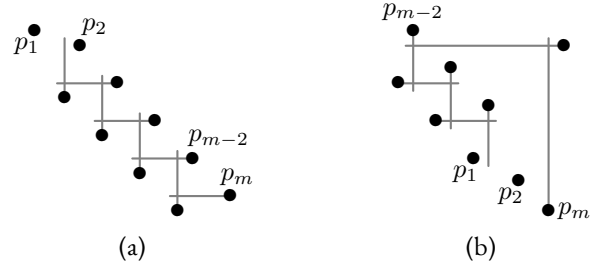 
	
	In Figure~\ref{fig:disconnected_cases}(a), we have that the entire pin sequence $p_1, p_2, \dots, p_m$ alternates between right and down pins. That is, $p_s = p_1$ and we are done.
	
	In Figure~\ref{fig:disconnected_cases}(b), we have that the pin sequence $p_1, p_2, \dots, p_{m-2}$ alternates between left and up pins, and then $p_{m-1}$ and $p_{m}$ are a right and a down pin, (not necessarily in that order). The figure reveals that $p_s = p_{m-2}$, and furthermore that $p_{m-2}, p_{m-3}, \dots, p_3, p_1 (\in N_{G_{\pi}}(p_m))$ is a path through the complement of $G_{\pi}$. That is, $p_{m-2} p_m \wedge p_{m-3} p_m \wedge \dots \wedge p_3 p_m \wedge p_1 p_m$, and therefore $p_1 p_2$ and $p_1 p_m$ are in the same edge class as $p_{m-2} p_m$, and thus the same edge class as each other. This gives the result.
\end{proof}

With this result, we can also be more precise with finishing the proof. We show that any two incident edges in the inversion graph $G_{\pi}$ of a simple permutation $\pi$ are in the same edge class, and again, there is nothing to be done in the case that two incident edges induce a $P_3$. So we consider any three indices $i<j<k$ such that $\pi(i) > \pi(j) > \pi(k)$, thus inducing a $K_3$. We know that we can find a proper pin sequence starting with $\pi(i)$ and $\pi(j)$ and reaching $\pi(k)$ by Proposition~\ref{prop:BHV_k-reaching}, and thus we have that $\pi(i) \pi(j)$ is in the same edge class as $\pi(i) \pi(k)$ by Lemma~\ref{lemma:final_pins}. To finish, we simply take the `dual' of this argument, showing that $\pi(j) \pi(k)$ is in the same edge class as $\pi(i) \pi(k)$, and thus all three of these edges are in the same edge class by transitivity. We are then done by observing that $G_{\pi}$ must be connected.

\section{Implications for Permutations}\label{sec:permutations}

\subsection{The Uniqueness of Inversion Graphs}

We have now seen how these results can be understood geometrically in the permutation setting. However, there is much more that can be said about permutations. The following corollary to Theorem~\ref{thm:gallai} is very important to the theory of permutations. Since it is not immediate, we show how it follows from Gallai's result in the subsequent discussion.

\begin{theorem}[Gallai~\cite{gallai:transitiv-orien:}]\label{thm:simple_perms_are_unique}
	If $\pi$ is a simple permutation, then there are $1$, $2$ or $4$ permutations~$\sigma$ such that $G_{\pi} \cong G_{\sigma}$. Moreover, the four possible permutations $\sigma$ such that $G_{\pi} \cong G_{\sigma}$ are $\pi$, $\pi^{-1}$, $\pi^{\rc}$, and $(\pi^{\rc})^{-1}$, where $\pi^{-1}$ is the inverse of $\pi$, and $\pi^{\rc}$ is its reverse complement. 
\end{theorem}

That is, $\pi^{-1}$ is obtained from $\pi$ by reflecting the plot of $\pi$ over the line $y=x$, $(\pi^{\rc})^{-1}$ is obtained by reflecting it over the line $y = n+1-x$, and $\pi^{\rc}$ is obtained by reflecting it over both. Indeed, since the reflections over these lines both preserve the sign of the slope of any line with nonzero, finite slope, it follows that all of these permutations have the same inversion graph. For example, in Figure~\ref{fig:PermutationSymmetries} we see the four symmetries of the simple permutation given in Figure~\ref{fig:NonSimpleAndSimplePermutations}(b). That is, Theorem~\ref{thm:simple_perms_are_unique} guarantees that these are the only permutations with this inversion graph.

\begin{figure}[h]
\begin{center}
	\begin{tikzpicture}[scale=0.25]

		\plotpermgraph{2,4,1,9,3,5,8,6,10,7};
		\draw [in = 135, out = 0] (4,9) to (10,7);
		
		\node [] at (5.5,-1) {$\pi$};
		
	\end{tikzpicture}
	\hspace{10mm}
	\begin{tikzpicture}[scale=0.25]
		\draw [color = gray!60, line width = 0.5mm, dotted] (0.75,0.75) to (10.25,10.25);
		\plotpermgraph{3,1,5,2,6,8,10,7,4,9};
		\draw [in = -45, out = 90] (9,4) to (7,10);
		
		\node [] at (5.5,-1) {$\pi^{-1}$};
		
	\end{tikzpicture}
	\hspace{10mm}
	\begin{tikzpicture}[scale=0.25]
		\draw [color = gray!60, line width = 0.5mm, dotted] (0.75,10.25) to (10.25,0.75);
		\draw [color = gray!60, line width = 0.5mm, dotted] (0.75,0.75) to (10.25,10.25);
		\plotpermgraph{4,1,5,3,6,8,2,10,7,9};
		\draw [in = -45, out = 180] (7,2) to (1,4);
		
		\node [] at (5.5,-1) {$\pi^{\rc}$};
		
	\end{tikzpicture}
	\hspace{10mm}
	\begin{tikzpicture}[scale=0.25]
		\draw [color = gray!60, line width = 0.5mm, dotted] (0.75,10.25) to (10.25,0.75);
		\plotpermgraph{2,7,4,1,3,5,9,6,10,8};
		\draw [in = 135, out = -90] (2,7) to (4,1);
		
		\node [] at (5.5,-1) {$(\pi^{\rc})^{-1}$};
		
	\end{tikzpicture}
\end{center}
\caption{A simple permutation and its four symmetries; the three permutations on the right are obtained by reflecting $\pi$ over the dashed line(s).}
\label{fig:PermutationSymmetries}
\end{figure}
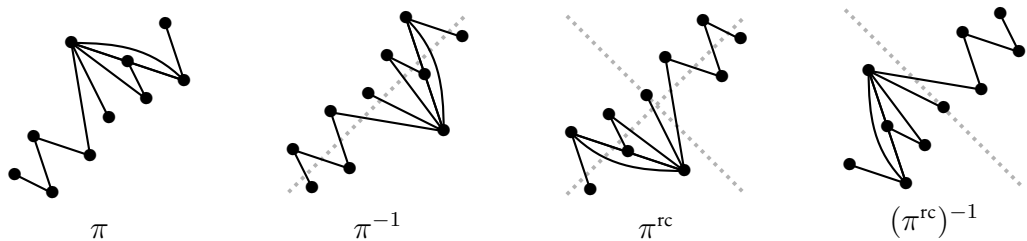

Furthermore, if a simple permutation is equal to exactly one of its inverse or reverse complement, then there are exactly two unique permutations with its inversion graph, and if it's equal to both, then it is the unique permutation with that inversion graph. For example, $3142$ is simple and equal to its reverse complement, but is not equal to its inverse, and thus there is exactly one other permutation with an isomorphic inversion graph: $2413$, its inverse.

 We are now ready to see how Theorem~\ref{thm:gallai} implies Theorem~\ref{thm:simple_perms_are_unique}. We begin by recalling from Theorem~\ref{thm:dimension_def} that inversion graphs are the comparability graphs of posets with order dimension~2. Throughout this discussion, we keep in mind Figure~\ref{fig:PosetFromPerm}, which captures the notion of giving a transitive orientation to an inversion graph drawn on the plot of its permutation by directing all of the edges to be arcs pointing towards the northwest.
 
So, we let $G_{\pi}$ be the inversion graph of a simple permutation $\pi$, and then show that the only permutations we can possibly recover from $G_{\pi}$ are $\pi$, $\pi^{-1}$, $\pi^{\rc}$, and $(\pi^{\rc})^{-1}$. By Proposition~\ref{prop:simple_is_prime} we have that $G_{\pi}$ is prime, and by Theorem~\ref{thm:dimension_def} it is transitively orientable. Since it is prime, it has one edge class and exactly two transitive orientations by Theorem~\ref{thm:gallai}. So, we begin recovering a permutation with this inversion graph by selecting one of the two possible orientations of its edges, and just as in Figure~\ref{fig:PosetFromPerm}, for each vertex, we interpret its out-neighborhood to be the entries above and to its left, and its in-neighborhood as the vertices below and to its right in the plot.

However, for each vertex we need also know which entries are above and to its right, or below and to its left, in order to recover a permutation. Fortunately, we recall Theorem~\ref{thm:comp_cocomp_def}, which informs us that the complement of $G_{\pi}$ is also transitively orientable. Indeed, we have that $\overline{G}_{\pi}$ is isomorphic to $G_{\pi^{\text{r}}}$, where $\pi^{\text{r}}$ is the reverse of $\pi$ and is also a simple permutation. Thus, applying Theorem~\ref{thm:simple_perms_are_unique} again we know that $\overline{G}_{\pi}$ has one edge class, and we can choose one of the two possible orientations of its edges, letting the out-neighborhood of each vertex determine the entries above and to its right, and its in-neighborhood determine the entries below and to its left in the plot. See Figure~\ref{fig:PermsFromAnInvGraph} for an example.

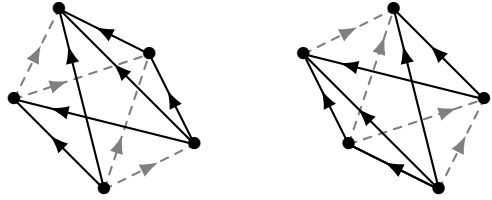
\begin{figure}[h]
\begin{center}
	\begin{tikzpicture}[scale=0.6]
		\draw [->-, line cap=round, dashed, color = gray] (1,3)--(2,5);
		\draw [->--, line cap=round, dashed, color = gray] (1,3)--(4,4);
		\draw [->--, line cap=round, dashed, color = gray] (3,1)--(4,4);
		\draw [->-, line cap=round, dashed, color = gray] (3,1)--(5,2);
		
		\draw [->-, line cap=round] (3,1)--(1,3);
		\draw [-->-, line cap=round] (5,2)--(1,3);
		\draw [-->-, line cap=round] (3,1)--(2,5);
		\draw [->-, line cap=round] (5,2)--(4,4);
		\draw [->-, line cap=round] (4,4)--(2,5);
		\draw [->-, line cap=round] (5,2)--(2,5);
		
		\plotperm{3,5,1,4,2};
		
	\end{tikzpicture}
	\hspace{10mm}
	\begin{tikzpicture}[scale=0.6]
		\draw [->-, line cap=round, dashed, color = gray] (1,4)--(3,5);
		\draw [-->-, line cap=round, dashed, color = gray] (2,2)--(3,5);	
		\draw [-->-, line cap=round, dashed, color = gray] (2,2)--(5,3);
		\draw [->-, line cap=round, dashed, color = gray] (4,1)--(5,3);		
	
		\draw [->-, line cap=round] (4,1)--(2,2);
		\draw [->-, line cap=round] (4,1)--(1,4);
		\draw [-->-, line cap=round] (4,1)--(3,5);
		\draw [-->-, line cap=round] (5,3)--(1,4);
		\draw [->-, line cap=round] (5,3)--(3,5);
		\draw [->-, line cap=round] (4,1)--(2,2);
		\draw [->-, line cap=round] (2,2)--(1,4);
		
		\plotperm{4,2,5,1,3};
		
	\end{tikzpicture}
\end{center}
\caption{Obtaining different permutations from the inversion graph of a simple permutation.}
\label{fig:PermsFromAnInvGraph}
\end{figure}

That is, given these choices, it is easy to recover the resulting permutation. The index of the entry corresponding to each vertex in the permutation can be recovered by counting the number of vertices that are to its `left': its in-degree in $\overline{G}_{\pi}$ plus its out-degree in~$G_{\pi}$. Furthermore, the value of the entry corresponding to each vertex is recovered by counting the number of vertices that are `below' it: its in-degree in $\overline{G}_{\pi}$ plus its in-degree in $G_{\pi}$. Certainly, there exists a choice for the edges of $G_{\pi}$ and for $\overline{G}_{\pi}$ such that we would recover the permutation~$\pi$. If we made the opposite choice for the edges of $G_{\pi}$, each arc would be reversed and it would be equivalent to reflecting the plot of the permutation over the line $y=x$, thus we would recover $\pi^{-1}$. Similarly, if we also made the opposite choice for the edges of $\overline{G}_{\pi}$, we would recover~$\pi^{\rc}$. 

A permutation with nontrivial intervals can have an inversion graph that is isomorphic to that of many other permutations. To see this, we observe that applying reflections `to' any interval preserves the inversion graph. For example, the only interval in the nonsimple permutation given in Figure~\ref{fig:NonSimpleAndSimplePermutations}(a) contains four entries that are order isomorphic to $3142$. As we mentioned above, its inverse is $2413$, and thus interchanging these permutations `at' that interval does not alter the inversion graph. We see this operation in Figure~\ref{fig:ElemMove}.

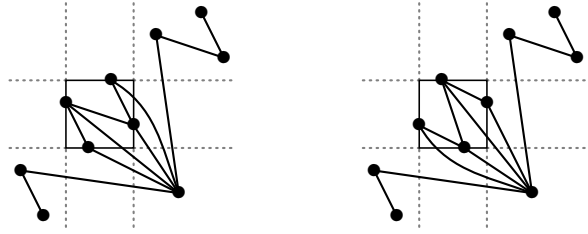
\begin{figure}[h]
\begin{center}
	\begin{tikzpicture}[scale=0.3]
		\draw [color = gray, line width = 0.6pt, dotted, thick, line cap=round] (0.5,4)--(3,4);
		\draw [color = gray , line width = 0.6pt, dotted, thick, line cap=round] (6,4)--(10.5,4);
		\draw [color = gray , line width = 0.6pt, dotted, thick, line cap=round] (0.5,7)--(3,7);
		\draw [color = gray , line width = 0.6pt, dotted, thick, line cap=round] (6,7)--(10.5,7);
		\draw [color = gray , line width = 0.6pt, dotted, thick, line cap=round] (3,0.5)--(3,4);
		\draw [color = gray , line width = 0.6pt, dotted, thick, line cap=round] (3,7)--(3,10.5);
		\draw [color = gray , line width = 0.6pt, dotted, thick, line cap=round] (6,0.5)--(6,4);
		\draw [color = gray , line width = 0.6pt, dotted, thick, line cap=round] (6,7)--(6,10.5);
		\draw [line width = 0.6pt] (3,4) rectangle (6,7);

		\plotperm{3,1,6,4,7,5,9,2,10,8};
		
		\draw [line cap=round] (1,3) -- (2,1);
		\draw [line cap=round] (1,3) -- (8,2);
		
		\draw [line cap=round] (3,6) -- (8,2);
		\draw [line cap=round] (4,4) -- (8,2);
		\draw [line cap=round] (6,5) -- (8,2);
		\draw [in = 110, out = -30] (5,7) to (8,2);
		\draw [line cap=round] (7,9) -- (8,2);
		\draw [line cap=round] (7,9) -- (10,8);
		\draw [line cap=round] (9,10) -- (10,8);
		
		\draw [line cap=round] (3,6) -- (4,4);
		\draw [line cap=round] (3,6) -- (6,5);
		\draw [line cap=round] (5,7) -- (6,5);
		
	\end{tikzpicture}
	\hspace{15mm}
	\begin{tikzpicture}[scale=0.3]
		\draw [color = gray, line width = 0.6pt, dotted, thick, line cap=round] (0.5,4)--(3,4);
		\draw [color = gray , line width = 0.6pt, dotted, thick, line cap=round] (6,4)--(10.5,4);
		\draw [color = gray , line width = 0.6pt, dotted, thick, line cap=round] (0.5,7)--(3,7);
		\draw [color = gray , line width = 0.6pt, dotted, thick, line cap=round] (6,7)--(10.5,7);
		\draw [color = gray , line width = 0.6pt, dotted, thick, line cap=round] (3,0.5)--(3,4);
		\draw [color = gray , line width = 0.6pt, dotted, thick, line cap=round] (3,7)--(3,10.5);
		\draw [color = gray , line width = 0.6pt, dotted, thick, line cap=round] (6,0.5)--(6,4);
		\draw [color = gray , line width = 0.6pt, dotted, thick, line cap=round] (6,7)--(6,10.5);
		\draw [line width = 0.6pt] (3,4) rectangle (6,7);

		\plotperm{3,1, 5,7,4,6 ,9,2,10,8};
		
		\draw [line cap=round] (1,3) -- (2,1);
		\draw [line cap=round] (1,3) -- (8,2);
		
		\draw [line cap=round] (5,4) -- (8,2);
		\draw [line cap=round] (4,7) -- (8,2);
		\draw [line cap=round] (6,6) -- (8,2);
		\draw [in = 160, out = -60] (3,5) to (8,2);
		\draw [line cap=round] (7,9) -- (8,2);
		\draw [line cap=round] (7,9) -- (10,8);
		\draw [line cap=round] (9,10) -- (10,8);
		
		\draw [line cap=round] (4,7) -- (6,6);
		\draw [line cap=round] (4,7) -- (5,4);
		\draw [line cap=round] (3,5) -- (5,4);
		
	\end{tikzpicture}
\end{center}
\caption{Two permutations with the same inversion graph that can't be obtained from each other by reflections of the entire plot.}
\label{fig:ElemMove}
\end{figure}

Although we don't go into much detail, it follows from Theorem~\ref{thm:simple_perms_are_unique} and some other results in Gallai's paper~\cite{gallai:transitiv-orien:} regarding modular decomposition that any two permutations with the same inversion graph can be obtained from each other by applying a series of these reflections to their intervals. For example, since the permutations in Figure~\ref{fig:ElemMove} have only that one nontrivial interval, it follows that there are exactly eight permutations with this inversion graph, (neither permutation in Figure~\ref{fig:ElemMove} is equal to its inverse or reverse complement).

\subsection{The Automorphism Group of Inversion Graphs}

We see next that it also follows from Theorem~\ref{thm:gallai} that the automorphisms of the inversion graph of a simple permutation are completely determined by its symmetries. We outline the argument given by Klav\'ik and Zeman~\cite{klavik:automorphism-gr:}. To see this, we let $\pi$ be a simple permutation and $\TO(G_{\pi})$ the set of transitive orientations of its inversion graph. Then for some $\tau \in \Aut(G_{\pi})$ and $\rightarrow \in \TO(G_{\pi})$, we define~$\tau(\rightarrow)$ by 
\[
x \rightarrow y \implies \tau(x) \tau (\rightarrow) \tau (y) \textnormal{ for all } x,y \in V(G_{\pi}),
\]
and it follows that $\tau(\rightarrow) \in \TO(G_{\pi})$. Hence, $\Aut(G_{\pi})$ induces an action on $\TO(G_{\pi})$, and similarly on~$\TO(\overline{G}_{\pi})$. 

If $\tau$ is in the stabilizer of $\rightarrow \in \TO(G_{\pi})$, then it can only interchange nonadjacent vertices of~$G_{\pi}$. This is clear since if, for example, a vertex $x$ is mapped to some vertex $y$ such that $x \rightarrow y$, then $y$ must map to some vertex $z$ such that $y \rightarrow z$, etc. This is not possible since $\rightarrow$ is a transitive orientation of a finite graph.

So, we define $\TO(G_{\pi}, \overline{G}_{\pi}) = \TO(G_{\pi}) \times \TO(\overline{G}_{\pi})$, and suppose that $\tau \in \Aut(G_{\pi})$ is in the stabilizer of $(\rightarrow, \overline{\rightarrow}) \in \TO(G_{\pi}, \overline{G}_{\pi})$. Therefore, from above we have that $\tau$ cannot interchange any adjacent vertices in $G_{\pi}$, but also cannot interchange any adjacent vertices in $\overline{G}_{\pi}$. This can only happen if~$\tau$ is the identity automorphism, and thus the action of $\Aut(G_{\pi})$ on $\TO(G_{\pi}, \overline{G}_{\pi})$ is \emph{free}---the stabilizer of each element contains only the identity.

Thus, if we fix an element $(\rightarrow, \overline{\rightarrow}) \in \TO(G_{\pi}, \overline{G}_{\pi})$, it follows that each element of $\Aut(G_{\pi})$ applied to $(\rightarrow, \overline{\rightarrow})$ determines a unique element of $\TO(G_{\pi}, \overline{G}_{\pi})$. Furthermore, it follows from Theorem~\ref{thm:gallai} that $\TO(G_{\pi}) = \{ \rightarrow, \leftarrow \}$ where $\leftarrow$ is obtained by reversing each arc in $\rightarrow$, and similarly, $\TO(\overline{G}_{\pi}) = \{ \overline{\rightarrow}, \overline{\leftarrow} \}$. We have $\pi = \pi^{-1}$ if and only if there exists $\varphi \in \Aut(G_{\pi})$ such that $\varphi(\rightarrow, \overline{\rightarrow}) = (\leftarrow, \overline{\rightarrow})$, and~$\pi = (\pi^{\rc})^{-1}$ if and only if there exists $\psi \in \Aut(G_{\pi})$ such that $\psi(\rightarrow, \overline{\rightarrow}) = (\rightarrow, \overline{\leftarrow})$. If $\pi$ has both of these symmetries, then we have $\psi \circ \varphi(\rightarrow, \overline{\rightarrow}) = (\leftarrow, \overline{\leftarrow})$. Since these are all the elements of $\TO(G_{\pi}, \overline{G}_{\pi})$, we have considered all of the possible elements of $\Aut(G_{\pi})$, and thus we have the following.

\begin{theorem}[Klav\'ik and Zeman~\text{\cite[Lemma 6.6]{klavik:automorphism-gr:}}]
\label{prop:prime:inv:symmetries}
	If $\pi$ is a simple permutation, then the automorphism group of $G_{\pi}$ is a subgroup of $\mathbb{Z}_2 \times \mathbb{Z}_2$. That is, every automorphism of $G_{\pi}$ corresponds to a symmetry of $\pi$, in particular, one of $\pi$, $\pi^{-1}$, $\pi^{\rc}$, or~$(\pi^{\rc})^{-1}$.
\end{theorem}

Note, as a corollary to Proposition~\ref{prop:prime:inv:symmetries}, we see immediately that $C_n$ for $n\ge 5$ and $T_2$ (see Figure~\ref{fig:T_2}), are not inversion graphs. This is because these graphs are prime, and if they were inversion graphs, they would correspond to simple permutations by Proposition~\ref{prop:simple_is_prime}, and thus their automorphism groups would be subgroups of $\mathbb{Z}_2\times \mathbb{Z}_2$. However, their automorphism groups are larger dihedral groups, ($\Aut(T_2) = D_3$ and $\Aut(C_n) = D_n$ for all $n \geq 3$).

\begin{figure}[h]
\begin{center}
	\begin{tikzpicture}[scale=0.5]
		\draw [] (0,0) -- (0, 2);
		\draw [] (0,0) -- (1.732, -1);
		\draw [] (0,0) -- (-1.732, -1);
		
		\absdot{(0,0)};
		\absdot{(0,1)};
		\absdot{(0,2)};
		
		\absdot{(0.866, -0.5)};
		\absdot{(1.732, -1)};
		
		\absdot{(-0.866, -0.5)};
		\absdot{(-1.732, -1)};
		
	\end{tikzpicture}
\end{center}
\caption{The graph $T_2$.}
\label{fig:T_2}
\end{figure}

We mention here a consequence of this observation that inversion graphs contain no induced cycles of length at least 5, (nor their complements). A major area of research in the field of graph theory is the study of perfect graphs~\cite{golumbic:algorithmic-gra:}. First defined by Claude Berge in 1961~\cite{berge:f:arbung-von-gr:}, a \emph{perfect graph} is one such that in every induced subgraph, the chromatic number is equal to the size of the largest maximal clique. In 2006, Chudnovsky, Robertson, Seymour and Thomas proved the strong perfect graph theorem~\cite{chudnovsky:the-strong-perf:}, stating that a graph is perfect if and only if it contains no induced cycles of odd length greater than or equal to~5, nor their complements. Thus, we have shown the following, which was first observed by Berge in the 1960's.

\begin{theorem}[Berge~\cite{berge:some-classes-of:}]
	Inversion graphs are perfect.
\end{theorem}

%% file: chap-lettericity.tex
\chapter{Graph Encoding}
\label{chap:lettericity}

\def\nodesize{3.5}

In this dissertation, we explore objects at the intersection of permutation theory and graph theory---namely, inversion graphs. In this chapter, we examine two notions that originated independently in these `separate' fields and were later discovered to be intimately related via inversion graphs. These notions are the geometric grid classes of permutations, and letter graphs from structural graph theory. Once we define these notions, we discuss this connection, answer some extremal questions, and then introduce and investigate a generalization of the letter graph construction.

\section{Letter Graphs and Geometric Grid Classes}

\subsection{Letter Graphs and Lettericity}

We begin with letter graphs and the associated invariant lettericity. For a finite alphabet~$\Sigma$, we consider a set of ordered pairs $D \subseteq \Sigma^2$ which we refer to as a \textit{decoder}. Then for a word $w$ with letters $w(1)$, $\dots$, $w(n)\in\Sigma$, we define the \textit{letter graph} of~$w$ with respect to~$D$ to be the graph $\Gamma_D(w)$ with the vertices $\{1, \dots,n \}$ and the edges $ij$ for all $i<j$ with $(w(i), w(j)) \in D$.  If $|\Sigma| = k$ then we say that $\Gamma_D(w)$ is a $k$-letter graph. Finally, for any graph~$G$, the least integer $k$ such that~$G$ is (isomorphic to) a $k$-letter graph is called the \textit{lettericity} of~$G$, denoted by $\ell (G)$.

We include some additional terminology here that will aid in the subsequent discussions. A word~$w$ is called a \textit{lettering} of a graph~$G$ if $\Gamma_D(w) = G$ for some decoder~$D$. We further say that each letter~$a\in\Sigma$ \textit{encodes} the vertices corresponding to the instances of $a$ in the word $w$. More precisely,~$a$ encodes the set $\{1\le i\le n: w(i) = a \} \subseteq V(\Gamma_D(w))$. The set of vertices encoded by a given letter $a\in\Sigma$ must either form a clique, (if $(a,a)\in D$), or an anticlique, (if $(a,a)\notin D$). Thus, letterings of graphs are special types of cocolorings, (a concept introduced by Lesniak-Foster and Straight~\cite{lesniak-foster:the-cochromatic:}), and the lettericity of a graph is bounded below by its cochromatic number. However, as we will see, lettericity is typically much greater than cochromatic number.

A notable example of a class of graphs with lettericity 2 is the class of threshold graphs~\cite{mahadev:threshold-graph:}. These graphs can be defined in various ways, but for our purposes, the most useful definition is as follows: a \emph{threshold graph} is constructed by iteratively adding either dominating vertices, (adjacent to all previously added vertices), or isolated vertices, (adjacent to none of the previously added vertices). Thus, threshold graphs are precisely the letter graphs on the alphabet $\Sigma=\{a,b\}$ with the decoder $D=\{(a,b),(b,b)\}$. To see this, simply encode vertices based on their order of addition to the graph, using~$a$ for isolated vertices and~$b$ for dominating vertices.

We remark that lettericity was first introduced by Petkov\v{s}ek~\cite{petkovsek:letter-graphs-a:} to investigate well-quasi-order~\cite{cherlin:forbidden-subst:, huczynska:well-quasi-orde:} in the induced subgraph order. Indeed, lettericity is a monotone parameter with respect to the induced subgraph relation since any induced subgraph of a letter graph~$\Gamma_D(w)$ can be obtained as~$\Gamma_D(w')$ for some not necessarily contiguous subword $w'$ of $w$. In other words,~$\ell(H) \leq \ell(G)$ if~$H$ is an induced subgraph of $G$. 

Recently, there has been significant interest in the lettericity parameter, both for its own sake~\mbox{\cite{%
	alecu:understanding-l:,
	ferguson:letter-graphs-a:,
	ferguson:on-the-letteric:,
	alecu:the-micro-world:,
	alecu:lettericity-of-:%
}},
and in its connections~\cite{%
	alecu:letter-graphs-a:iff,
	alecu:letter-graphs-a:,
	alecu:letter-graphs-a:abstract%
}
to geometric grid classes of permutations~\cite{%
	albert:geometric-grid-:,
	bevan:growth-rates-of:geom,
	albert:inflations-of-g:,
	elizalde:schur-positive-:%
}. In the rest of this section, we define geometric grid classes and then discuss the connection between these notions.

\subsection{Geometric Grid Classes}

The general idea of a geometric grid class is a set of permutations whose plots can be drawn in the plane on a specified shape. We begin with a $0/\pm 1$ matrix $M$ that specifies how to divide the plane into \emph{cells}. That is, the \emph{standard figure} of the matrix $M = M(i,j)$, with its entries indexed by cartesian coordinates, is the set of points in $\mathbb{R}^2$ consisting of 
\begin{itemize}
	\item the increasing line segment from $(i-1, j-1)$ to $(i,j)$ if $M(i,j) = 1$, or
	\item the decreasing line segment from $(i-1, j)$ to $(i, j-1)$ if $M(i,j) = -1$.
\end{itemize}

For example, we see the standard figure of $\left( \begin{smallmatrix} 0 & -1 & 1 \\ 1 & 0 & -1 \end{smallmatrix} \right)$ in Figure~\ref{fig:standard_figure}.

\begin{figure}[h]
\begin{center}
	\begin{tikzpicture}[scale=0.75]
	
		\draw[step=1cm,very thin] (0,0) grid (3,2);
		
		\draw[line cap = round] (0,0) to (1,1);
		\draw[line cap = round] (1,2) to (2,1);
		\draw[line cap = round] (3,2) to (2,1);
		\draw[line cap = round] (2,1) to (3,0);
				
	\end{tikzpicture}
\end{center}
\caption{The standard figure of a $0 / \pm 1$ matrix.}
\label{fig:standard_figure}
\end{figure}
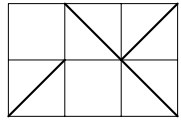

Next, we see that we can draw the permutation $13 86 5472$ on this standard figure in Figure~\ref{fig:perm_on_sf} so that each pair of entries are in the same relative positions as in its plot. We call such a drawing of a permutation $\pi$ on the standard figure of a $0/ \pm 1$ matrix $M$, (which is simply a subset of $\mathbb{R}^2$), an \emph{$M$-drawing} of $\pi$. As evidenced by the two subfigures, if there exists an $M$-drawing of a permutation, then it is not unique.

\begin{figure}[h]
\begin{center}
	\def\nodesize{2.5}
	\begin{tikzpicture}[scale=1.3]
	
		\draw[step=1cm,very thin] (0,0) grid (3,2);
		
		\draw[line cap = round] (0,0) to (1,1);
		\draw[line cap = round] (1,2) to (2,1);
		\draw[line cap = round] (3,2) to (2,1);
		\draw[line cap = round] (2,1) to (3,0);
		
		\node [draw, circle, fill, minimum size = \nodesize, label = \tiny$1$] at (0.2,0.2) {};
		\node [draw, circle, fill, minimum size = \nodesize, label = \tiny$3$] at (0.65,0.65) {};
		
		\node [draw, circle, fill, minimum size = \nodesize, label = \tiny$8$] at (1.28,1.72) {};
		\node [draw, circle, fill, minimum size = \nodesize, label = \tiny$6$] at (1.6,1.4) {};
		\node [draw, circle, fill, minimum size = \nodesize, label = \tiny$5$] at (1.85,1.15) {};
		
		\node [draw, circle, fill, minimum size = \nodesize, label = \tiny$4$] at (2.25,0.75) {};
		\node [draw, circle, fill, minimum size = \nodesize, label = \tiny$7$] at (2.62,1.62) {};
		
		\node [draw, circle, fill, minimum size = \nodesize, label = \tiny$2$] at (2.7,0.3) {};

	\end{tikzpicture}
	\hspace{10mm}
	\begin{tikzpicture}[scale=1.3]
	
		\draw[step=1cm,very thin] (0,0) grid (3,2);
		
		\draw[line cap = round] (0,0) to (1,1);
		\draw[line cap = round] (1,2) to (2,1);
		\draw[line cap = round] (3,2) to (2,1);
		\draw[line cap = round] (2,1) to (3,0);
		
		\node [draw, circle, fill, minimum size = \nodesize, label = \tiny$1$] at (0.1,0.1) {};
		\node [draw, circle, fill, minimum size = \nodesize, label = \tiny$3$] at (0.5,0.5) {};
		
		\node [draw, circle, fill, minimum size = \nodesize, label = \tiny$8$] at (1.3,1.7) {};
		
		\node [draw, circle, fill, minimum size = \nodesize, label = \tiny$6$] at (2.1,1.1) {};
		
		\node [draw, circle, fill, minimum size = \nodesize, label = \tiny$5$] at (2.2,0.8) {};
		\node [draw, circle, fill, minimum size = \nodesize, label = \tiny$4$] at (2.35,0.65) {};
		
		\node [draw, circle, fill, minimum size = \nodesize, label = \tiny$7$] at (2.6,1.6) {};
		
		\node [draw, circle, fill, minimum size = \nodesize, label = \tiny$2$] at (2.75,0.25) {};
				
	\end{tikzpicture}
\end{center}
\caption{The permutation $13 86 5472$ drawn on a standard figure in two different ways.}
\label{fig:perm_on_sf}
\end{figure}
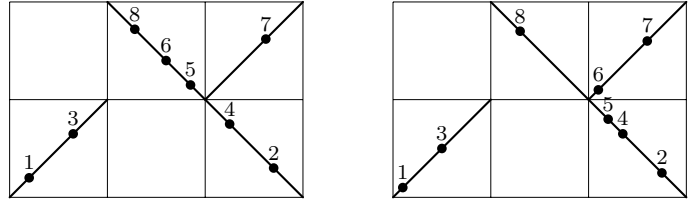

The \emph{geometric grid class} of $M$ is then the set of permutations for which there exists an $M$-drawing. We denote this geometric grid class by $\Geom(M)$. Further, we say that a class $\mathcal{C}$ of permutations is \emph{geometrically griddable} if there exists a $0/ \pm 1$ matrix $M$ such that $\mathcal{C} \subseteq \Geom(M)$. We will define more specifically what a class is below.

It turns out that threshold graphs are inversion graphs, and that they all belong to the class $\Geom \left( \begin{smallmatrix} 1 \\ -1 \end{smallmatrix} \right)$. We can see this with the example given in Figure~\ref{fig:threshold_perm}, where we can read the entries from left to right and we see that any entry in the upper cell can be interpreted as adding an isolated vertex, and any entry in the lower cell can be interpreted as adding a dominating vertex. If we return to how we encoded threshold graphs above with the decoder $\{(a,b), (b,b)\}$, (with $a$ encoding isolated vertices and $b$ encoding dominating vertices), the inversion graph of the permutation in this figure can be encoded with the word $abaabb$.

\begin{figure}[h]
\begin{center}
	\def\nodesize{2.5}
	\begin{tikzpicture}[scale=1.4]
	
		\draw[step=1cm,very thin] (0,0) grid (1,2);
		
		\draw[line cap = round] (0,1) to (1,2);
		\draw[line cap = round] (0,1) to (1,0);
		
		\node [draw, circle, fill, minimum size = \nodesize, label = \tiny$4$] at (0.1,1.1) {};
		\node [draw, circle, fill, minimum size = \nodesize, label = \tiny$3$] at (0.25,0.75) {};
		\node [draw, circle, fill, minimum size = \nodesize, label = \tiny$5$] at (0.4,1.4) {};
		\node [draw, circle, fill, minimum size = \nodesize, label = \tiny$6$] at (0.55,1.55) {};
		\node [draw, circle, fill, minimum size = \nodesize, label = \tiny$2$] at (0.7,0.3) {};	
		\node [draw, circle, fill, minimum size = \nodesize, label = \tiny$1$] at (0.9,0.1) {};		
				
	\end{tikzpicture}
\end{center}
\caption{A permutation drawn on the standard figure of $\left( \begin{smallmatrix} 1 & -1 \end{smallmatrix} \right)^T$.}
\label{fig:threshold_perm}
\end{figure}

As we have just seen that threshold graphs can be obtained from the inversion graphs of permutations drawn on the shape $<$, it is also not too difficult to see that they can also be obtained from the permutations drawn on $\wedge$, $>$, or $\vee$. In any of these cases, we can encode their inversion graphs as $2$-letter graphs, with each of the vertices drawn in the same cell being encoded by the same letter. We generalize this idea below.

\subsection{The Connection}

Common to both permutations and graphs is the notion of \emph{class} or \emph{hereditary property}. To understand these notions, we recall the notion of induced subgraph and define the analogous concept of pattern containment in permutations. That is, if we have two permutations $\pi$ and $\sigma$, then we say that $\pi$ \emph{contains} $\sigma$ as a pattern if some subsequence of the entries of $\pi$ are in the same relative order as the entries of $\sigma$. Otherwise, we say that $\pi$ \emph{avoids} $\sigma$. For example, the permutation $25134$ contains the pattern $132$ given by the substring $253$ or $254$, however it avoids the pattern $321$. Furthermore, we recall that a subgraph $H$ of a graph $G$ is an induced subgraph if $E(H) = V(H)^2 \cap E(G)$. That is, any two vertices in $H$ are adjacent in $H$ if and only if they are adjacent in $G$. In either the permutation or graph case, the general idea is that one object contains another if the latter can be obtained from the former by deleting a subset of the vertices/entries, and the relationships between the remaining vertices/entries are unchanged.

Then, a \emph{permutation class} is a set of permutations $\mathcal{C}$ such that if $\pi \in \mathcal{C}$, then any pattern~$\sigma$ contained in $\pi$ is also contained in $\mathcal{C}$. A \emph{graph class} is a set of graphs $\mathcal{G}$ such that if~$G \in \mathcal{G}$, then any induced subgraph $H$ of $G$ is also contained in $\mathcal{G}$. It follows that given a class $\mathcal{C}$ of permutations, the set $G_{\mathcal{C}} = \{ G_{\pi} : \pi \in \mathcal{C} \}$ of inversion graphs is a class of graphs.

The class of all inversion graphs has unbounded lettericity. Indeed, we can obtain the perfect matching $m K_2$ for any integer $m \geq 1$ by taking the inversion graph of the permutation $21 \oplus 21 \oplus \dots \oplus 21 = 214365 \dots (2m)(2m-1)$, and we see that any class containing these graphs has unbounded lettericity.

\begin{theorem}[Alecu, Lozin, and de Werra~\text{\cite[Lemma 3]{alecu:the-micro-world:}}]\label{thm:matching_lettericity}
	For $m \geq 1$, the lettericity of the perfect matching $m K_2$ is $m$.
\end{theorem}

It is well known that the class of inversion graphs also contains all paths. The class of paths must have unbounded lettericity since it contains all perfect matchings, but the exact lettericity of paths is also known.

\begin{theorem}[Ferguson \cite{ferguson:on-the-letteric:}]
	For $n \geq 3$, the lettericity of the $n$-vertex path $P_n$ is $\lfloor \frac{n+4}{3} \rfloor$.
\end{theorem}

It is now understood exactly when a class of inversion graphs has bounded lettericity, and it can now be very elegantly stated given our previous discussion of geometric grid classes. This connection between geometric griddability and lettericity is highlighted in the following result, demonstrating that these two concepts capture essentially the same information about their respective objects.

\begin{theorem}[Alecu, Ferguson, Kant\'e, Lozin, Vatter, and Zamaraev~\cite{alecu:letter-graphs-a:iff}]\label{thm:gg_equals_lettericity}
	The permutation class~$\mathcal{C}$ is geometrically griddable if and only if the corresponding graph class $G_{\mathcal{C}}$ has bounded lettericity.
\end{theorem}

Though not central to this work, we mention that both geometric grid classes and $k$-letter graphs, (for a fixed integer $k$), have the important property of being \emph{labelled well-quasi ordered}~\cite{atminas:labelled-induce:, brignall:labelled-well-q:}. In particular, this property implies that these classes can be defined by a finite set of obstructions. In other words, there is a finite set of the respective objects such that an object is in the class if and only if it does not contain any of the objects in that set.

We do not include a full proof of either direction of Theorem~\ref{thm:gg_equals_lettericity}, however the following discussion will make clear that these notions are intimately related. This will be done by proving Proposition~\ref{prop:lettericity_bound_on_geom}. Before that, we must address some minor technicalities.

A matrix $M$ of size $s \times t$ is called a \emph{partial multiplication matrix} if there exists \emph{column} and \emph{row signs}
\[
c_1, \dots, c_s, r_1, \dots, r_t \in \{ 1, -1\}
\]
such that if $M(i,j)$ is nonzero then it is equal to $c_i r_j$. These matrices are necessarily $0/ \pm 1$ matrices, and the correspondence between their geometric grid classes and letter graphs can be much more simply stated than for $0/ \pm 1$ matrices in general.

Fortunately, there is a simple way to turn any $0/ \pm 1$ matrix into a partial multiplication matrix such that the corresponding geometric grid class is preserved. Indeed, we can `expand' each entry of a matrix using the substitution rules
\[
0 \rightarrow \left( \begin{matrix}
	0 & 0 \\ 0 & 0 
\end{matrix} \right), \hspace{3mm} 1 \rightarrow \left( \begin{matrix}
	0 & 1 \\ 1 & 0 
\end{matrix} \right), \text{ and } -1 \rightarrow \left( \begin{matrix}
	-1 & 0 \\ 0 & -1 
\end{matrix} \right).
\]
For example, the matrix $\left( \begin{smallmatrix}
	1 & 1 \\ 1 & -1
\end{smallmatrix} \right)$ can be expanded to the partial multiplication matrix
\[
\begin{pmatrix}
	0 & 1 & 0 & 1 \\
	1 & 0 & 1 & 0 \\
	0 & 1 & -1 & 0 \\
	1 & 0 & 0 & -1
\end{pmatrix},
\]
and their standard figures only differ in scale. This can be seen in Figure~\ref{fig:scale_to_pmm}, and we note that every matrix obtained in this manner is indeed a partial multiplication matrix as witnessed by the column and row signs $c_k = (-1)^{k}$ and $r_{\ell} = (-1)^{\ell}$, respectively. 

\begin{figure}[h]
\begin{center}
	\begin{tikzpicture}[scale=0.5]
	
		\draw[step=1cm,very thin] (0,0) grid (2,2);
		
		\draw[line cap = round] (0,0) to (1,1);
		\draw[line cap = round] (0,1) to (1,2);
		\draw[line cap = round] (1,1) to (2,2);
		\draw[line cap = round] (1,1) to (2,0);
				
	\end{tikzpicture}
	\hspace{10mm}
	\begin{tikzpicture}[scale=1]
	
		\draw[step=0.5cm,very thin] (0,0) grid (2,2);
		
		\draw[line cap = round] (0,0) to (1,1);
		\draw[line cap = round] (0,1) to (1,2);
		\draw[line cap = round] (1,1) to (2,2);
		\draw[line cap = round] (1,1) to (2,0);
				
	\end{tikzpicture}
\end{center}
\caption{Scaling the standard figure of a $0/ \pm 1$ matrix to that of a partial multiplication matrix.}
\label{fig:scale_to_pmm}
\end{figure}
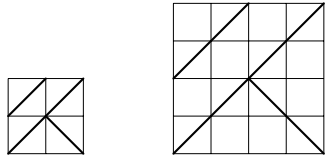

We are now able to state the result we will prove.

\begin{proposition}[Alecu, Lozin, De Werra, and Zamaraev~\cite{alecu:letter-graphs-a:}]\label{prop:lettericity_bound_on_geom}
	Suppose that $M$ is a partial multiplication matrix. If $\pi \in \Geom(M)$, then 
	\[
	\ell(G_{\pi}) \leq \sum_{i,j} |M(i,j)|.
	\]
\end{proposition}

Thus, if $M$ is not necessarily a partial multiplication matrix, the expansion process depicted in an example in Figure~\ref{fig:scale_to_pmm} shows that the lettericity of each inversion graph in $\Geom (M)$ is at most \emph{double} the number of nonzero entries of $M$.

The general idea behind this result is that if we fix an $M$-drawing of a permutation on the standard figure of a partial multiplication matrix $M$, then we can encode its inversion graph in a word such that all of the entries in the same cell are encoded by the same letter. So we begin by supposing that~$\Pi$ is a fixed $M$-drawing of a permutation $\pi$ on the standard figure of the partial multiplication matrix $M$ with fixed column and row signs~$(c_k)$ and~$(r_{\ell})$, respectively. 

We then use the column or row signs to give an orientation to the line segments in each of the cells containing one. If $c_k = 1$ we direct each of the line segments in column $k$ to the right, or to the left if $c_k = -1$. If we instead use the row signs, we direct the line segments upwards in row $\ell$ if $r_{\ell} = 1$, and downwards if $r_{\ell} = -1$, but the orientations obtained will be the same. This can clearly be seen in Figure~\ref{fig:orienting_sf}, continuing with the example from above. 

\begin{figure}[h]
\begin{center}
	\def\nodesize{2.5}
	\begin{tikzpicture}[scale=1.5]
	
		\draw[step=1cm, very thin] (0,0) grid (3,2);
		
		\draw[->, thin] (0.1,-0.1) to (0.9,-0.1);
		\node[] at (0.5,-0.3) {\small$c_1 = 1$};
		\draw[->, thin] (1.1,-0.1) to (1.9,-0.1);
		\node[] at (1.5,-0.3) {\small$c_2 = 1$};
		\draw[->, thin] (2.9,-0.1) to (2.1,-0.1);
		\node[] at (2.5,-0.3) {\small$c_3 = -1$};
		
		\draw[->, thin] (-0.1,0.1) to (-0.1,0.9);
		\node[anchor = east] at (-0.25,0.5) {\small$r_1 = 1$};
		\draw[->, thin] (-0.1,1.9) to (-0.1,1.1);
		\node[anchor = east] at (-0.25,1.5) {\small$r_2 = -1$};
		
		\draw[->, line cap = round] (0,0) to (0.97,0.97);
		\draw[->, line cap = round] (1,2) to (1.97,1.03);
		\draw[->, line cap = round] (3,2) to (2.03,1.03);
		\draw[->, line cap = round] (3,0) to (2.03,0.97);
		
		\node [draw, circle, fill, minimum size = \nodesize, label = \tiny$1$] at (0.2,0.2) {};
		\node [draw, circle, fill, minimum size = \nodesize, label = \tiny$3$] at (0.65,0.65) {};
		
		\node [draw, circle, fill, minimum size = \nodesize, label = \tiny$8$] at (1.28,1.72) {};
		\node [draw, circle, fill, minimum size = \nodesize, label = \tiny$6$] at (1.6,1.4) {};
		\node [draw, circle, fill, minimum size = \nodesize, label = \tiny$5$] at (1.85,1.15) {};
		
		\node [draw, circle, fill, minimum size = \nodesize, label = \tiny$4$] at (2.25,0.75) {};
		\node [draw, circle, fill, minimum size = \nodesize, label = \tiny$7$] at (2.62,1.62) {};
		
		\node [draw, circle, fill, minimum size = \nodesize, label = \tiny$2$] at (2.7,0.3) {};
		
	\end{tikzpicture}
\end{center}
\caption{Orienting the cells of a standard figure according to row and column signs.}
\label{fig:orienting_sf}
\end{figure}
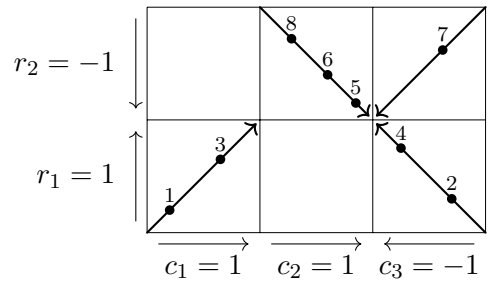

Now we define the \emph{base point} of each of these oriented line segments to be the endpoint of the line from which the arrow originates. Next, we consider the euclidean distance $d(i)$ of each point~$i \in \Pi$ from the base point in the cell in which it lies. That is, every $d(i)$ is between 0 and $\sqrt{2}$, and it is clear that for any two entries sharing a column or row of the standard figure that their distances are pairwise unique. For entries not sharing a row or column, some distances may be equivalent. We can avoid this situation by slightly perturbing the points in $\Pi$, so we assume without loss of generality that all of the distances are, in fact, pairwise unique. That is, emphasizing that we are identifying the points in $\Pi$ with the values of their corresponding entries in the permutation $\pi$, there exists a permutation $\tau$ of the values such that 
\[
0 < d(\tau(1)) < d(\tau(2)) < \dots < d(\tau(n)) < \sqrt{2}.
\]
From our running example in Figure~\ref{fig:orienting_sf}, we obtain $\tau = 18276345$. We remark that this ordering of the entries is not the only satisfactory one. All that matters is that the entries in each row and column are in the proper order according to their row or column sign, respectively. For example, since $r_1 = 1$, we read the entries in the first row from bottom to top, and so we must have that $\tau$ contains the substring $1234$. Further, since $r_2 = -1$, we read the entries drawn in that row from top to bottom, and thus $\tau$ must contain the substring $8765$. Likewise, since $c_3 = -1$, it must also contain the substring $274$. It should now be evident that these column and row relationships are captured by the distance-to-base-point ordering.

It is this permutation $\tau$ that gives us a viable ordering of the entries of $\pi$ with which to encode $G_{\pi}$ in a word. So, we first define a cell alphabet 
\[
\Sigma = \{ a_{ij} : M(i,j) \neq 0 \},
\]
and then the letters of our word $w$ to be
\[
w(i) = a_{k\ell} \hspace{1mm} \text{ if the entry with value $\tau(i)$ is plotted in the cell $(k, \ell)$ in the $M$-drawing $\Pi$.}
\]
Returning again to our running example, we obtain $w = a_{11} a_{22} a_{31} a_{32} a_{22} a_{11} a_{31} a_{22}$. Given this word $w$, it is straightforward to construct a proper decoder $D$ such that~$\Gamma_D(w) = G_{\pi}$.
\begin{itemize}
	\item Regarding entries within the same cell:
	\begin{itemize}
		\item If $M(k, \ell) = 1$, then the entries in the cell $(k,\ell)$ form an anticlique, and thus $(a_{k \ell}, a_{k \ell}) \notin D$.
		\item If $M(k, \ell) = -1$, then the entries in the cell $(k,\ell)$ form a clique, and thus $(a_{k \ell}, a_{k \ell}) \in D$.
		\item[E.g.,] we require $(a_{22}, a_{22}), (a_{31}, a_{31}) \in D$ for our running example.
	\end{itemize}
	\item Regarding cells not sharing a row or column:
	\begin{itemize}
		\item If $k<k'$ and $\ell<\ell'$, then no entry in the cell $(k, \ell)$ forms an inversion with an entry in the cell $(k', \ell')$. Thus $(a_{k \ell}, a_{k' \ell'}), (a_{k' \ell'}, a_{k \ell}) \notin D$.
		\item If $k<k'$ and $\ell>\ell'$, then every entry in the cell $(k, \ell)$ forms an inversion with every entry in the cell $(k',\ell')$. Thus $(a_{k \ell}, a_{k' \ell'}), (a_{k' \ell'}, a_{k \ell}) \in D$.
		\item[E.g.,] we require $(a_{22}, a_{31}), (a_{31}, a_{22}) \in D$ for our running example.
	\end{itemize}
	\item Regarding distinct cells sharing either a row or a column:
	\begin{itemize}
		\item Suppose $k < k'$. If $r_{\ell} = 1$ then the entries read from bottom to top in row $\ell$ are encoded from left to right in $w$, and thus $(a_{k' \ell}, a_{k \ell}) \in D, (a_{k \ell}, a_{k' \ell}) \notin D$. If $r_{\ell} = -1$ then the entries read from top to bottom in row $\ell$ are encoded from left to right in $w$, and thus $(a_{k \ell}, a_{k' \ell}) \in D, (a_{k' \ell}, a_{k \ell}) \notin D$.
		\item Suppose $\ell < \ell'$. If $c_k = 1$ then the entries read from left to right in column $k$ are encoded from left to right in $w$, and thus $(a_{k \ell'}, a_{k \ell}) \in D, (a_{k \ell}, a_{k \ell'}) \notin D$. If $c_k = -1$ then the entries read from right to left in column $k$ are encoded from left to right in $w$, and thus $(a_{k \ell}, a_{k \ell'}) \in D, (a_{k \ell'}, a_{k \ell}) \notin D$.
		\item[E.g.,] for our running example, $r_1 = 1$ requires we have $(a_{31},a_{11}) \in D$, $r_2 = -1$ requires we have $(a_{22}, a_{32}) \in D$, and $c_3 = -1$ requires we have $(a_{31}, a_{32}) \in D$.
	\end{itemize}
\end{itemize}

It should now be evident that there is a very close connection between letter graphs and geometric grid classes of permutations. We again point the reader toward \cite{alecu:letter-graphs-a:iff} for a further discussion of this connection and a full proof of Theorem~\ref{thm:gg_equals_lettericity}. We will return to this connection in Section~\ref{sec:lettericity_of_perms}, where we will make specific use of Proposition~\ref{prop:lettericity_bound_on_geom} to answer questions about the extremal and expected lettericity of inversion graphs.


\section{Lettericity of Almost All Graphs}\label{sec:BOTLOG}

\subsection{Petkov\v{s}ek's Problem 3}

This section closely follows the paper Mandrick and Vatter~\cite{mandrick:bounds-on-the-l:}.

In Section~5.3 of his paper~\cite{petkovsek:letter-graphs-a:} introducing the letter graph construction and lettericity, Petkov\v{s}ek shows that there are $n$-vertex graphs with lettericity at least $0.707n$, and then Problem~3 of his conclusion asks to
``find the maximal possible lettericity of an $n$-vertex graph, and the corresponding extremal graphs.'' Despite the significant recent interest in lettericity, this question remained unaddressed until~\cite{mandrick:bounds-on-the-l:}. The results in this paper demonstrate that the answer to the question is much greater than~$0.707n$. In particular, the greatest lettericity of an $n$-vertex graph lies between approximately $n-2\log_2 n$ and $n-\tfrac{1}{2}\log_2 n$.

Every $n$-vertex graph is an $n$-letter graph, as one can simply encode each vertex with its own letter and then add the appropriate pairs to the decoder. From this perspective, one `saves' letters by encoding multiple vertices with the same letter. It isn't difficult to see that the first and last vertices can always be encoded by the same letter, provided that no other vertices are encoded by that letter, and thus we can always save at least one letter. In other words, we have the elementary bound $\ell(G) \leq n-1$ for all $n$-vertex graphs~$G$. This gives rise to the following questions that we look to answer in this paper:
\begin{itemize}
	\item How many letters can we save in all graphs?
	\item How many letters can we expect to save in a random graph?
\end{itemize}

In the next subsection, we use a Ramsey-type approach to show that we can save at least $k \approx \tfrac{1}{2} \log_2 n$ letters for every $n$-vertex graph, and thus the lettericity of every $n$-vertex graph is bounded above by approximately $n-\tfrac{1}{2} \log_2 n$. Then in the following subsection, we show almost all $n$-vertex graphs have lettericity at least $n-(2\log_2 n+2\log_2\log_2 n)$.

Before getting to these results, the following proposition outlines the construction that will be used throughout this section. We establish the upper bound in by exploring ways to find induced subgraphs satisfying the hypotheses of Proposition~\ref{mainprop}. Then in establishing our lower bound, we show that for almost all graphs, the only way to save letters is by finding induced subgraphs that satisfy these hypotheses.

\begin{proposition}
\label{mainprop}
Suppose~$G$ is a graph with $n$ vertices containing an induced subgraph~$H$ with~$2k$ vertices that is a $k$-letter graph for a word of the form
\[
	w = \ell_1 \hspace{1mm} \ell_2 \dots \ell_k \hspace{1mm} \ell_{\pi(1)} \hspace{1mm} \ell_{\pi(2)} \dots \ell_{\pi(k)}
\]
for some permutation $\pi$ of $\{1, \dots, k \}$. Then, $w$ can be extended to a lettering of~$G$ by inserting new letters into the middle of $w$, and thus $\ell(G) \leq n-k$. 
\end{proposition}

\begin{proof}
Suppose~$H$ and $w$ are as in the hypothesis and that $D_1 \subseteq \{\ell_1, \ell_2, \dots, \ell_k \}^2$ is a decoder for which
\[
\Gamma_{D_1}(w) = H.
\]

Label the vertices of $G-H$ as $v_1,v_2,\dots,v_{n-2k}$ and let $\lambda_1$, $\dots$, $\lambda_{n-2k}$ be a set of distinct new letters disjoint from $\ell_1$, $\dots$, $\ell_k$. By choosing as our decoder the set $D_2 = \{ (\lambda_i, \lambda_j) : 1 \leq i < j \leq n-2k \textnormal{ and } v_i v_j \in E(G) \}$, we see immediately that
\[
\Gamma_{D_2}(\lambda_1 \lambda_2 \dots \lambda_{n-2k}) = G-H.
\]
Next define the word 
\[
w' = \ell_1 \dots \ell_k \hspace{1mm} \lambda_1 \lambda_2 \dots \lambda_{n-2k} \hspace{1mm} \ell_{\pi(1)} \dots \ell_{\pi(k)},
\]
and for each $1\le i\le k$, let $x_i$ and $y_i$ be the vertices encoded by the left and right instances of $\ell_i$ in~$w$, respectively. Now define the sets 
\begin{align*}
    D_x & = \{(\ell_i, \lambda_j):
    	\text{$1\le i\le k$, $1\le j\le n-2k$, and $x_i v_j \in E(G)$}\},\\
    D_y & = \{(\lambda_j, \ell_i):
    	\text{$1\le i\le k$, $1\le j\le n-2k$, and $v_j y_i \in E(G)$}\}.
\end{align*}
Letting $D = D_1 \cup D_2 \cup D_x \cup D_y$, it follows that $\Gamma_D(w') = G$, which proves the result.
\end{proof}

\subsection{Saving Letters in All Graphs}
\label{sec:upper-bound}

One way to satisfy the hypothesis of Proposition~\ref{mainprop} is to take~$H$ to be a clique or anticlique. Letting $R(k)$ denote the $k$th diagonal Ramsey number, we know that every graph on at least~$R(2k)$ vertices has such a subgraph, and so we obtain the following.

\begin{proposition}
\label{ramseyprop}
For each $k$ and any graph~$G$ on $n \geq R(2k)$ vertices,~$G$ has an induced subgraph with $2k$ vertices that is a $k$-letter graph on the word
\[
	w = \ell_1 \hspace{1mm} \ell_2 \dots \ell_k \hspace{1mm} \ell_{\pi(1)} \hspace{1mm} \ell_{\pi(2)} \dots \ell_{\pi(k)}
\]
for \emph{any} permutation $\pi$ of $\{1, \dots, k \}$. Thus, $\ell(G) \leq n - k$ by Proposition~\ref{mainprop}.
\end{proposition}


As it is known that
\[
	\sqrt{2}^k < R(k) \leq 4^k
\]
for all $k$, Proposition~\ref{ramseyprop} implies that for all graphs on $n \approx 4^{2k}$ vertices, we can save $k \approx \tfrac{1}{4} \log_2 n$ letters. We show below that we can save twice this many letters in every~$n$-vertex graph.

\begin{theorem}
\label{mainconstruction}
For every $k$ and each graph~$G$ on $n \geq 2(k-1) + 2^{2(k-1)} + 1$ vertices,~$G$ has an induced subgraph with $2k$ vertices that is a $k$-letter graph on the word
\[
w = \ell_1 \hspace{1mm} \ell_2 \dots \ell_k \hspace{1mm} \ell_k \dots \ell_2 \hspace{1mm} \ell_1.
\]
Thus, $\ell(G) \leq n - k$ by Proposition~\ref{mainprop}.
\end{theorem}

\begin{proof}
We use induction on $k$. For the base case of $k=1$, we have a graph~$G$ on $n\ge 2$ vertices, and the desired induced subgraph can be obtained by taking any two vertices.

Now suppose the result holds for some $k \geq 1$ and that~$G$ is a graph on $n \geq 2k + 2^{2k} + 1$ vertices. By our hypotheses, we have that~$G$ has an induced subgraph~$H$ that is a $k$-letter graph on the word $w = \ell_1 \hspace{1mm} \ell_2 \dots \ell_k \hspace{1mm} \ell_k \dots \ell_2 \hspace{1mm} \ell_1$, say with decoder $D$. Since there are $2^{2k} + 1$ vertices in~$G$ that are not in~$H$, the pigeonhole principle tells us that two of these vertices, call them $u$ and $v$, must agree on all of the vertices in~$H$. Let $H'$ be the induced subgraph of~$G$ on the vertex set $V(H) \cup \{ u,v \}$. 

We claim that $H'$ is a letter graph on the word $w' = \ell_1 \hspace{1mm} \ell_2 \dots \ell_k \hspace{1mm} \ell_{k+1} \hspace{1mm} \ell_{k+1} \hspace{1mm} \ell_k \dots \ell_2 \hspace{1mm} \ell_1$. For each ${1\le i\le k}$, let $x_i$ denote the vertex in~$H$ that is encoded by the left occurrence of $\ell_i$ in $w$, and similarly, let $y_i$ be the vertex that is encoded by the right occurrence of $\ell_i$, (just as in the proof of Proposition~\ref{mainprop}). Now let $X = \{ (\ell_i, \ell_{k+1}): x_i u \in E(G), (\textnormal{equivalently } x_i v \in E(G)) \}$ and $Y = \{ (\ell_{k+1}, \ell_i): u y_i \in E(G), (\textnormal{equivalently } v y_i \in E(G)) \}$. Next, let~$Z$ be the set $\{(\ell_{k+1}, \ell_{k+1}) \}$ if ${uv \in E(G)}$ and~$\varnothing$ otherwise. Then it follows that~$H'$ is a letter graph on the word~$w'$ with decoder $D' = D \cup X \cup Y \cup Z$, that is, $\Gamma_{D'} (w') = H'$. This gives the result.
\end{proof}

\subsection{Failing to Save Letters in Almost All Graphs}
\label{sec:lower-bound}

We now focus on demonstrating that almost all graphs have large lettericity, indicating that there is little room for improvement on the upper bound given in Theorem~\ref{mainconstruction}. To do so, we use the random graph $G(n, \nicefrac{1}{2})$, which is the graph on $n$ vertices where each edge appears independently with probability $\nicefrac{1}{2}$. Thus, every labeled $n$-vertex graph occurs with equal probability. We show that with probability tending to $1$ as $n\to\infty$, the lettericity of $G(n, \nicefrac{1}{2})$ is at least $n-(2\log_2 n + 2\log_2\log_2 n)$. First, we prove two results that greatly restrict the possible letterings of almost all graphs.

\begin{proposition}\label{threeverts}
For almost all graphs~$G$, no three vertices can be encoded by the same letter in a lettering of~$G$. 
\end{proposition}

\begin{proof}
Letting $G=G(n, \nicefrac{1}{2})$, we show that the probability that three vertices can be encoded with the same letter in a lettering of~$G$ tends to $0$ as $n \to \infty$. Assume that we have a lettering $w$ of~$G$ using the alphabet $\Sigma$, and that three vertices are encoded by the letter $a \in \Sigma$. Then there exist four possibly empty words $w_1$, $w_2$, $w_3$, and $w_4$, such that
	\[
	w = w_1 \hspace{1mm} a \hspace{1mm} w_2 \hspace{1mm} a \hspace{1mm} w_3 \hspace{1mm} a \hspace{1mm} w_4.
	\]

Let $x$, $y$ and $z$ be the vertices encoded by the left, middle and right instances of $a$ in~$w$, respectively. If a vertex is encoded by the instance of some letter in $w_1$ or $w_4$, then it must agree on each of the vertices $x,y$ and $z$. If a vertex is encoded in one of $w_2$ or $w_3$, then it either must agree on $y$ and $z$, or on $x$ and $y$, respectively.

For any vertex $v \in V(G) \setminus \{x,y,z\}$, there are four possible ways it can agree or disagree with~$x$,~$y$, and~$z$: it either agrees on all three vertices, agrees on~$\{x,y\}$ but not on~$z$, agrees on~$\{y,z\}$ but not on~$x$, or agrees on~$\{x,z\}$ but not on~$y$. Since $G=G(n, \nicefrac{1}{2})$, these four possibilities are equally likely, and only the last option prevents~$w$ from having a place in which $v$ can be encoded. Thus, the probability that~$v$ can be encoded in~$w$ is $\nicefrac{3}{4}$. It follows that the probability that every vertex in $V(G)\setminus\{x,y,z\}$ can be encoded in $w$ is~$(\nicefrac{3}{4})^{n-3}$.
	
Now let $A_{(x,y,z)}$ be the event that the vertices $x$, $y$ and $z$  can be encoded, in that order, by the same letter in a lettering of~$G$. We see from above that $\Pr{A_{(x,y,z)}} \leq (\nicefrac{3}{4})^{n-3}$. (In fact, the probability is~$0$ if $x$, $y$ and $z$ do not form a clique or anticlique.) Next, define the event
\[
	A = \bigcup_{(x,y,z)} A_{(x,y,z)},
\]
where the union is taken over all sequences $(x,y,z)$ of distinct vertices of~$G$. Thus, $A$ is the event that any three vertices can be encoded by the same letter in a lettering of~$G$. We have that
\[
	\Pr{A} \leq \sum_{(x,y,z)} \Pr{A_{(x,y,z)}} \leq n(n-1)(n-2) \cdot (\nicefrac{3}{4})^{n-3}, 
\]
and therefore $\Pr{A}\to 0$ as $n \to \infty$.
\end{proof}

Because of this result, we may henceforth assume that we are not able to encode three or more vertices using the same letter. As a consequence, if we are to save letters, we must do so in pairs. With this assumption, we see next that in any lettering of almost every graph~$G$, the letters that appear in pairs are crossing or nested. That is, for almost all graphs~$G$, if the letters $a,b \in \Sigma$ both appear twice in a lettering $w$ of~$G$, then it never happens that they appear \textit{separated} as $\cdots a\cdots a\cdots b\cdots b\cdots$, but rather they must appear \emph{crossing} as $\cdots a\cdots b\cdots a\cdots b\cdots$ or \emph{nested} as $\cdots a\cdots b\cdots b\cdots a\cdots$.

\begin{proposition}\label{prop:crossornest}
For almost all graphs~$G$, if two letters appear twice in a lettering of~$G$, they must appear in a crossing or nested pattern.
\end{proposition}

\begin{proof}
Letting $G=G(n, \nicefrac{1}{2})$, we show that the probability that~$G$ has a lettering in which two pairs of letters appear in a separated pattern tends to~$0$ as $n \to \infty$. This will yield the result since the only other possibility is that the pairs of letters are crossing or nested. Suppose that we have a lettering~$w$ of~$G$ over the alphabet $\Sigma$ containing $a$ and $b$ given by
\[
	w = w_1 \hspace{1mm} a \hspace{1mm} w_2 \hspace{1mm} a \hspace{1mm} w_3 \hspace{1mm} b \hspace{1mm} w_4 \hspace{1mm} b \hspace{1mm} w_5,
\]
for some possibly empty words $w_1$, $w_2$, $w_3$, $w_4$, and $w_5$. Let $x$, $y$, $s$, and~$t$ be the vertices of~$G$ encoded in $w$ by these instances of $a$ and $b$, reading left to right.
Fix a vertex in $V(G) \setminus \{x,y,s,t\}$.
This vertex can be encoded in $w_1$, $w_3$, or $w_5$ only if it agrees on~$\{x,y\}$ and on~$\{s,t\}$. Further, it can be encoded in~$w_2$ or~$w_4$ only if it agrees on~$\{s,t\}$ or on~$\{x,y\}$, respectively.

For each vertex $v \in V(G) \setminus \{x, y, s, t\}$, there are four possible ways it can agree or disagree on the pairs~$\{x,y\}$ and~$\{s,t\}$: it either agrees on both pairs, agrees only on~$\{x,y\}$, agrees only on~$\{s,t\}$, or disagrees on both pairs. These possibilities are equally likely because $G = G(n, \nicefrac{1}{2})$, and only the last case prevents~$w$ from having a place in which~$v$ can be encoded. Thus, the probability that~$v$ can be encoded somewhere in~$w$ is~$\nicefrac{3}{4}$. Hence, the probability that every vertex in $V(G) \setminus \{x,y,s,t\}$ can be encoded in~$w$ is $(\nicefrac{3}{4})^{n-4}$.

Let $B_{(x,y,s,t)}$ be the event that there is a lettering of~$G$ in which the vertices $x$, $y$, $s$ and $t$ are encoded in that order, $x$ and $y$ are encoded by the same letter, and $s$ and $t$ are encoded by a second letter. Thus, this is the event that these four vertices can correspond to a separated pattern encoded in the given order. From above, we have that $\Pr{B_{(x,y,s,t)}} \le (\nicefrac{3}{4})^{n-4}$. (In fact, this probability is $0$ if $x$ and $y$ do not agree on $s$ and $t$, and vice versa.) Next, define the event
\[
B = \bigcup_{(x,y,s,t)} B_{(x,y,s,t)},
\]
where the union is over all sequences $(x,y,s,t)$ of distinct vertices of~$G$. Thus,~$B$ is the event that a lettering of~$G$ has two pairs of letters in a separated pattern. We have that 
\[
\Pr{B} \leq \sum_{(x,y,s,t)} \Pr{B_{(x,y,s,t)}} \le n(n-1)(n-2)(n-3) \cdot (\nicefrac{3}{4})^{n-4}, 
\]
and therefore $\Pr{B}\to 0$ as $n \to \infty$.
\end{proof}

By Propopsition~\ref{threeverts}, we know that for almost all graphs~$G$ on $n$ vertices, if~$G$ has a lettering with~${n-k}$ letters, then it will have $k$ letters that appear twice and $n-2k$ letters that appear once. Suppose that the letters appearing twice are $\ell_1$, $\ell_2$, $\dots$, $\ell_k$. By Proposition~\ref{prop:crossornest}, we know that for almost all graphs that have such a lettering, there is a permutation~$\pi$ of $\{1, \dots, k \}$ such that the subword of $w$ containing all of these letters is
\begin{equation}\label{word}\tag{$\dagger$}
    \ell_1 \hspace{1mm} \ell_2 \dots  \ell_k \hspace{1mm} \ell_{\pi(1)} \ell_{\pi(2)} \dots \ell_{\pi(k)}.
\end{equation}
Note that this is the same construction considered in Proposition~\ref{mainprop}.

It remains to analyze the probability that there is an induced subgraph that can be lettered by a word such as that of (\ref{word}). To this end, let~$G = G(n, \nicefrac{1}{2})$ and $k$ an integer satisfying $2k\le n$. Further, let $(v_i)=(v_1,\dots,v_{2k})$ be a sequence of distinct vertices of~$G$ and $\pi$ a permutation of~$\{1,\dots,k\}$. We define $C_{(v_i),\pi}$ to be the event that there exists a decoder $D\subseteq\{\ell_1,\dots,\ell_k\}^2$ such that the mapping $v_i\mapsto i$ is an isomorphism between the induced subgraph $G[\{v_1,\dots,v_{2k}\}]$ and the letter graph $\Gamma_D(\ell_1\cdots\ell_k \hspace{1mm} \ell_{\pi(1)}\cdots\ell_{\pi(k)})$.

To evaluate $\Pr{C_{(v_i), \pi}}$, for every pair $i < j$ we handle the case of the $\ell_i$'s and $\ell_j$'s being either crossing or nested in the word $\ell_1\cdots\ell_k\hspace{1mm}\ell_{\pi(1)}\cdots\ell_{\pi(k)}$. If these letters are crossing, then, as is indicated by the three lines in Figure~\ref{fig:cross}, there are three potential edges that must agree. That is, all three of these edges will be decided by the presence or absence of $(\ell_i, \ell_j)$ in the decoder, and hence they must all be edges or nonedges. With each edge decided with probability $\nicefrac{1}{2}$, the probability that the three potential edges agree is $\nicefrac{1}{4}$.

\begin{figure}[ht]
    \centering
    \begin{tikzpicture}[scale=0.7]
        \node (w) at (0,0) {$\ell_i$};
        \node (x) at (2,0) {$\ell_j$};
    
        \node (y) at (4,0) {$\ell_i$};
        \node (z) at (6,0) {$\ell_j$};

        \draw [line width = 0.3mm]   (w) to[out=45,in=135] (z);
        \draw [line width = 0.3mm]   (w) to[out=35,in=145] (x);
        \draw [line width = 0.3mm]   (y) to[out=35,in=145] (z);
        
       \end{tikzpicture}
       \caption{The potential edges that must agree in a crossing pattern.}
       \label{fig:cross}
\end{figure}
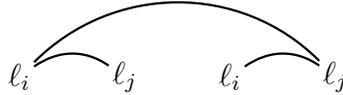

If the letters $\ell_i$ and $\ell_j$ are nested, then we see in Figure \ref{fig:nest} that there are two pairs of potential edges that must agree. That is, the solid lines must agree since they are both determined by the presence or absence of $(\ell_i, \ell_j)$ in the decoder, and the dotted lines must agree because they are both determined by the presence or absence of $(\ell_j, \ell_i)$ in the decoder. Again, with each of these edges decided with probability $\nicefrac{1}{2}$, the probability that both of these pairs of potential edges agree is $\nicefrac{1}{4}$. 

\begin{figure}[ht]
    \centering
    \begin{tikzpicture}[scale=0.7]
        \node (w) at (0,0) {$\ell_i$};
        \node (x) at (2,0) {$\ell_j$};
    
        \node (y) at (4,0) {$\ell_j$};
        \node (z) at (6,0) {$\ell_i$};
    
        \draw [line width = 0.3mm]   (w) to[out=45,in=135] (y);
        \draw [line width = 0.3mm]   (w) to[out=35,in=145] (x);
        \draw [line width = 0.5mm, dotted]   (y) to[out=35,in=145] (z);
        \draw [line width = 0.5mm, dotted]   (x) to[out=45,in=135] (z);
        \end{tikzpicture}
        \caption{The potential edges that must agree in a nested pattern.}
    \label{fig:nest}
\end{figure}
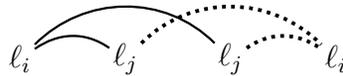

For each pair $i<j$, which of these two cases must be satisfied is determined by $\pi$, and since each case has probability~$\nicefrac{1}{4}$ of being satisfied, it follows that
\[
	\Pr{C_{(v_i), \pi}} = (\nicefrac{1}{4})^{\binom{k}{2}} = 2^{-k(k-1)}.
\]
Next, we define the event
\[
	C = \bigcup_{(v_i), \pi} C_{(v_i), \pi},
\]
where the union is over all sequences $(v_i)$ of $2k$ distinct vertices of~$G$ and all permutations~$\pi$ of $\{1, \dots, k \}$. Thus, $C$ is the event that an induced subgraph of~$G$ can be encoded in the form of (\ref{word}). In light of the preceding arguments, we can regard $C$ as the only event in which $k$ letters can be saved in a lettering of almost all graphs~$G$. To obtain the main result of this section, we simply need to minimize the value of $k$ so that the probability of~$C$ still goes to zero as $n\to\infty$.

\begin{theorem}\label{thm:lettericity_lb}
	For almost all graphs~$G$ with $n$ vertices, we have 
	\[
	\ell(G) \geq n - (2 \log_2 n + 2 \log_2 \log_2 n).
	\]
\end{theorem}

\begin{proof}
It is clear from above that
\begin{equation*}
	\Pr{C}
	\leq
	\sum_{(v_i), \pi} \Pr{C_{(v_i), \pi}}
	=
	n(n-1)\cdots (n-2k+1) \cdot k! \cdot 2^{-k(k-1)}.
\end{equation*}
Using the straightforward inequalities $n(n-1)\cdots (n-2k+1)\le n^{2k}$ and $k! \leq k^k$, we see that
\[
	\Pr{C}
	\leq
	n^{2k} k^k 2^{-k(k-1)}
	=
	(n^2 k 2^{-k+1})^{k}.
\]
Setting $k = 2 \log_2 n + 2 \log_2 \log_2 n$, simple computations show that $\Pr{C}$ tends to $0$ as $n \to \infty$, and therefore the result follows from the preceding arguments.
\end{proof}


\subsection{Possible Improvement}

By Proposition~\ref{mainprop}, we see that we can save letters in any graph by encoding induced subgraphs with a word of the form given in~(\ref{word}). Then, in establishing our upper bound on the maximum lettericity in graphs in Theorem~\ref{mainconstruction}, we only consider finding induced subgraphs encoded by the word $\ell_1 \hspace{1mm} \ell_2 \dots \ell_k \hspace{1mm} \ell_k \dots \ell_2 \hspace{1mm} \ell_1$ for some integer $k$. This is quite narrow in scope since there are many more possibilities outlined by~(\ref{word}). These considerations suggest that it might be possible to improve on this upper bound by expanding our understanding of the graphs that can be encoded by words of the form~(\ref{word}), (in fact, Propositions~\ref{threeverts} and \ref{prop:crossornest} indicate that this is essentially the only way we can hope to do so).

In this direction, we return to the chains introduced in Chapter~\ref{chap:uniqueness}. We recall that they are sequences of unique vertices $p_1, \dots, p_m$ of a graph such that for each $p_i$ such that $i \geq 3$, the vertex $p_{i-1}$ is its unique neighbor or nonneighbor amongst $p_1, \dots, p_{i-1}$. We see with the next proposition that because of their nice edge behavior, we are able to encode them with words of the form~(\ref{word}) in a more interesting way.

\begin{proposition}\label{prop:encoding_chain}
	If $G$ is a graph with a chain $p_1, p_2, \dots, p_{2k}$, then we can encode the induced subgraph $G[p_1, p_2, \dots, p_{2k}]$ with a word
	\[
	w = \ell_1 \hspace{1mm} \ell_2 \dots \ell_k \hspace{1mm} \ell_{\pi(1)} \hspace{1mm} \ell_{\pi(2)} \dots \ell_{\pi(k)}
	\]
	for some permutation $\pi$ of $\{1, \dots, k \}$. Thus, $\ell(G) \leq n - k$ by Proposition~\ref{mainprop}.
\end{proposition}

\begin{proof}
	We prove the result by induction on $k$, with the base case of $k=1$ being trivial. For our induction to work, we strengthen the hypothesis by adding the stipulation that, in encoding a chain $p_1, \dots, p_{2k}$ by such a word $w$, we can encode the last vertex $p_{2k}$ with the last or first letter of $w$. This is of course true for our base case $k=1$.
	
	So suppose the result is true for some $k \geq 1$ and that $p_1, \dots, p_{2k+2}$ is a chain in a graph $G$. By our inductive hypothesis, we can take the prefix $p_1, \dots, p_{2k}$ of this chain and encode its induced subgraph $G[p_1, \dots, p_{2k}]$ with a word $w = \ell_1 \dots \ell_k \hspace{1mm} \ell_{\pi(1)} \dots \ell_{\pi(k)}$ for a permutation~$\pi$ and such that $p_{2k}$ is encoded by the first or last letter of $w$. We can assume without loss of generality that $p_{2k}$ is encoded by the last letter of $w$, (that is, given any letter graph~$\Gamma_D(w)$, we can reverse the word $w$ and each element of the decoder $D$ to obtain~$w^{\text{r}}$ and~$D^{\text{r}}$, respectively, and it follows that $\Gamma_D(w) = \Gamma_{D^{\text{r}}}(w^{\text{r}})$). 
	
	We prove the result by extending the word $w$ to a word $w'$ in such a way that the vertices~$p_{2k+1}$ and $p_{2k+2}$ are both encoded by the same letter $\ell_{k+1}$, one of which lies in the first half of $w'$, the other in the second half, and $p_{2k+2}$ is encoded by the first or last letter of~$w'$. We do so in two cases, for which we define the vertex $p_i$, for each $i \geq 3$, of a chain $p_1, p_2, \dots$ to be
	\begin{itemize}
		\item a \emph{pendant} if $p_i \sim p_{i-1}$ and $p_i \not\sim p_1, \dots, p_{i-2}$, or
		\item an \emph{antipendant} if $p_i \not\sim p_{i-1}$ and $p_i \sim p_1, \dots, p_{i-2}$.
	\end{itemize}
	
	In the first case, we assume $p_{2k+1}$ and $p_{2k+2}$ are either both pendants or both antipendants. That is, $p_{2k+1}$ and $p_{2k+2}$ agree on each of $p_1, \dots, p_{2k-1}$, but then disagree on $p_{2k}$. In either case, we insert $\ell_{k+1}$ in the middle of $w$ to encode $p_{2k+1}$, and also at the end to encode $p_{2k+2}$ as in Figure~\ref{fig:both_pendants}, yielding a satisfactory $w'$.
	
	\begin{figure}[ht]
    \centering
    \begin{tikzpicture}[scale=1]
    	\node (w') at (-1,0.05) {$w'=$};
    
        \node (1) at (0,0) {$\ell_1$};
        \node (2) at (1,0) {$\ell_2$};
        \node (3) at (2,0) {$\dots$};
        \node (4) at (3,0) {$\ell_k$};
        
        \node (2k+1) at (4,1) {$p_{2k+1}$};
        \draw[->] (4,0.8) to (4,0.4);
        \node (5) at (4,0) {$\ell_{k+1}$};
        \draw[thin] ($(5.south west) + (-0.02,-0.02)$) rectangle ($(5.north east) + (0.02,0.02)$); 
        
        \draw[thick] (5,-0.5) to (5, 0.5);
        
        \node (6) at (6,0) {$\ell_{\pi(1)}$};
        \node (7) at (7,0) {$\ell_{\pi(2)}$};
        \node (8) at (8,0) {$\dots$};
        \node (2k) at (9,1) {$p_{2k}$};
        \draw[->] (9,0.8) to (9,0.4);
        \node (9) at (9,0) {$\ell_{\pi(k)}$};
        
        \node (2k+2) at (10,1) {$p_{2k+2}$};
        \draw[->] (10,0.8) to (10,0.4);
        \node (10) at (10,0) {$\ell_{k+1}$};
        \draw[thin] ($(10.south west) + (-0.02,-0.02)$) rectangle ($(10.north east) + (0.02,0.02)$); 

        \end{tikzpicture}
        \vspace{-2mm}
        \caption{Expanding the word $w$ in the proof of Proposition~\ref{prop:encoding_chain} in the case that $p_{2k+1}$ and~$p_{2k+2}$ are both pendants or antipendants.}
    \label{fig:both_pendants}
	\end{figure}
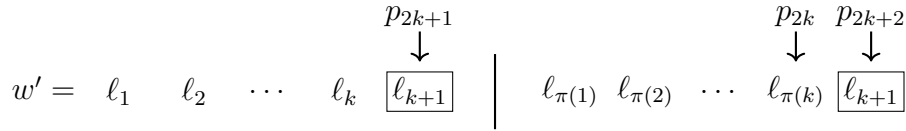
	
	In the case that they are both pendants, we need only add $(\ell_{k+1}, \ell_{\pi(k)})$ and $(\ell_{k+1}, \ell_{k+1})$ to obtain a viable decoder. If they are both antipendants, then we add $(\ell_{i}, \ell_{k+1})$ for each $i \in \{ 1, \dots, k\}$, and $(\ell_{k+1}, \ell_{\pi(j)})$ for each $j \in \{ 1, \dots, k-1 \}$.
	
	In the second case, we assume that exactly one of $p_{2k+1}$ and $p_{2k+2}$ is a pendant, and the other is an antipendant. That is, $p_{2k+1}$ and $p_{2k+2}$ agree only on $p_{2k}$, and disagree on each of $p_1, \dots, p_{2k-1}$. In either case, we insert $\ell_{k+1}$ to encode $p_{2k+1}$ just to the left of the last letter in $w$ that encodes~$p_{2k}$, and also at the very start to encode $p_{2k+2}$ as in Figure~\ref{fig:one_pendant_one_antipendant}, yielding a satisfactory~$w'$.
	
	\begin{figure}[ht]
    \centering
    \begin{tikzpicture}[scale=1]
    	\node (w') at (-1,0.05) {$w'=$};
    	
    	\node (2k+2) at (0,1) {$p_{2k+2}$};
        \draw[->] (0,0.8) to (0,0.4);
        \node (5) at (0,0) {$\ell_{k+1}$};
        \draw[thin] ($(5.south west) + (-0.02,-0.02)$) rectangle ($(5.north east) + (0.02,0.02)$); 
    
        \node (1) at (1,0) {$\ell_1$};
        \node (2) at (2,0) {$\ell_2$};
        \node (3) at (3,0) {$\dots$};
        \node (4) at (4,0) {$\ell_k$};
        
        \draw[thick] (5,-0.5) to (5, 0.5);
        
        \node (6) at (6,0) {$\ell_{\pi(1)}$};
        \node (7) at (7,0) {$\ell_{\pi(2)}$};
        \node (8) at (8,0) {$\dots$};
        
        \node (2k+1) at (9,1) {$p_{2k+1}$};
        \draw[->] (9,0.8) to (9,0.4);
        \node (10) at (9,0) {$\ell_{k+1}$};
        \draw[thin] ($(10.south west) + (-0.02,-0.02)$) rectangle ($(10.north east) + (0.02,0.02)$); 
        
        \node (2k) at (10,1) {$p_{2k}$};
        \draw[->] (10,0.8) to (10,0.4);
        \node (9) at (10,0) {$\ell_{\pi(k)}$};

        \end{tikzpicture}
        \vspace{-2mm}
        \caption{Expanding the word $w$ in the proof of Proposition~\ref{prop:encoding_chain} in the case that exactly one of $p_{2k+1}$ and~$p_{2k+2}$ is a pendant.}
    \label{fig:one_pendant_one_antipendant}
	\end{figure}
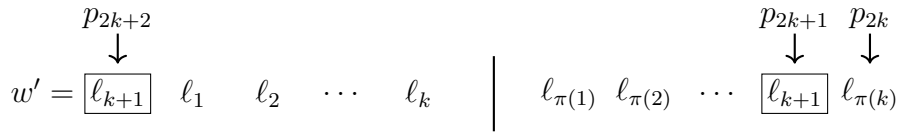
	
	In the case that $p_{2k+1}$ is the pendant, we need to add $(\ell_{k+1}, \ell_{i})$ for each $i \in \{ 1, \dots, k \}$ to obtain a viable decoder. Then in the case that $p_{2k+2}$ is the pendant, we add $(\ell_{j}, \ell_{k+1})$ for each $j \in \{ 1, \dots, k+1 \}$. This gives the result. 
\end{proof}

We now discuss some ideas that might contribute to improving on the upper bound given in Theorem~\ref{mainconstruction} in light of Proposition~\ref{prop:encoding_chain}. In~\cite{chudnovsky:unavoidable-ind:}, it is shown that large enough prime graphs must contain long chains or some other `well-behaved' structures. Is it possible that these other structures, (or some generalization of them), are also encodable by words of the form~(\ref{word})? Is it then possible that we can use their ideas to improve on our bound in the case of prime graphs?

For the case of nonprime graphs, we include another idea to approach lowering this upper bound. For example, suppose that a graph $G$ has disjoint nontrivial modules~$M_1$ and~$M_2$. Then for any two unique vertices $p_1, p_2 \in M_1$, any chain starting with these vertices has all of its vertices contained in $M_1$, and likewise for $M_2$. So we can take the longest even-length chain $p_1, p_2, \dots, p_{2s}$ beginning with two vertices $p_1, p_2 \in M_1$, and further take the longest even-length chain $q_1, q_2, \dots, q_{2t}$ such that $q_1, q_2 \in M_2$. By Proposition~\ref{prop:encoding_chain}, we can encode the chain $p_1, p_2, \dots, p_{2s}$ with the a word $\ell_1 \dots \ell_s \ell_{\sigma(1)} \dots \ell_{\sigma(s)}$ for some permutation $\sigma$ of $\{ 1, \dots, s\}$, and the chain $q_1, q_2, \dots, q_{2t}$ with another word $\ell_{s+1} \dots \ell_{s+t} \ell_{\tau(s+1)} \dots \ell_{\tau(s+t)}$ for a permutation $\tau$ of $\{ s+1, \dots, s+t\}$. Because they are contained in disjoint modules, it isn't difficult to see that $G[p_1, \dots, p_{2s},q_1, \dots, q_{2t}]$ can be encoded by the word
\[
w^* = \ell_1 \dots \ell_s \hspace{1mm} \ell_{s+1} \dots \ell_{s+t} \hspace{1mm} \ell_{\tau(s+1)} \dots \ell_{\tau(s+t)} \hspace{1mm} \ell_{\sigma(1)} \dots \ell_{\sigma(s)},
\]
which we note is also in the form of~(\ref{word}). It also isn't difficult to see that if there is another module disjoint from $M_1$ and $M_2$, we can find another chain that can be encoded by words of the form~(\ref{word}), and then extend this word $w^*$ by nesting it in the middle. Can we then improve on our upper bound in the case of graphs that can be partitioned into many disjoint modules?

\section{Lettericity of Inversion Graphs}\label{sec:lettericity_of_perms}

\subsection{Maximum Lettericity of Inversion Graphs}

As we have seen with Theorem~\ref{thm:matching_lettericity}, there are $n$-vertex inversion graphs with lettericity~$\nicefrac{n}{2}$. However, as we have seen in Section~\ref{sec:BOTLOG}, almost all $n$-vertex graphs have lettericity near $n$, essentially double. It is natural to ask if there exists $n$-vertex inversion graphs that also have lettericity near $n$. It turns out that the answer is negative, and we can prove it by utilizing geometric griddings of permutations. Moreover, the following upper bound on the lettericity of inversion graph is tight to Theorem~\ref{thm:matching_lettericity}, and thus perfect matchings are extremal.

\begin{theorem}\label{thm:inv_lettericity_ub}
	For any permutation $\pi$ of $[n]$, the inversion graph $G_{\pi}$ has lettericity at most~$\lceil \nicefrac{n}{2} \rceil$.
\end{theorem}

\begin{proof}
	We prove the result by finding a row $0/ \pm 1$ matrix $M$---which is thus automatically a partial multiplication matrix---of size $\lceil \nicefrac{n}{2} \rceil \times 1$, (recall that we are indexing our matrices by cartesian coordinates), for which there exists an $M$-drawing of $\pi$. We will choose the entries of $M$ such that the first two entries of $\pi$ can be drawn in the cell $(1,1)$ of the standard figure, the next two entries can be drawn in the cell $(2,1)$, and so on in an $M$-drawing of $\pi$. That is, each cell contains two entries except for the last one in the case that $n$ is odd. 
	
	So, if $\pi(1) < \pi(2)$ we set $M(1,1) = 1$, and $M(1,1) = -1$ otherwise. Similarly, we set $M(2,1)$ equal to 1 if $\pi(3) < \pi (4)$, and $-1$ otherwise. We continue in this way until we reach the end of $\pi$, where in the case that $n$ is odd we can set $M(\lceil \nicefrac{n}{2} \rceil, 1)$ to either $1$ or $-1$. For example, letting $\pi = 38 49 61 27 5$, since $\pi(1) < \pi(2)$ we set $M(1,1) = 1$, then set $M(2,1)=1$ since $\pi(3) < \pi(4)$, and next $M(3,1) = -1$ since $\pi(5) > \pi(6)$, etc., resulting in $M = \begin{pmatrix} 1 & 1 & -1 & 1 & -1 \end{pmatrix}$.
	
	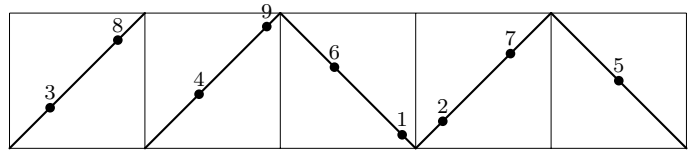
\begin{figure}[h]
	\begin{center}
		\def\nodesize{2.5}
		\begin{tikzpicture}[scale=1.8]
	
			\draw[step=1cm,very thin] (0,0) grid (5,1);
		
			\draw[line cap = round] (0,0) to (1,1);
			\draw[line cap = round] (1,0) to (2,1);
			\draw[line cap = round] (2,1) to (3,0);
			\draw[line cap = round] (3,0) to (4,1);
			\draw[line cap = round] (4,1) to (5,0);

			\node [draw, circle, fill, minimum size = \nodesize, label = \tiny$3$] at (0.3,0.3) {};	
			\node [draw, circle, fill, minimum size = \nodesize, label = \tiny$8$] at (0.8,0.8) {};	
		
			\node [draw, circle, fill, minimum size = \nodesize, label = \tiny$4$] at (1.4,0.4) {};	
			\node [draw, circle, fill, minimum size = \nodesize, label = \tiny$9$] at (1.9,0.9) {};
		
			\node [draw, circle, fill, minimum size = \nodesize, label = \tiny$6$] at (2.4,0.6) {};	
			\node [draw, circle, fill, minimum size = \nodesize, label = \tiny$1$] at (2.9,0.1) {};	
		
			\node [draw, circle, fill, minimum size = \nodesize, label = \tiny$2$] at (3.2,0.2) {};	
			\node [draw, circle, fill, minimum size = \nodesize, label = \tiny$7$] at (3.7,0.7) {};	
			
			\node [draw, circle, fill, minimum size = \nodesize, label = \tiny$5$] at (4.5,0.5) {};	
				
		\end{tikzpicture}
	\end{center}
	\caption{Drawing the permutation $384961275$ on the standard figure of a row matrix.}
	\label{fig:inv_lettericity_ub}
	\end{figure}
	
	It is then clear how to construct such $M$-drawings for permutations in general, as seen for this example in Figure~\ref{fig:inv_lettericity_ub}. The result then follows from Proposition~\ref{prop:lettericity_bound_on_geom}.
\end{proof}

\subsection{Expected Lettericity of Inversion Graphs}

We now turn our interests towards the expected lettericity of inversion graphs, and thus we define the random variable $X_{\ell}:S_n \to \mathbb{R}$ by setting $X_{\ell}(\pi) = \ell(G_{\pi})$. We will use the notation~$\Ex{X}$ to denote the expectation (i.e., mean) of a random variable $X$, and thus we are interested in the value of~$\Ex{X_{\ell}}$.

Given this approach of finding standard figures of row matrices on which to draw permutations, it seems as though we could often expect to do better. For example, given the permutation $347156982$, we can partition it into three monotone runs: $347$, $1569$ and $82$, and thus we can find an $M$-drawing of this permutation for $M = \begin{pmatrix} 1 & 1 & -1 \end{pmatrix}$ as seen in Figure~\ref{fig:perm_on_small_rm}. That is, we have $\ell(G_{347156982})$ is at most~$3$, as opposed to $\lceil \nicefrac{9}{2} \rceil = 5$ as given by Theorem~\ref{thm:inv_lettericity_ub}.

\begin{figure}[h]
	\begin{center}
		\def\nodesize{2.5}
		\begin{tikzpicture}[scale=1.8]
	
			\draw[step=1cm,very thin] (0,0) grid (3,1);
		
			\draw[line cap = round] (0,0) to (1,1);
			\draw[line cap = round] (1,0) to (2,1);
			\draw[line cap = round] (2,1) to (3,0);

			\node [draw, circle, fill, minimum size = \nodesize, label = \tiny$3$] at (0.3,0.3) {};	
			\node [draw, circle, fill, minimum size = \nodesize, label = \tiny$4$] at (0.4,0.4) {};	
			\node [draw, circle, fill, minimum size = \nodesize, label = \tiny$7$] at (0.7,0.7) {};
				
			\node [draw, circle, fill, minimum size = \nodesize, label = \tiny$1$] at (1.1,0.1) {};
			\node [draw, circle, fill, minimum size = \nodesize, label = \tiny$5$] at (1.5,0.5) {};
			\node [draw, circle, fill, minimum size = \nodesize, label = \tiny$6$] at (1.6,0.6) {};	
			\node [draw, circle, fill, minimum size = \nodesize, label = \tiny$9$] at (1.9,0.9) {};	
		
			\node [draw, circle, fill, minimum size = \nodesize, label = \tiny$8$] at (2.2,0.8) {};	
			\node [draw, circle, fill, minimum size = \nodesize, label = \tiny$2$] at (2.8,0.2) {};	
				
		\end{tikzpicture}
	\end{center}
	\caption{Drawing the permutation $347 1569 82$ on the standard figure of a row matrix.}
	\label{fig:perm_on_small_rm}
\end{figure}

We can then rephrase the problem as follows: `how few bars do we need to place between entries of a permutation written in one-line notation so that the permutation is broken up into monotone runs?' For our example above we require only two bars,
\[
347|1569|82,
\]
thus breaking up the permutation into three monotone runs. So, we define the random variable $X_{\text{r}}: S_n \to \mathbb{R}$ such that $X_{\text{r}}(\pi)$ is the minimum possible number of monotone runs obtained by breaking up~$\pi$ in this way. Our discussions show that $X_{\ell} \leq X_{\text{r}}$, and thus it follows that $\Ex{X_{\ell}} \leq \Ex{X_{\text{r}}}$. 

Despite its natural definition, the author is not aware of any existing results about $\Ex{X_{\text{r}}}$ in the literature. Here we include some simple arguments to obtain an upper bound on $\Ex{X_{\text{r}}}$, and thus also on $\Ex{X_{\ell}}$. We begin by noting that for a permutation $\pi$ of $[n]$, if we place a bar between $\pi(i)$ and $\pi(i+1)$ for each $i \leq n-1$ such that $\pi(i) > \pi(i+1)$---in which case we call $i$ a \emph{descent}---then~$\pi$ is partitioned into blocks of ascending runs. We see, for example, the bars inserted in this manner here:
\[
\sigma = 67|5|4|19|8|23.
\]

Therefore, if $X_d:S_n \to \mathbb{R}$ is the random variable such that $X_d(\pi)$ is the number of descents of~$\pi$, then $X_{\text{r}} \leq 1 + X_{d}$. It is well known that $\Ex{X_{d}} = \frac{n-1}{2}$, and thus $\Ex{X_{\text{r}}} \leq \frac{n+1}{2}$. However, it is clear that by inserting bars at each descent there might be many expendable bars for our objective. For example, in our annotated permutation $\sigma$ given above, we can simply delete the bar between $5$ and $4$ and the remaining bars still break up $\sigma$ into monotone runs.

The reason we can do this is because that bar demarcates the middle of three descents occurring in a row, and thus is in the middle of a descending run. We thus delete all such bars, and therefore we can define the random variable $X_{ddd}: S_n \to \mathbb{R}$ such that $X_{ddd}(\pi)$ is the number of occurrences of three descents in a row in $\pi$, and it follows that $X_{\text{r}} \leq 1 + X_d - X_{ddd}$. Further, it is a simple exercise to show \[ \Ex{X_{ddd}}=\frac{n-3}{24}. \]

Even with those bars removed, there are further possible ways to remove bars. Indeed, we see that after deleting the bar between $5$ and $4$ we have the substring $54|19|8|2$. We can replace two of the bars with one and obtain $541|98|2$, still breaking up the permutation into monotone runs, as desired.

The reason we can do this is because we have a series of at least two descents, then a single ascent, followed by another series of at least two descents. The two entries participating in the single ascent are enclosed between two bars, however the first such entry can be added to the descending run preceding it, and the second entry can be added to the descending run succeeding it. 

It follows that we can shed one bar in each such case, and so we define the random variable $X_{ddadd}: S_n \to \mathbb{R}$ such that $X_{ddadd}(\pi)$ is the number of occurrences of $i \leq n-5$ such that $\pi(i) > \pi(i+1) > \pi(i+2) < \pi(i+3) > \pi(i+4) > \pi(i+5)$, and it follows that $X_{\text{r}} \leq 1 + X_d - X_{ddd} - X_{ddadd}$. It is another simple exercise to verify that \[\Ex{X_{ddadd}} = \frac{n-5}{6!} \left( \binom{6}{3}-1 \right).\] In summary, we have proved the following.

\begin{theorem}
	For all $n \geq 6$, the expected lettericity $\Ex{X_{\ell}}$ of the inversion graph of a permutation of $[n]$ is at most $0.432n + 0.758$.
\end{theorem}

\begin{proof}
	We simply collect the above arguments and make use of the linearity of expectation:
	\begin{align*}
	\Ex{X_{\ell}} & \leq \Ex{X_{\text{r}}} \\
	& \leq 1 + \Ex{X_d} - \Ex{X_{ddd}} - \Ex{X_{ddadd}} \\
	& = \frac{311}{720} \cdot n + \frac{109}{144} \\
	& \approx 0.432 n + 0.758.
	\end{align*}
	This gives the result.
\end{proof}

Although the calculations become somewhat involved, we can continue to lower the coefficient on $n$ here by using essentially the same technique. For example, if we have a contiguous substring of a permutation $\pi$ that reads `descent, descent, ascent, descent, ascent, descent, descent' from left to right, then we can replace three bars with two, very similarly as to how we replaced two bars with one in the case of $X_{ddadd}$ above. We then need to compute the expectation of the random variable $X_{ddadadd}: S_n \to \mathbb{R}$ counting such occurrences, and we can subtract that from this upper bound. From here, it is evident how to generalize this argument to continue lowering this coefficient by constructing more random variables, and then assessing their expectations using inclusion-exclusion~\cite[Chp. 2]{stanley:enumerative-com:1}. More explicitly, continuing with our notation for the random variables constructed above, we have
\[
X_{\ell} \leq X_{\text{r}} \leq 1 + X_d - \sum_{k \geq 0} X_{dd(ad)^k d}.
\]

\section{Permutation Letter Graphs}

\subsection{Encoding on Permutations}

Throughout this chapter, we answered several extremal questions regarding the number of letters required to encode a graph `in a string' with a decoder mechanism. It is then natural to ask if we can significantly lower the number of letters required to encode a graph `on a permutation,' utilizing the same type of decoder mechanism. That is, given a finite alphabet $\Sigma$, we take two sets~$I, N \subseteq \Sigma^2$ to be our decoders, $I$ to decode inversions and $N$ to decode noninversions. More specifically, given a fixed integer $n$, we let $w$ be a word with letters $w(1), \dots, w(n) \in \Sigma$, $\pi$ a permutation of $[n]$, and define the \emph{permutation letter graph} of $w$ with respect to the permutation $\pi$, inversion decoder~$I$ and noninversion decoder $N$ to be the graph $\Gamma_{IN}^{\pi}(w)$ with vertices $\{ 1, \dots, n \}$ and edges $ij$ for all $i<j$ such that 
\begin{itemize}
	\item $(w(i), w(j)) \in I$ if $\pi(i) > \pi(j)$, or
	\item $(w(i), w(j)) \in N$ if $\pi(i) < \pi(j)$.
\end{itemize}

We see, for example, the permutation letter graph $\Gamma_{IN}^{\pi}(w)$ given by $\pi = 253614$, $w = aabccb$, $I = \{ (a,b), (b,c) \}$ and $N = \{ (a,c), (b,c) \}$ drawn on the plot of $\pi$ in Figure~\ref{fig:perm_letter_graph_ex}.

\begin{figure}[h]
\begin{center}
	\begin{tikzpicture}[scale=0.6]
		
		\node [draw, circle, fill, minimum size = \nodesize] at (1,2) {};
		\node [] at (0.65,2) {$a$};
		
		\node [draw, circle, fill, minimum size = \nodesize]at (2,5) {};
		\node [] at (1.65,5) {$a$};
		
		\node [draw, circle, fill, minimum size = \nodesize] at (3,3) {};
		\node [] at (2.75,2.75) {$b$};
		
		\node [draw, circle, fill, minimum size = \nodesize] at (4,6) {};
		\node [] at (4.35,6) {$c$};
		
		\node [draw, circle, fill, minimum size = \nodesize] at (5,1) {};
		\node [] at (5.35,1) {$c$};
		
		\node [draw, circle, fill, minimum size = \nodesize] at (6,4) {};
		\node [] at (6.35,4) {$b$};
		
		\draw [] (2,5) to (3,3);
		\draw [] (2,5) to (6,4);
		\draw [] (3,3) to (5,1);
		
		\draw [] (1,2) to (4,6);
		\draw [] (2,5) to (4,6);
		\draw [] (3,3) to (4,6);
		
	\end{tikzpicture}
\end{center}
\caption{An example of a permutation letter graph $\Gamma_{IN}^{\pi}(w)$ drawn on the plot of~$\pi$.}
\label{fig:perm_letter_graph_ex}
\end{figure}
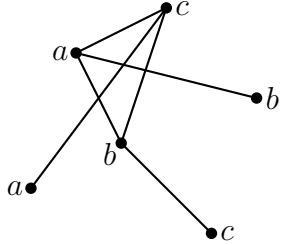

If $|\Sigma|=k$, then we say that $\Gamma_{IN}^{\pi}(w)$ is a \emph{permutation $k$-letter graph}, and we denote the least $k$ such that $G$ is a permutation $k$-letter graph by $\ell_{\perm}(G)$.

Just as lettericity is monotone with respect to the induced subgraph relation, so is the parameter~$\ell_{\perm}$. That is, given an $n$-vertex inversion letter graph $\Gamma_{IN}^{\pi}(w)$, for any vertex $i \in \{ 1, \dots, n\}$ we have that $\Gamma_{IN}^{\pi}(w)[1, \dots, i-1, i+1, \dots, n] = \Gamma_{IN}^{\pi'}(w')$ where~$\pi'$ and~$w'$ are obtained by deleting the entry at the $i^{\text{th}}$ index from both of $\pi$ and $w$, respectively. In other words, $\ell_{\perm}(H) \leq \ell_{\perm} (G)$ if $H$ is an induced subgraph of $G$.

Furthermore, the class of permutation $k$-letter graphs is also closed under complements. Indeed, for any permutation letter graph $\Gamma_{IN}^{\pi}(w)$ on the alphabet $\Sigma$, it follows that $\overline{\Gamma}_{IN}^{\pi}(w) = \Gamma_{\overline{I} \hspace{0.5mm} \overline{N}}^{\pi}(w)$ where $\overline{I} = \Sigma^2 \setminus I$ and $\overline{N} = \Sigma^2 \setminus N$. We see next that the permutation $1$-letter graphs are a very familiar class.

\begin{theorem}
	If $G$ is a graph on $n$ vertices, then $\ell_{\perm}(G) = 1$ if and only if $G$ is the inversion graph of a permutation.
\end{theorem}

\begin{proof}
	We let $\Sigma = \{a \}$ be our alphabet and we fix the set $A = \{(a,a) \}$. Then for any permutation~$\pi$, the only permutation letter graphs on this alphabet that can be drawn on~$\pi$ are then 
	\begin{itemize}
		\item $\Gamma_{\varnothing \varnothing}^{\pi}(aa \dots a) = \overline{K}_n$, (which is isomorphic to $G_{12 \dots n}$),
		\item $\Gamma_{A \varnothing}^{\pi}(aa \dots a) = G_{\pi}$, 
		\item $\Gamma_{\varnothing A}^{\pi}(aa \dots a) = \overline{G}_{\pi}$, (which is isomorphic to $G_{\pi^{\text{r}}}$), and 
		\item $\Gamma_{A A}^{\pi}(aa \dots a) = K_n$, (which is isomorphic to $G_{n \dots 21}$).
	\end{itemize}
	This gives the result.
\end{proof}

As we saw above, the classes of perfect matchings and paths have unbounded lettericity. With the added flexibility of being able to encode on a permutation we have $\ell_{\perm}(nK_2) = 1$ and $\ell_{\perm}(P_n) = 1$ for all positive integers $n$. Furthermore, the class of all cycles has unbounded lettericity, (this can be seen since it contains all paths), and we showed in the previous chapter that~$C_n$ is not an inversion graph for all $n \geq 5$. The drawing in Figure~\ref{fig:cycle_as_perm_letter_graph} makes it clear that $\ell_{\perm}(C_n) = 2$ for $n \geq 5$. Indeed, we have the permutation $2$-letter graph $\Gamma_{IN}^{\pi}(w) = C_n$ for all $n \geq 5$ where~$\pi$ is such that $G_{\pi} = P_n$, the word $w$ is such that the two entries of $\pi$ corresponding to the leaves in its inversion graph are encoded with $b$ and the rest of the entries are $a$, and with decoders $I = \{ (a,a), (a,b), (b,a)\}$ and $N = \{(b,b) \}$.

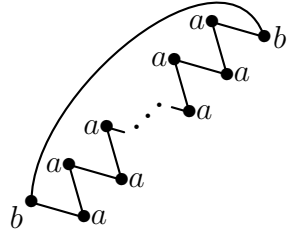
\begin{figure}[h]
\begin{center}
	\begin{tikzpicture}[scale=1]
		
		\absdot{(2.1, 1.4)};
		\node [] at (2.3,1.4) {$b$};
		
		\absdot{(1.4,1.6)};
		\node [] at (1.2,1.6) {$a$};
		
		\absdot{(1.6,0.9)};
		\node [] at (1.8,0.9) {$a$};
		
		\absdot{(0.9, 1.1)};
		\node [] at (0.7,1.1) {$a$};
		
		\absdot{(1.1, 0.4)};
		\node [] at (1.3,0.4) {$a$};
		
		
		\absdot{(0, 0.2)};
		\node [] at (-0.2,0.2) {$a$};
		
		\absdot{(0.2, -0.5)};
		\node [] at (0.4,-0.5) {$a$};
		
		\absdot{(-0.5, -0.3)};
		\node [] at (-0.7,-0.3) {$a$};
		
		\absdot{(-0.3, -1)};
		\node [] at (-0.1,-1) {$a$};
		
		\absdot{(-1, -0.8)};
		\node [] at (-1.2,-1) {$b$};
		
		\draw [thick, line cap=round] (2.1,1.4)--(1.4,1.6)--(1.6,0.9)--(0.9,1.1)--(1.1,0.4) --(0.8667,0.4667) ;
		
		\draw [thick, line cap=round] (-1, -0.8)--(-0.3, -1)--(-0.5, -0.3)--(0.2, -0.5)--(0, 0.2) -- (0.2333,0.1333);
		
		\draw [out = 90, in = 110] (-1,-0.8) to (2.1,1.4);
		
		\node (x) at (0.7, 0.5) {\Large$.$};
		\node (x) at (0.55, 0.35) {\Large$.$};
		\node (x) at (0.4, 0.2) {\Large$.$};
		
	\end{tikzpicture}
\end{center}
\caption{Encoding a cycle as a permutation letter graph.}
\label{fig:cycle_as_perm_letter_graph}
\end{figure}

\subsection{Extremal Considerations}

The next step is to investigate the maximum value of $\ell_{\perm}$ over all $n$-vertex graphs. We saw with letters graph that since almost all $n$-vertex graphs have lettericity near $n$, we can't expect to save much data at all by representing a graph $G$ as a $\ell(G)$-letter graph. That is, since almost every vertex receives its own letter, almost every possible decoder element decides one edge in the resulting graph. The question is now whether or not we can save data with the permutation letter graph construction. The following upper bound on $\ell_{\perm}(G)$ for all $n$-vertex graphs $G$ is obtained by finding a word and permutation on which each edge is determined by a unique inversion or noninversion decoder element, thus we do not save any data at all.

\begin{theorem}
	If $G$ is a graph on $n$ vertices, then $\ell_{\perm}(G) \leq \lceil \nicefrac{n}{2} \rceil$.
\end{theorem}

\begin{proof}
	We can encode every graph on $n = 2k$ vertices with the permutation $k \dots 321 \oplus k \dots 321$ with the word $\ell_k \dots \ell_3 \ell_2 \ell_1 \ell_1 \ell_2 \ell_3 \dots \ell_k$ as seen in Figure~\ref{fig-allgraphs}. For each pair of vertices on the first half of this word, the corresponding edge is uniquely determined by the presence of some tuple $(\ell_{j}, \ell_i)$, $i<j$, in the inversion decoder. Similarly, edges between vertices in the second half of the word are determined by tuples $(\ell_i, \ell_j)$, $i<j$, in the inversion encoder. Finally, the edges between the two halves are uniquely determined by tuples in the noninversion decoder.
	
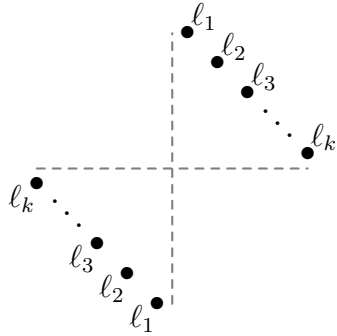
\begin{figure}[h]
\begin{center}
	\begin{tikzpicture}[scale=0.4]
		\absdot{(1,5)};
		\node at (1.6, 4.4) {\large $.$};
		\node at (2, 4) {\large $.$};
		\node at (2.4, 3.6) {\large $.$};
		\absdot{(3,3)};
		\absdot{(4,2)};
		\absdot{(5,1)};
		
		\absdot{(6,10)};
		\absdot{(7,9)};
		\absdot{(8,8)};
		\node at (8.6, 7.4) {\large $.$};
		\node at (9, 7) {\large $.$};
		\node at (9.4, 6.6) {\large $.$};
		\absdot{(10,6)};
		\draw [gray, dashed, thick, cap = round] (5.5,1)--(5.5,10);
		\draw [gray, dashed, thick, cap=round] (1,5.5)--(10,5.5);
		\node [label = {}] at (0.5,4.5) {$\ell_k$};
		\node [label = {}] at (2.5,2.5) {$\ell_3$};
		\node [label = {}] at (3.5,1.5) {$\ell_2$};
		\node [label = {}] at (4.5,0.5) {$\ell_1$};
		
		\node [label = {}] at (6.55,10.55) {$\ell_1$};
		\node [label = {}] at (7.55,9.55) {$\ell_2$};
		\node [label = {}] at (8.55,8.55) {$\ell_3$};
		\node [label = {}] at (10.55,6.55) {$\ell_k$};
	\end{tikzpicture}
\end{center}
\caption{A word $w$ and permutation $\pi$ such that for all graphs $G$ there exists a permutation letter graph $\Gamma_{IN}^{\pi}(w)$ isomorphic to $G$.}
\label{fig-allgraphs}
\end{figure}
	
	Similarly, it is easy to see that every graph on $n = 2k+1$ vertices is a permutation letter graph with the permutation $(k+1)k \dots 21 \oplus k \dots 21$ with the word $\ell_{k+1} \ell_k \dots \ell_2 \ell_1 \ell_1 \ell_2 \dots \ell_k$.
\end{proof}

It doesn't take too much work to show that this bound is tight. Interestingly, despite being able to save many letters when encoding classes such as paths, matchings and cycles with the permutation letter graph construction over the letter graph construction, this means that there exists graphs for which we can save no data at all with the permutation letter graph construction, (Theorem~\ref{mainconstruction} showed that we can always save at least some data when encoding a graph as a letter graph). To prove this upper bound, we use the same exact type of argument as Petkov\v{s}ek~\cite{petkovsek:letter-graphs-a:} used to show that for large enough $n$ that there exists an $n$-vertex graph $G$ such that $\ell(G)$ is at least $0.707n$. The author thanks Vince Vatter for recommending this argument.

\begin{theorem}
	For each $\alpha < \nicefrac{1}{2}$, there is an $N$ such that for all $n > N$ there are $n$-vertex graphs~$G$ with $\ell_{\perm}(G) > \alpha n$.
\end{theorem}

\begin{proof}
	Assume that $\ell_{\perm}(G) \le \alpha n$ for all graphs $G$ on $n$ vertices and write $k = \lfloor \alpha n \rfloor$. Then, by assumption, all $2^{\binom{n}{2}}$ labeled graphs on $n$ vertices are permutation $k$-letter graphs.
	
	Over a $k$-letter alphabet, there are $k^2$ pairs of letters, $2^{2k^2}$ ways to choose the decoders $I$ and $N$, $k^n$ words of length $n$, $n!$ choices for the underlying permutation, and at most $n!$ possible labelings of a graph on $n$ vertices. Hence, there are no more than $(n!)^2 \cdot k^n \cdot 2^{2k^2}$ labeled $k$-letter graphs on $n$ vertices. Therefore
	\[
	2^{\binom{n}{2}} \le (n!)^2 \cdot k^n \cdot 2^{2k^2} \le n^{2n} \cdot (\alpha n)^n \cdot 2^{2(\alpha n)^2}.
	\] 
	After taking logarithms and making some small manipulations, we obtain
	\[
	(\nicefrac{1}{2} - 2 \alpha^2) \cdot n^2 \leq 3n \cdot \log_2n + (\nicefrac{1}{2} + \log_2\alpha ) \cdot n.
	\]
	However, since $\alpha < \nicefrac{1}{2}$ we have that $\nicefrac{1}{2} - 2 \alpha^2$ is positive, and so this inequality doesn't hold for large $n$.
\end{proof}

\subsection{Further Directions}

We now have a much tighter extremal bound on $\ell_{\perm}$ for all $n$-vertex graphs than we do for lettericity. However, our lower bound on lettericity given in Theorem~\ref{thm:lettericity_lb} describes the behavior of almost all graphs. This inspires the following problem.

\begin{problem}
	Find the optimal function $f(n)$ such that $\ell_{\perm}(G) \geq f(n)$ for almost all $n$-vertex graphs $G$.
\end{problem}

This permutation letter graph construction leads to many structural problems which can be phrased along the lines of the following.

\begin{problem}
	For fixed decoders $I$ and $N$ on a small alphabet $\Sigma$, classify all graphs $\Gamma_{IN}^{\pi}(w)$ spanning over all $w$ and $\pi$. For a more controlled problem, fix further some permutation $\pi \neq e^{\text{r}}$ and let only $w$ vary, or vice versa.
\end{problem}

The following problem suggests another graph parameter that lies somewhere between lettericity and~$\ell_{\perm}$.

\begin{problem}
	Consider the graph parameter defined by the least number of letters required to encode a graph as a permutation letter graph $\Gamma_{IN}^{\pi}(w)$ such that $N$ is necessarily empty.
\end{problem}

%% file: chap-reflections.tex
\chapter{Reflections in Graphs}
\label{chap:reflections}

\newcommand{\flipOmega}{\rotatebox[origin=c]{180}{$\Omega$}}
\newcommand{\flipPsi}{\rotatebox[origin=c]{180}{$\Psi$}}
\def\nodesize{3.5}

Fundamental to algebraic combinatorics is the systematic study of symmetry via Coxeter groups. The structure of a Coxeter group is captured in its Bruhat graph, with arcs representing the reflections that generate its elements. Permutations, with transpositions playing the role of reflections, form a Coxeter system. In this chapter, we begin by studying how transpositions interact with the inversion graph representation of permutations, allowing us to generalize these reflections to an operation on arbitrary simple graphs. This enables us to define a graph on the set of all $n$-vertex graphs that is analogous to the Bruhat graph on permutations. We investigate this graph throughout this chapter.

\section{Preliminaries}

\subsection{Reflections in Permutations}

The set of permutations $S_n$ with the operation of function composition forms a group called the \emph{symmetric group}~\cite{sagan:the-symmetric-g:}. The symmetric group is generated by the set of \emph{simple reflections} $S = \{s_1, s_2, \dots, s_{n-1}\}$, where~$s_i = (i \hspace{1mm} (i\mathrm{+}1)) \in S_n$ is the adjacent transposition that swaps $i$ and $i+1$ and fixes all other elements. The symmetric group with these generators form a Coxeter system~\cite{bjorner:combinatorics-o:}.

In a general Coxeter system, the set of all \emph{reflections} is the set of all conjugates of the simple reflections. Thus, in the symmetric group, the set of reflections is given by 
\[
	T = \bigcup_{\pi \in S_n} \pi S \pi^{-1} = \{t_{ij}: 1 \leq i < j \leq n \},
\]
where $t_{ij} = (i \hspace{1mm} j) \in S_n$ is the transposition swapping $i$ and $j$ and leaving all other elements fixed. That is, the reflections in $S_n$ are exactly the transpositions. Applying the reflection~$t_{ij}$ on the left of a permutation $\pi \in S_n$ results in a permutation with the values~$i$ and~$j$ swapped in one-line notation. Furthermore, applying $t_{ij}$ on the right swaps the indices $i$ and $j$. For example, given $\pi = 24315$ we have $t_{1,4} \circ \pi = 2 \mathbf{1}3\mathbf{4}5$ and $\pi \circ t_{1,4} = \mathbf{1}43\mathbf{2}5$. 

Combining the notions of reflection and inversion, the \emph{Bruhat graph} $\Omega_n$ is the directed graph with vertex set $S_n$, and in which $(\pi, \sigma)$ is an arc if and only if $\sigma = \pi\circ t_{ij}$ for some reflection $t_{ij}$ such that $(i,j)$ is an inversion of $\sigma$. That is, $\sigma$ is obtained from $\pi$ by applying a reflection that creates more inversions. Similarly, the \emph{weak Bruhat graph} $\Psi_n$ is the directed graph with vertex set $S_n$, and in which $(\pi, \sigma)$ is an arc if and only if $\sigma = \pi \circ s_i$ for some simple reflection $s_i$ such that $(i, i+1)$ is an inversion of $\sigma$. It is clear that~$\Psi_n$ is a subgraph of $\Omega_n$. We remark that the (strong) \emph{Bruhat order} and \emph{weak Bruhat order}  on~$S_n$ are obtained by taking the transitive closures of the arc sets of $\Omega_n$ and $\Psi_n$, respectively, (and taking $x \rightarrow y$ to mean $x \leq y$).

Coxeter systems employ reflections to generate the elements of the underlying Coxeter group. This notion of generating elements corresponds with the paths that begin at the identity element and travel along the arcs of the Bruhat graph. For the symmetric group~$S_n$, we can equivalently consider sorting permutations, starting at an element and traveling backwards along the arcs of the Bruhat graph until we reach the identity. This is the point of view taken throughout this chapter, and thus we consider the \emph{flipped Bruhat graph} $\flipOmega_n$ and \emph{flipped weak Bruhat graph} $\flipPsi_n$, obtained by taking the converse, (i.e., dual), of the graphs $\Omega_n$ and $\Psi_n$, respectively. In other words, we reverse the arcs in both of these graphs. In either of the flipped Bruhat graphs, an arc $(\pi, \sigma)$ implies that there is an inversion in $\pi$ whose indices swap values to obtain $\sigma$. We refer to a reflection that swaps two values of an inversion as a \emph{reduction}; this is equivalent to defining a reduction as a reflection that reduces length. The graphs $\flipOmega_3$ and $\flipPsi_3$ can be seen in Figure~\ref{fig:BruhatGraphs}.

\begin{figure}[h]
	\begin{center}
	\begin{pic}
	\makenodescircle{v}{6}{0.5}[90][circle, fill, minimum size = \nodesize pt, draw]
	\labelnode[180]{v3}{$213$}
	\labelnode[180]{v2}{$231$}
	\labelnode[0]{v5}{$132$}
	\labelnode[0]{v6}{$312$}
	\labelnode[90]{v1}{$321$}
	\labelnode[-90]{v4}{$123$}
	
	\draw[->-] (v1) to (v2);
	\draw[->-] (v1) to (v6);
	\draw[->-] (v2) to (v3);
	\draw[->-] (v6) to (v5);
	\draw[->-] (v5) to (v4);
	\draw[->-] (v3) to (v4);
	
	\draw[-->-, dashed] (v1) to (v4);
	\draw[-->-, dashed] (v2) to (v5);
	\draw[-->-, dashed] (v6) to (v3);
	
	\end{pic}
\end{center}
\vspace{-2mm}
\caption{The flipped Bruhat graph $\flipOmega_3$; the flipped weak Bruhat graph $\flipPsi_3$ can be obtained by removing the dashed arcs.}	
\label{fig:BruhatGraphs}
\end{figure}

The \emph{length} of a permutation $\pi \in S_n$ is given by $l(\pi) = \min \{ k : \pi = s_{i_1} s_{i_2} \dots s_{i_k} \}$, i.e., the minimum number of factors required to write $\pi$ as a product of simple reflections. We note that these minimum length expressions are often referred to as \emph{reduced words} or \emph{reduced expressions}. Given that a simple reflection either increases or decreases the number of inversions of a permutation by 1, it follows that $l(\pi)$ is the number of inversions of $\pi$, and corresponds with the length of any path 
\[
\pi = \sigma_0 \rightarrow \sigma_1 \rightarrow \sigma_{2} \rightarrow \dots \rightarrow \sigma_{m} = e
\]
in the flipped weak Bruhat graph $\flipPsi_n$, where $e$ is the identity permutation.

Similarly, the \emph{absolute length} of a permutation $\pi \in S_n$ is given by $l'(\pi) = \min \{ k: \pi = t_{i_1 j_1} t_{i_2 j_2} \dots t_{i_k j_k} \}$, i.e., the minimum number of factors required to write $\pi$ as a product of reflections. Unlike the situation above, it is not immediate that $l'(\pi)$ corresponds with the minimum length of a path 
\[
\pi = \sigma_0 \rightarrow \sigma_1 \rightarrow \sigma_{2} \rightarrow \dots \rightarrow \sigma_{m'} = e
\]
in the flipped Bruhat graph $\flipOmega_n$. That is, it is not obvious that there isn't a shorter path from~$\pi$ to~$e$ in the \emph{undirected Bruhat graph} $\bar{\Omega}_n (= \bar{\flipOmega}_n)$, obtained by unorienting the arcs of~$\Omega_n$. The following confirms what one might hope to be true. 

\begin{theorem}[Dyer~\text{\cite[Theorem 1.2]{dyer:on-minimal-leng:}}]\label{thm:abs_length_in_bruhat_graphs}
	For any $\pi \in S_n$, the minimum length of a path from~$\pi$ to~$e$ in~$\flipOmega_n$ is equal to~$l'(\pi)$.
\end{theorem}

In other words, any permutation can be optimally sorted by reflections using only reductions, (although sorting by reductions is not always optimal). Dyer~\cite[Theorem 1.2]{dyer:on-minimal-leng:} establishes Theorem~\ref{thm:abs_length_in_bruhat_graphs} by proving a more general result for all Coxeter systems using some extra machinery, but it is elementary to prove for the symmetric group. That is, one can show that applying a reflection to a permutation either splits one cycle into two, or joins two cycles into one. The result then follows from the observation that every nonidentity permutation has an inversion such that swapping its values splits a cycle. From this also follows the well-known fact that the absolute length of a permutation is equal to its number of entries minus the number of cycles.

The set of permutations~$\pi$ for which $l(\pi)=l'(\pi)$ were characterized by both Edelman~\cite{edelman:on-inversions-a:} and Petersen and Tenner~\cite{petersen:the-depth-of-a-:}. Edelman's characterization is in terms of the cycle structure of these permutations. More specifically, he describes them as those permutations for which each cycle is comprised of the elements of an interval of $[n]$ and can be written as $(c_1, c_2, \dots, c_k)$ where there exists $i \in [k]$ such that $c_1 < c_2 < \dots < c_i > \dots > c_k$. The characterization given by Petersen and Tenner, which is more relevant to this paper, is given below.

\begin{theorem}[Petersen and Tenner~\cite{petersen:the-depth-of-a-:}]
\label{thm:boolean_perms}
The length $l(\pi)$ and absolute length $l'(\pi)$ of a permutation $\pi \in S_n$ coincide if and only if~$\pi$  avoids the patterns $321$ and $3412$.
\end{theorem}

Prior to the proof of Theorem~\ref{thm:boolean_perms}, Tenner~\cite{tenner:pattern-avoidan:} had established a different characterization of the class of permutations avoiding $321$ and $3412$. She proved that these are the \emph{boolean permutations}, meaning that their principal order ideals in the Bruhat order are isomorphic to boolean lattices. In proving this, Tenner observed that these are also precisely the permutations whose reduced words contain no repeated letters. This result is how Petersen and Tenner~\cite{petersen:the-depth-of-a-:} proved Theorem~\ref{thm:boolean_perms}.

Despite all this interest in the class of permutations avoiding $321$ and $3412$, one natural characterization of these permutations seems not to have been observed in the literature. Using this characterization, a simple graph-theoretic proof of Theorem~\ref{thm:boolean_perms} will follow from our results. Given the topic of this dissertation, the reader may have anticipated that this characterization is in terms of inversion graphs.

\begin{proposition}
\label{prop:321:3412:forest}
A permutation avoids $321$ and $3412$ if and only if its inversion graph is a forest.
\end{proposition}

\begin{proof}
	A quick investigation reveals that $321$ is the only permutation that yields a $C_3$ as its inversion graph, and $3412$ is the only permutation that yields a $C_4$. Further, we saw in Chapter~\ref{chap:uniqueness} that there are no permutations that yield a cycle $C_n$ for any $n \geq 5$.
\end{proof}

As mentioned above, we use this characterization to give a graph-theoretic proof of Theorem~\ref{thm:boolean_perms}. We discuss in the paragraph following Theorem~\ref{thm-forests-m-transpositions} how Theorem~\ref{thm:boolean_perms} is an easy corollary to that result.

For easy reference, we compile all of the mentioned characterizations of this permutation class in the following list. That is, letting~$\pi$ be a permutation, the following statements are equivalent:
\begin{itemize}
	\item its length $l(\pi)$ and absolute length $l'(\pi)$ coincide;
	\item it avoids the patterns $231$ and $3412$;
	\item each of its cycles comprises an interval and can be written as $(c_1, c_2, \dots, c_k)$ with $i \in [k]$ such that $c_1 < c_2 < \dots < c_i > \dots > c_k$;
	\item its principal order ideal in Bruhat order is isomorphic to a boolean lattice;
	\item all of its reduced words contain no repeated letters; and
	\item its inversion graph $G_{\pi}$ is a forest.
\end{itemize}

\subsection{From Permutations to Graphs}

While it is trivial how the application of a reflection changes a permutation written in one-line notation, it is not so obvious how reflections affect the underlying inversion graph. In Figure~\ref{fig-inversion_graph_transposition}, we see an example of how a reduction manipulates a permutation's inversion graph. Specifically, we see the inversion graph of $651324$ before and after the reflection $t_{2,5}$ is applied, (on either the left or right).

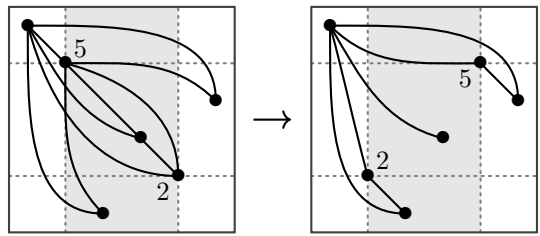
\begin{figure}[h]
\begin{footnotesize}
\begin{center}
	\begin{tikzpicture}[scale=0.5]
		\draw [draw = none, fill=gray!20] (2,0.5) rectangle (5,6.5);
		\draw [color = gray, dotted, thick, line cap=round] (2,0.5)--(2,6.5);
		\draw [color = gray, dotted, thick, line cap=round] (5,0.5)--(5,6.5);
		\draw [color = gray, dotted, thick, line cap=round] (0.5,2)--(6.5,2);
		\draw [color = gray, dotted, thick, line cap=round] (0.5,5)--(6.5,5);
		\draw [out=0, in=90] (1,6) to (6,4);
		\draw [out=-45, in=135] (1,6) to (2,5);
		\draw [out=-60, in=165] (1,6) to (4,3);
		\draw [out=-75, in=180] (1,6) to (5,2);
		\draw [out=-90, in=180] (1,6) to (3,1);
		\draw (2,5) -- (4,3) -- (5,2);
		\draw [out=-15, in=90] (2,5) to (5,2);
		\draw [out=0, in=135] (2,5) to (6,4);
		\draw [out=270, in=135] (2,5) to (3,1);
		
		\plotperm{6,5,1,3,2,4};
		\plotpermbox{1}{1}{6}{6};
		
		\node[rotate = 0] at (2.4, 5.4) {$5$};
		\node[rotate = 0] at (4.6, 1.6) {$2$};
		
		\draw[->] (7,3.5) to (8,3.5);
		
	\end{tikzpicture}
	\hspace{1mm}
	\begin{tikzpicture}[scale=0.5]
		\draw [draw = none, fill=gray!20] (2,0.5) rectangle (5,6.5);
		\draw [color = gray, dotted, thick, line cap=round] (2,0.5)--(2,6.5);
		\draw [color = gray, dotted, thick, line cap=round] (5,0.5)--(5,6.5);
		\draw [color = gray, dotted, thick, line cap=round] (0.5,2)--(6.5,2);
		\draw [color = gray, dotted, thick, line cap=round] (0.5,5)--(6.5,5);
		\draw [out=0, in=90] (1,6) to (6,4);
		\draw [out=-45, in=180] (1,6) to (5,5);
		\draw [out=-60, in=165] (1,6) to (4,3);
		\draw [out=-75, in=105] (1,6) to (2,2);
		\draw [out=-90, in=180] (1,6) to (3,1);
		\draw (2,2) -- (3,1);
		\draw (5,5) -- (6,4);
		
		\plotperm{6,2,1,3,5,4};
		\plotpermbox{1}{1}{6}{6};
		
		\node[rotate = 0] at (4.6, 4.6) {$5$};
		\node[rotate = 0] at (2.4, 2.4) {$2$};
	\end{tikzpicture}
\end{center}
\end{footnotesize}
\caption{A permutation's inversion graph before and after a reflection.}
\label{fig-inversion_graph_transposition}
\end{figure}

To better understand what is happening to the inversion graph in this figure, we can look individually at the different cases of where the other entries are in relation to the inversion we are reflecting. All such cases can be seen in Figure~\ref{fig-inversion_graph_cases}, either by the exact relative position  to the inversion or a 180$^{\circ}$ rotation of the plot.

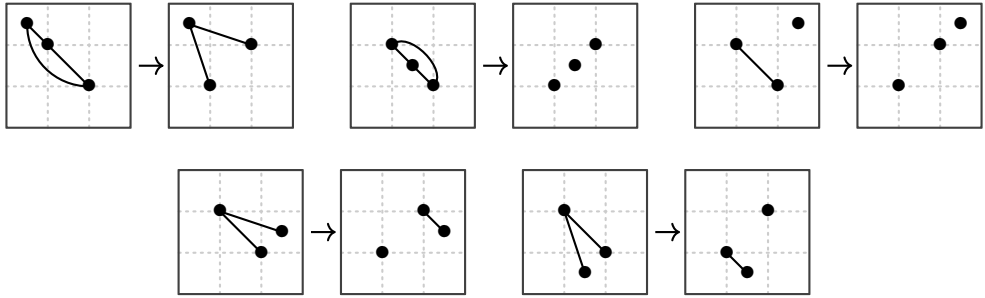
\begin{figure}[h]
\begin{footnotesize}
\begin{center}
	\begin{tikzpicture}[scale=0.55]
		\draw [color = lightgray , dotted, thick, line cap=round] (1,0)--(1,3);
		\draw [color = lightgray , dotted, thick, line cap=round] (2,0)--(2,3);
		\draw [color = lightgray , dotted, thick, line cap=round] (0,1)--(3,1);
		\draw [color = lightgray , dotted, thick, line cap=round] (0,2)--(3,2);
		\draw (0.5, 2.5) -- (1,2);
		\draw [out=-90, in=180] (0.5,2.5) to (2,1);
		
		\plotpermgraph{2,1};
		\plotpermbox{0.5}{0.5}{2.5}{2.5};
		\absdot{(0.5,2.5)};
		\draw[->] (3.2,1.5) to (3.8,1.5);
	\end{tikzpicture}
	\hspace{-1mm}
	\begin{tikzpicture}[scale=0.55]
		\draw [color = lightgray , dotted, thick, line cap=round] (1,0)--(1,3);
		\draw [color = lightgray , dotted, thick, line cap=round] (2,0)--(2,3);
		\draw [color = lightgray , dotted, thick, line cap=round] (0,1)--(3,1);
		\draw [color = lightgray , dotted, thick, line cap=round] (0,2)--(3,2);
		\draw (1,1) -- (0.5, 2.5) -- (2,2);
		\plotperm{1,2};
		\plotpermbox{0.5}{0.5}{2.5}{2.5};
		\absdot{(0.5,2.5)};
	\end{tikzpicture}
	\hspace{6mm} 
	\begin{tikzpicture}[scale=0.55]
		\draw [color = lightgray , dotted, thick, line cap=round] (1,0)--(1,3);
		\draw [color = lightgray , dotted, thick, line cap=round] (2,0)--(2,3);
		\draw [color = lightgray , dotted, thick, line cap=round] (0,1)--(3,1);
		\draw [color = lightgray , dotted, thick, line cap=round] (0,2)--(3,2);
		\draw [out=45, in=45] (1,2) to (2,1);
		\draw (1,2) -- (1.5,1.5) -- (2,1);
		\plotperm{2,1};
		\plotpermbox{0.5}{0.5}{2.5}{2.5};
		\absdot{(1.5,1.5)};
		\draw[->] (3.2,1.5) to (3.8,1.5);
	\end{tikzpicture}
	\hspace{-1mm}
	\begin{tikzpicture}[scale=0.55]
		\draw [color = lightgray , dotted, thick, line cap=round] (1,0)--(1,3);
		\draw [color = lightgray , dotted, thick, line cap=round] (2,0)--(2,3);
		\draw [color = lightgray , dotted, thick, line cap=round] (0,1)--(3,1);
		\draw [color = lightgray , dotted, thick, line cap=round] (0,2)--(3,2);
		\plotperm{1,2};
		\plotpermbox{0.5}{0.5}{2.5}{2.5};
		\absdot{(1.5,1.5)};
	\end{tikzpicture}
	\hspace{6mm} 
	\begin{tikzpicture}[scale=0.55]
		\draw [color = lightgray , dotted, thick, line cap=round] (1,0)--(1,3);
		\draw [color = lightgray , dotted, thick, line cap=round] (2,0)--(2,3);
		\draw [color = lightgray , dotted, thick, line cap=round] (0,1)--(3,1);
		\draw [color = lightgray , dotted, thick, line cap=round] (0,2)--(3,2);
		\absdot{(2.5,2.5)};
		\plotpermgraph{2,1};
		\plotpermbox{0.5}{0.5}{2.5}{2.5};
		
		\draw[->] (3.2,1.5) to (3.8,1.5);
	\end{tikzpicture}
	\hspace{-1mm}
	\begin{tikzpicture}[scale=0.55]
		\draw [color = lightgray , dotted, thick, line cap=round] (1,0)--(1,3);
		\draw [color = lightgray , dotted, thick, line cap=round] (2,0)--(2,3);
		\draw [color = lightgray , dotted, thick, line cap=round] (0,1)--(3,1);
		\draw [color = lightgray , dotted, thick, line cap=round] (0,2)--(3,2);
		\absdot{(2.5,2.5)};
		\plotperm{1,2};
		\plotpermbox{0.5}{0.5}{2.5}{2.5};
	\end{tikzpicture}
	
	\vspace{5mm} 
	\begin{tikzpicture}[scale=0.55]
		\draw [color = lightgray , dotted, thick, line cap=round] (1,0)--(1,3);
		\draw [color = lightgray , dotted, thick, line cap=round] (2,0)--(2,3);
		\draw [color = lightgray , dotted, thick, line cap=round] (0,1)--(3,1);
		\draw [color = lightgray , dotted, thick, line cap=round] (0,2)--(3,2);
		\draw (1,2) -- (2.5,1.5);
		\absdot{(2.5,1.5)};
		\plotpermgraph{2,1};
		\plotpermbox{0.5}{0.5}{2.5}{2.5};
		
		\draw[->] (3.2,1.5) to (3.8,1.5);
	\end{tikzpicture}
	\hspace{-1mm}
	\begin{tikzpicture}[scale=0.55]
		\draw [color = lightgray , dotted, thick, line cap=round] (1,0)--(1,3);
		\draw [color = lightgray , dotted, thick, line cap=round] (2,0)--(2,3);
		\draw [color = lightgray , dotted, thick, line cap=round] (0,1)--(3,1);
		\draw [color = lightgray , dotted, thick, line cap=round] (0,2)--(3,2);
		\draw (2,2) -- (2.5,1.5);
		\absdot{(2.5,1.5)};
		\plotperm{1,2};
		\plotpermbox{0.5}{0.5}{2.5}{2.5};
	\end{tikzpicture}
	\hspace{6mm} 
	\begin{tikzpicture}[scale=0.55]
		\draw [color = lightgray , dotted, thick, line cap=round] (1,0)--(1,3);
		\draw [color = lightgray , dotted, thick, line cap=round] (2,0)--(2,3);
		\draw [color = lightgray , dotted, thick, line cap=round] (0,1)--(3,1);
		\draw [color = lightgray , dotted, thick, line cap=round] (0,2)--(3,2);
		\draw (1,2) -- (1.5,0.5);
		\absdot{(1.5,0.5)};
		\plotpermgraph{2,1};
		\plotpermbox{0.5}{0.5}{2.5}{2.5};
		
		\draw[->] (3.2,1.5) to (3.8,1.5);
	\end{tikzpicture}
	\hspace{-1mm}
	\begin{tikzpicture}[scale=0.55]
		\draw [color = lightgray , dotted, thick, line cap=round] (1,0)--(1,3);
		\draw [color = lightgray , dotted, thick, line cap=round] (2,0)--(2,3);
		\draw [color = lightgray , dotted, thick, line cap=round] (0,1)--(3,1);
		\draw [color = lightgray , dotted, thick, line cap=round] (0,2)--(3,2);
		\draw (1,1) -- (1.5,0.5);
		\absdot{(1.5,0.5)};
		
		\plotperm{1,2};
		\plotpermbox{0.5}{0.5}{2.5}{2.5};
	\end{tikzpicture}
\end{center}
\end{footnotesize}
\caption{The different effects of a reflection on a permutation's inversion graph.}
\label{fig-inversion_graph_cases}
\end{figure}

In light of these figures, we can summarize the effects of a reflection on an inversion graph in a relatively simple graph-theoretic manner. That is, given the permutations~$\pi$ and  $\sigma = t_{k \ell} \circ \pi$, we can obtain $G_{\sigma}$ from $G_{\pi}$ by deleting the edge $k \ell$ and toggling every edge between $\{ k, \ell \}$ and the vertices lying horizontally between them in the plot. For example, in Figure~\ref{fig-inversion_graph_transposition} we delete the edge connecting~$2$ and~$5$ and toggle all of the edges between $\{2,5\}$ and the vertices lying in the shaded gray region. This suggests the following graph operations which we investigate throughout this chapter.

An \emph{edge reflection} on a graph $G$, the graphical analog of a reduction, consists of the following steps:
\begin{enumerate}
	\item Choose two vertices $u,v \in V(G)$ with $uv \notin E(G)$.
	\item Choose some subset $X$ of $\overline{N(u) \cap N(v)}$ containing $u$ and $v$.
	\item Toggle all edges between $\{u,v\}$ and $X$ (including adding the edge $uv$).
\end{enumerate}
We will denote this reflection by $t_{uv}^X$, and thus the resulting graph by $t_{uv}^X \circ G$. This operation is a generalization of a reduction in the sense that if $G_{\pi}$ is the inversion graph of a permutation $\pi \in S_n$, and $\sigma$ is another permutation that can be obtained by applying some reduction $t_{k \ell}$ on the left of~$\pi$, then there is some set $X$ for which $t_{k \ell}^X \circ G_{\pi} = G_{\sigma}$. Note, however, that there are often many more possible edge reflections on an inversion graph than there are reductions of the corresponding permutation.

Since reflections are involutions, the inverse of a reflection is itself. However, after applying a reduction to a permutation, applying the same reflection again is the opposite of a reduction. Thus, we need to define the graphical analog of a nonreduction reflection. So let a \emph{nonedge reflection} on a graph $G$ consist of the following steps:
\begin{enumerate}
	\item Choose two vertices $u,v \in V(G)$ with $uv \notin E(G)$.
	\item Choose some subset $X$ of $\overline{N(u) \cap N(v)}$ containing $u$ and $v$.
	\item Toggle all edges between $\{u,v\}$ and $X$ (including adding the edge $uv$).
\end{enumerate}
By also denoting this operation with $t_{uv}^X$, we now have that any reflection on a graph is an involution. 

It is now clear how to construct a graph analogous to the flipped Bruhat graph $\flipOmega_n$ for these graphical reflections. That is, define the \emph{reflection graph} $\Gamma_n$ to be the directed graph with vertex set $\mathscr{G}_n$ of all graphs on vertex set $[n]$, and such that $(G,H)$ is an arc if and only if $H$ can be obtained by applying an edge reflection to~$G$. By identifying a permutation with its inversion graph, the above discussions show that $\flipOmega_n$ is a subgraph of $\Gamma_n$.

Analogous to the question of absolute length in permutations we can ask, given a graph $G \in \mathscr{G}_n$ and noting that the inversion graph $G_e$ of the identity permutation $e \in S_n$ is the edgeless graph $\overline{K}_n$, what is the length of a shortest path
\[
G = G_0 \rightarrow G_1 \rightarrow \dots \rightarrow G_k = \overline{K}_n
\]
in the reflection graph $\Gamma_n$? That is, how many edge reflections are required to empty some graph of all its edges? These are the questions that we look to answer throughout this chapter, summarizing our situation in the following table.
\begin{center}
\small
\begin{tabular}{|l|l|l|l|}
\hline
Operation on&No. needed to&Operation on&No. needed to\\
permutations&reach identity&graphs&delete all edges\\
\hline & & & \\[-10pt]
Simple reflection&Length&Toggle an edge&No. edges\\
&\quad$=$ No. inversions & & \\[4pt]
&\quad\rotatebox[origin=c]{270}{$\ge$}&&\quad\rotatebox[origin=c]{270}{$\ge$}\\[4pt]
Reduction&Absolute length&Edge reflection&\textit{Studied here}\\
&\quad (by Theorem~\ref{thm:abs_length_in_bruhat_graphs}) & & \\[4pt]
\quad\rotatebox[origin=c]{270}{$\subseteq$}&\quad\rotatebox[origin=c]{270}{$=$}
	&\quad\rotatebox[origin=c]{270}{$\subseteq$}&\quad\rotatebox[origin=c]{270}{$\ge$}\\[4pt]
Reflection&Absolute length&Edge or nonedge&\quad ?\\
&\quad (by definition)&\quad reflection & \\
\hline
\end{tabular}
\end{center}

We begin with a result that follows easily from the above observation that $\flipOmega_n$ is a subgraph of $\Gamma_n$.

\begin{theorem}\label{thm:edge_reflections_abs_length}
	For a permutation $\pi \in S_n$, the number of edge reflections required to empty its inversion graph $G_{\pi}$ of all its edges is at most its absolute length $l'(\pi)$.
\end{theorem}

This bound is tight in come cases, but certainly not in all. For example, Theorem~\ref{thm-forests-m-transpositions} shows that Theorem~\ref{thm:edge_reflections_abs_length} is tight for the above discussed permutations avoiding $321$ and $3412$. However, the permutation $3421$ has absolute length 3, yet its inversion graph $G_{3421}$, a diamond, is isomorphic to the nested triangle $N_2$, (defined at the end of the next section), and requires 1 edge reflection to reach the edgeless graph.

Given that these graphical reflections are not dependent on vertex labelings, it is clear that the number of edges required to empty a graph of its edges is a graph invariant. However, it is interesting to note that the absolute length of permutations with the same inversion graph is not necessarily the same. For example, we have $G_{3421} \cong G_{4231}$ with $l'(3421) \neq l'(4231)$.

Analogous to Theorem~\ref{thm:abs_length_in_bruhat_graphs}, we can also ask if there are shorter paths from a graph~$G$ to~$\overline{K}_n$ in the \emph{undirected reflection graph} $\bar{\Gamma}_n$, obtained by removing all of the orientations of the arcs in $\Gamma_n$. That is, can we empty some graph of all its edges with fewer reflections if we also allow nonedge reflections? In contrast with the Dyer's result, it turns out the answer is `yes.' We see in Figure~\ref{fig:P6_counterX} that $P_6$, the path on six vertices, is a small example. 

\begin{figure}[ht]
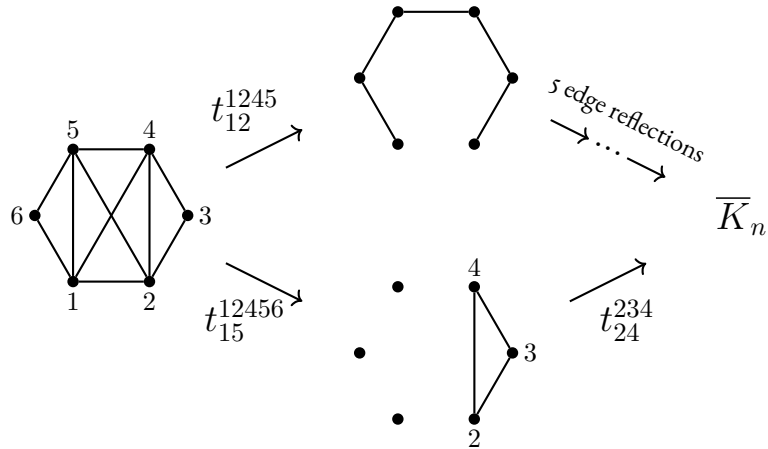

\begin{center}
	\begin{pic}
	\makenodescircle{v}{6}{0.4}[0][circle, fill, minimum size = \nodesize pt, draw]
	\drawedges{v1/v2, v2/v3, v3/v4, v4/v5, v5/v6, v6/v1}
	\drawjoin{v5,v6}{v2,v3}
	\labelnode[90]{v3}{$5$}
	\labelnode[90]{v2}{$4$}
	\labelnode[-90]{v5}{$1$}
	\labelnode[-90]{v6}{$2$}
	\labelnode[0]{v1}{$3$}
	\labelnode[180]{v4}{$6$}
	
	\node[rotate = 0] (pathref) at (0.7, 0.55) {\large $t_{12}^{1245}$};
	\draw[->] (0.6, 0.25) to (1, 0.45);
	
	\node[rotate = 0] (k3ref) at (0.7, -0.55) {\large $t_{15}^{12456}$};
	\draw[->] (0.6, -0.25) to (1, -0.45);

	\node [rotate = -26.565] (5er) at (2.7, 0.5) {$\text{5 edge reflections}$}; 
	\draw[->] (2.3, 0.5) to (2.5, 0.4);
	\node (dot2) at (2.55, 0.375) {\large $.$};
	\node (dot2) at (2.6, 0.35) {\large $.$};
	\node (dot2) at (2.65, 0.325) {\large $.$};
	\draw[->] (2.7, 0.3) to (2.9, 0.2);
	 
	\node[rotate = 0] (bottomrightref) at (2.7, -0.55) {\large $t_{24}^{234}$};
	\draw[->] (2.4, -0.45) to (2.8, -0.25);
	
	\def\heightadj{0.28} 
	\def\rightadj{0.2}   
	\def\nodesz{3.5}
	
	\node (v1_u) at (1.9 + \rightadj, 1 - \heightadj) [circle, fill, minimum size = \nodesize pt, draw]{};
	\node (v2_u) at (1.7 + \rightadj, 1.34641 - \heightadj) [circle, fill, minimum size = \nodesize pt, draw]{};
	\node (v3_u) at (1.3 + \rightadj, 1.34641 - \heightadj) [circle, fill, minimum size = \nodesize pt, draw]{};
	\node (v4_u) at (1.1 + \rightadj, 1 - \heightadj) [circle, fill, minimum size = \nodesize pt, draw]{};
	\node (v5_u) at (1.3 + \rightadj, 0.65359 - \heightadj) [circle, fill, minimum size = \nodesize pt, draw]{};
	\node (v6_u) at (1.7 + \rightadj, 0.65359 - \heightadj) [circle, fill, minimum size = \nodesize pt, draw]{};
	\drawedges{v6_u/v1_u, v1_u/v2_u, v2_u/v3_u, v3_u/v4_u, v4_u/v5_u}
	
	\node (v1_d) at (1.9 + \rightadj, -1 + \heightadj) [circle, fill, minimum size = \nodesize pt, draw]{};
	\node (v2_d) at (1.7 + \rightadj, -0.65359 + \heightadj) [circle, fill, minimum size = \nodesize pt, draw]{};
	\node (v3_d) at (1.3 + \rightadj, -0.65359 + \heightadj) [circle, fill, minimum size = \nodesize pt, draw]{};
	\node (v4_d) at (1.1 + \rightadj, -1 + \heightadj) [circle, fill, minimum size = \nodesize pt, draw]{};
	\node (v5_d) at (1.3 + \rightadj, -1.34641 + \heightadj) [circle, fill, minimum size = \nodesize pt, draw]{};
	\node (v6_d) at (1.7 + \rightadj, -1.34641 + \heightadj) [circle, fill, minimum size = \nodesize pt, draw]{};
	\drawedges{v6_d/v1_d, v1_d/v2_d, v2_d/v6_d}
	\labelnode[90]{v2_d}{$4$}
	\labelnode[-90]{v6_d}{$2$}
	\labelnode[0]{v1_d}{$3$}
		
	\node (comp_k_6) at (2.9 + \rightadj + \rightadj,0) {\large $\overline{K}_n$};
	
	\end{pic}
\end{center}
\caption{$P_6$ requires 5 edge reflections to reach the edgeless graph, but only 3 arbitrary reflections.}
\label{fig:P6_counterX}
\end{figure}

That is, after removing the orientations of the edges, we can start from $P_6$ at the top of the figure and take the counterclockwise path to $\overline{K}_n$ that has 3 edges. However, Theorem~\ref{thm-forests-m-transpositions} in the next section says that since $P_6$ is a tree with 6 vertices, it requires 5 edge reflections to reach $\overline{K}_n$. This is a first result toward filling in the bottom right corner of the table above.

In the next section, we find the least number of edge reflections required to reach the edgeless graph for several graph families. We begin with trees and forests, which are at the core of this study, and then move onto cycles, complete graphs, and nested triangles.

\section{Special Families}


\subsection{Trees and Forests}

In investigating the definition of edge reflection, it is easy to see that they cannot create cycles.

\begin{proposition}
\label{prop-forest-one-transposition}
Applying any single edge reflection to a forest results in a forest with one fewer edge.
\end{proposition}

\begin{proof}
Suppose $F$ is a forest and consider any possible edge reflection $t_{uv}^X$. We start applying this reflection by deleting the edge $uv$, thus breaking up one component into two components which are trees, one including the vertex $u$ and the other including $v$. Now for each vertex $w \in X \setminus \{u,v \}$, we know it is adjacent to exactly one of $u$ and $v$ since $F$ is a forest. Assume without loss of generality that $w$ is adjacent to $u$. We then delete the edge $uw$, breaking up the component containing $u$ into two smaller components. Next we add the edge $vw$, connecting the component containing $w$ to the component containing~$v$. Since these components are trees the resulting component is also a tree. This can be seen in Figure~\ref{fig:transp_on_forest}.

\begin{figure}[h]
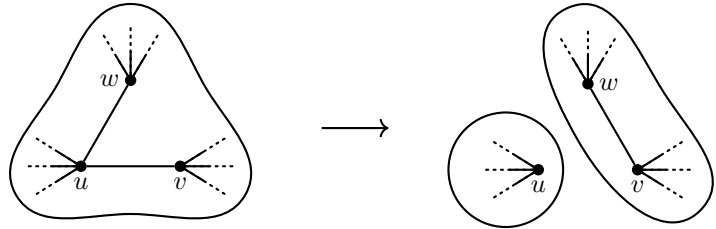

	\begin{center}
	\begin{pic}
		\makenodescircle{v}{3}{0.3}[90][circle, fill, minimum size = \nodesize pt, draw]
		\drawedges{v1/v2, v2/v3}
		\labelnode[180]{v1}{$w$}
		\labelnode[-90]{v2}{$u$}
		\labelnode[-90]{v3}{$v$}
		
		\begin{scope}   [dotted, thick, line cap = round] 
  	 		\draw (v1) to (-0.15,0.55); 
  	 		\draw (v1) to (0,0.6);
  	 		\draw (v1) to (0.15,0.55);
  	 		
  	 		\draw (v2) to (-0.5,0);
  	 		\draw (v2) to (-0.56,-0.15);
  	 		\draw (v2) to (-0.5,-0.3);
  	 		
  	 		\draw (v3) to (0.5,0);
  	 		\draw (v3) to (0.56,-0.15);
  	 		\draw (v3) to (0.5,-0.3);
		\end{scope}
		
		\begin{scope}   [thick, line cap = round] 
			\draw (v1) to (-0.075,0.425); 
			\draw (v1) to (0,0.45);
			\draw (v1) to (0.075,0.425);
			
			\draw (v2) to (-0.379903, -0.075);
			\draw (v2) to (-0.409903,-0.15);	
			\draw (v2) to (-0.379903, -0.15-0.075);
			
			\draw (v3) to (0.379903, -0.075);
			\draw (v3) to (0.409903,-0.15);	
			\draw (v3) to (0.379903, -0.15-0.075);
		\end{scope}
		
		\begin{scope}[on background layer]
		\draw[use Hobby shortcut] ([out angle=180,in angle=0]0, -0.4) .. (-0.6, -0.3) .. (-0.4, 0.25) .. (0, 0.7) .. (0.4, 0.25) .. (0.6, -0.3) .. (0, -0.4);
		
		\end{scope}
		\draw[->] (1,0.05) to (1.35,0.05);	
	\end{pic}
	\hspace{6mm}
	\begin{pic}
		\makenodescircle{v}{3}{0.3}[90][circle, fill, minimum size = \nodesize pt, draw]
		\labelnode[0]{v1}{$w$}
		\labelnode[-90]{v2}{$u$}
		\labelnode[-90]{v3}{$v$}
		\drawedges{v1/v3}
		
		\begin{scope}   [dotted, thick, line cap = round] 
  	 		\draw (v1) to (-0.15,0.55); 
  	 		\draw (v1) to (0,0.6);
  	 		\draw (v1) to (0.15,0.55);
  	 		
  	 		\draw (v2) to (-0.5,0);
  	 		\draw (v2) to (-0.56,-0.15);
  	 		\draw (v2) to (-0.5,-0.3);
  	 		
  	 		\draw (v3) to (0.5,0);
  	 		\draw (v3) to (0.56,-0.15);
  	 		\draw (v3) to (0.5,-0.3);
		\end{scope}
		
		\begin{scope}   [thick, line cap = round] 
			\draw (v1) to (-0.075,0.425); 
			\draw (v1) to (0,0.45);
			\draw (v1) to (0.075,0.425);
			
			\draw (v2) to (-0.379903, -0.075);
			\draw (v2) to (-0.409903,-0.15);	
			\draw (v2) to (-0.379903, -0.15-0.075);
			
			\draw (v3) to (0.379903, -0.075);
			\draw (v3) to (0.409903,-0.15);	
			\draw (v3) to (0.379903, -0.15-0.075);
		\end{scope}
		
		\begin{scope}[on background layer]
		\draw [] (-0.43, -0.15) circle (0.3);
		
		\draw[use Hobby shortcut] ([out angle=120,in angle=-60]-0.05, 0) .. (-0.1, 0.7) .. (0.4, 0.25) .. (0.6, -0.35)  .. (-0.05,0);
		
		\end{scope}	
	\end{pic}
	\end{center}
	\caption{An edge reflection on a forest; enclosed regions correspond to components.}
	\label{fig:transp_on_forest}
\end{figure}

 After doing this for all vertices $w \in X \setminus\{ u,v \}$, we see that the graph resulting from this reflection is a forest with one fewer edge.
\end{proof} 

The following result is then an easy consequence of Proposition~\ref{prop-forest-one-transposition}.

\begin{theorem}
\label{thm-forests-m-transpositions}
For a forest with $m$ edges, exactly $m$ edge reflections are required to transform it into the edgeless graph. Moreover, for a tree with \( n \) vertices, exactly \( n-1 \) edge reflections are required to reach the edgeless graph.
\end{theorem}

Given the characterization that the permutations that avoid $321$ and $3412$ are those whose inversion graphs are forests, it follows from Theorems~\ref{thm:edge_reflections_abs_length} and~\ref{thm-forests-m-transpositions} that their length and absolute length correspond. To obtain the other direction of Theorem~\ref{thm:boolean_perms}, assume a permutation $\pi \in S_n$ has a $321$ or $3412$ pattern. It is clear that one reduction can be applied to~$\pi$ to turn the $321$ pattern into a $123$ pattern, decreasing the number of inversions by 3, or two reductions can be applied to turn $3412$ into $1234$, decreasing the inversions by 4. The result then follows from the fact that every further reduction decreases the number of inversions by at least 1.

Along these same lines, it also follows that if the inversion graph $G_{\pi}$ of a permutation~$\pi$ is a tree, then~$\pi$ consists of a single cycle. Indeed, if~$\pi$ is a permutation of $[n]$ and $G_{\pi}$ is a tree, then~$\pi$ must have absolute length $n-1$, and since its absolute length is $n$ minus its number of cycles, this proves the claim.

\subsection{Cycles}

The result for cycles can be derived similarly.

\begin{theorem}
	The cycle $C_n$ requires exactly $n-2$ edge reflections to reach the edgeless graph.
\end{theorem}

\begin{proof}
	We consider the cycle $C_n$ with vertices $c_1, c_2, \dots,c_n$ and edges $c_i c_{i+1}$ for each $i \in [n]$ with indices taken modulo $n$. For $n=3$, we simply have a triangle and we can transform it into the edgeless graph with the edge reflection $t_{c_i, c_{i+1}}^{V(C_3)}$ for any $i \in \{ 1,2,3 \}$.
	
	Now suppose that $n \geq 4$, and that the result holds for $n-1$. If we apply some edge reflection~$t_{uv}^X$ to the cycle $C_n$, it must be that $uv = c_i c_{i+1}$ for some $i \in [n]$, and $X = \{ c_i, c_{i+1} \} \cup Y$ where $Y \subseteq \{ c_{i-1}, c_{i+2} \}$, (again, all indices taken modulo $n$). That is, there are essentially only 4 edge reflections that can be applied to this cycle, all of which are considered in Figure~\ref{fig:transposition_on_cycle}.

\begin{figure}[ht]
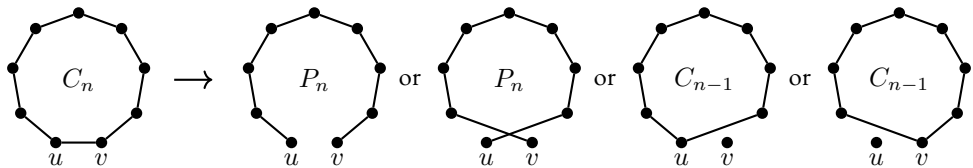

\begin{center}
	\begin{pic}
	\makenodescircle{v}{9}{0.35}[90][circle, fill, minimum size = \nodesize pt, draw]
	\drawedges{v1/v2, v2/v3, v3/v4, v4/v5, v5/v6, v6/v7, v7/v8, v8/v9, v9/v1}
	\labelnode[-90]{v5}{$u$}
	\labelnode[-90]{v6}{$v$}
	\node[] at (0,0) {$C_n$};
	\draw[->] (0.5,0) to (0.7,0);
	\end{pic}
\quad
\hspace{-2mm}
	\begin{pic}
	\makenodescircle{v}{9}{0.35}[90][circle, fill, minimum size = \nodesize pt, draw]
	\drawedges{v1/v2, v2/v3, v3/v4, v4/v5, v6/v7, v7/v8, v8/v9, v9/v1}
	\labelnode[-90]{v5}{$u$}
	\labelnode[-90]{v6}{$v$}
	\node[] at (0,0) {$P_n$};
	\node[] at (0.5,0) {\text{or}};
	\end{pic}
\quad
\hspace{-4mm}
	\begin{pic}
	\makenodescircle{v}{9}{0.35}[90][circle, fill, minimum size = \nodesize pt, draw]
	\drawedges{v1/v2, v2/v3, v3/v4, v4/v6, v5/v7, v7/v8, v8/v9, v9/v1}
	\labelnode[-90]{v5}{$u$}
	\labelnode[-90]{v6}{$v$}
	\node[] at (0,0) {$P_n$};
	\node[] at (0.5,0) {\text{or}};
	\end{pic}
\quad
\hspace{-4mm}
	\begin{pic}
	\makenodescircle{v}{9}{0.35}[90][circle, fill, minimum size = \nodesize pt, draw]
	\drawedges{v1/v2, v2/v3, v3/v4, v4/v5, v5/v7, v7/v8, v8/v9, v9/v1}
	\labelnode[-90]{v5}{$u$}
	\labelnode[-90]{v6}{$v$}
	\node[] at (0,0) {$C_{n-1}$};
	\node[] at (0.5,0) {\text{or}};
	\end{pic}
\quad
\hspace{-4mm}
	\begin{pic}
	\makenodescircle{v}{9}{0.35}[90][circle, fill, minimum size = \nodesize pt, draw]
	\drawedges{v1/v2, v2/v3, v3/v4, v4/v6, v6/v7, v7/v8, v8/v9, v9/v1}
	\labelnode[-90]{v5}{$u$}
	\labelnode[-90]{v6}{$v$}
	\node[] at (0,0) {$C_{n-1}$};
	\end{pic}
\end{center}
\caption{All possible outcomes of applying an edge reflection to a cycle.}
\label{fig:transposition_on_cycle}
\end{figure}

If $Y$ is $\emptyset$ or $\{ c_{i-1}, c_{i+2}\}$ then we see that we obtain the path $P_n$, otherwise if $|Y| = 1$ then we obtain the cycle $C_{n-1}$, (with an isolated vertex). Since $P_n$ is a tree on $n$ vertices, it would require a further $n-1$ edge reflections to reach the edgeless graph, but from the inductive hypothesis we have that $C_{n-1}$ requires $n-3$ transpositions to reach the edgeless graph. Since we have considered every possibility, it follows that~${n-2}$ edge reflections are required to empty $C_n$ of its edges.
\end{proof}

\subsection{Complete Graphs}

Similar to the well-studied concepts of edge and vertex covers in graph theory, we define an \emph{edge-edge cover} of a graph to be a subset of its edges so that all of the edges of the graph are incident to one of the edges in that set. 

\begin{proposition}\label{prop:edge-edge}
	For a graph $G$, if a minimum edge-edge cover contains $k$ edges, then $G$ requires at least $k$ edge reflections to reach the empty the graph.
\end{proposition}

\begin{proof}
	Suppose we apply the $j$ edge reflections $t_{u_1 v_1}^{X_1}, t_{u_2 v_2}^{X_2}, \dots, t_{u_j v_j}^{X_j}$ to $G$ where $j<k$. There then must be some edge $uv$ that is not incident to any of the edges $u_1 v_1, u_2 v_2, \dots , u_j v_j$. Thus, the resulting graph will still contain the edge $uv$.
\end{proof}

This proposition makes the following easy to prove.

\begin{theorem}
	The complete graph $K_n$ requires exactly $\lfloor \nicefrac{n}{2} \rfloor$ edge reflections to reach the edgeless graph.
\end{theorem}

\begin{proof}
	We have that $G_{e^{\text{r}}} = K_n$ where $e^{\text{r}} = n \dots 21$ is the reverse identity permutation. It is not too hard to show that $l'(G_{e^{\text{r}}}) = \lfloor \nicefrac{n}{2} \rfloor$, which is thus is an upper bound on the number of edge reflections required to reach $\overline{K}_n$ by Theorem~\ref{thm:edge_reflections_abs_length}. To see we can't do any better than this, it is easily seen that the cardinality of a minimum edge-edge cover of $K_n$ is $\lfloor \nicefrac{n}{2} \rfloor$, and so the result follows from Proposition~\ref{prop:edge-edge}.
\end{proof}

\subsection{Complete Bipartite Graphs}

We have another example that is easy to establish using Theorem~\ref{thm:edge_reflections_abs_length} and Proposition~\ref{prop:edge-edge}.

\begin{theorem}
	The complete bipartite graph $K_{m,m}$ requires exactly $m$ edge reflections to reach the edgeless graph.
\end{theorem}

\begin{proof}
	Letting $\sigma = 12\dots m \ominus 12\dots m \in S_{2m}$, we have that $K_{m,m}$ is isomorphic to $G_{\sigma}$. The permutation $\sigma$ is composed of $m$ 2-cycles, and hence $l'(\sigma) = m$. Then by Theorem~\ref{thm:edge_reflections_abs_length}, $K_{m,m}$ requires at most $m$ edge reflections to reach the edgeless graph. The result then follows from the observation that a minimum edge-edge cover of $K_{m,m}$ has cardinality~$m$.
\end{proof}

\subsection{Nested Triangles}

Finally, the graphs that require exactly one edge reflection to reach the edgeless graph are the \textit{nested triangles} seen in Figure~\ref{nested-triangle}. These are the split graphs given by $N_k = K_2 \vee \overline{K}_k$, for $k \geq 0$, with $\vee$ denoting the join operation. We note that $N_0$ is simply the single-edge graph $K_2$.

\begin{figure}[h]
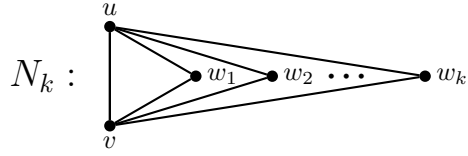

	\begin{center}
		\begin{pic}
			\makenodescircle{v}{3}{0.3}[0][circle, fill, minimum size = \nodesize pt, draw]
			
			\node at (-0.5, 0) {\large$N_k:$};
			
			\node (v4) at (0.7,0) [circle, fill, minimum size = \nodesize pt, draw]{};
			\node at (1.1, 0) {\Large \dots};
			\node (vk) at (1.5,0) [circle, fill, minimum size = \nodesize pt, draw]{};
			\drawedges{v1/v2, v2/v3, v3/v1, v2/v4, v3/v4, v2/v3, v2/vk, v3/vk}
	
			\labelnode[0]{v1}{$w_1$}
			\labelnode[0]{v4}{$w_2$}
			\labelnode[0]{vk}{$w_k$}
			\labelnode[90]{v2}{$u$}
			\labelnode[-90]{v3}{$v$}
		\end{pic}
	\end{center}
	\caption{The graphs requiring exactly one edge reflection to reach the edgeless graph.}
	\label{nested-triangle}
\end{figure}

\section{Extremal Considerations}

\subsection{The Extremal Bound}

As discussed in the previous section, trees on $n$ vertices require $n-1$ edge reflections to reach the edgeless graph. The next result shows that any $n$-vertex graph can always reach the edgeless graph with $n-1$ or fewer edge reflections.

\begin{theorem} \label{thm:first}
A graph on \( n \) vertices can be transformed into the edgeless graph with \( n-1 \) or fewer edge reflections.
\end{theorem}

\begin{proof}
Suppose $v_1, v_2, \dots , v_n $ are the vertices of the graph $G$. Beginning with $v_1$, if its neighborhood is nonempty, we consider its closed neighborhood $N [ v_1 ]$, (otherwise $v_1$ is already isolated from the rest of the vertices of $G$ and there is nothing to do). Then for any $v_k \in N(v_1)$, we can apply the edge reflection $t_{v_1 v_k}^{N[v_1]}$ to $G$, isolating the vertex $v_1$.

We continue going through the list of vertices in this manner, applying an edge reflection which isolates (at least) the current vertex if it is not already isolated. Eventually, we will reach the vertex $v_{n-1}$, which can only possibly be adjacent to $v_n$. If they are adjacent, we need only apply the edge reflection $t_{v_{n-1} v_n}^{\{v_{n-1}, v_n\}}$, which isolates both vertices. Otherwise, the graph is already edgeless. This process requires at most $n-1$ edge reflections.
\end{proof}

The reflection graph $\Gamma_n$ is not strongly connected, however Theorems~\ref{thm:first} and~\ref{thm-forests-m-transpositions} allow us to establish some concept of diameter in the next result.

\begin{theorem}
	For all $n \geq 1$, it holds that
	\[
	\max_{G \in \mathscr{G}_n} \min \{ k : G= G_0 \rightarrow \dots \rightarrow G_k = \overline{K}_n \text{ is a path in $\Gamma_n$} \} = n-1.
	\]
\end{theorem}

By Theorem~\ref{thm:first}, we see that we can always empty a graph of its edges with at most one fewer edge reflection than its number of vertices. Thus, considering the components of a graph independently, the next result follows quickly.

\begin{corollary}\label{cor-connected-components}
	A graph on $n$ vertices with $k$ components can be transformed into the edgeless graph with $n-k$ or fewer edge reflections. 
\end{corollary}


\subsection{Trees are the Only Extremal Graphs}

We now see that trees are actually the only graphs that require the extremal number of edge reflections to reach the edgeless graph.

\begin{theorem}\label{thm:tree_iff_n-1}
A graph on \( n \) vertices requires exactly \( n-1 \) edge reflections to become an edgeless graph if and only if it is a tree.
\end{theorem}

\begin{proof}
With Theorem~\ref{thm-forests-m-transpositions} and Corollary~\ref{cor-connected-components} established, we are left to consider connected graphs with a cycle. So suppose $G$ is a graph on $n$ vertices with an induced cycle $c_1 c_2 \dots c_k$. Considering the rest of the vertices $v_1, v_2, \dots, v_{n-k}$, we begin by selectively isolating vertices in this list. That is, reading from $v_1$ to $v_{n-k}$,
\begin{enumerate}
	\item if $v_i$ is adjacent to some vertex $v_j$ in $\{ v_{i+1}, \dots, v_{n-k}\}$ when we come to it, we apply the edge reflection $t_{v_i v_j}^{N[v_i]}$ thus isolating $v_i$, or
	\item if $v_i$ is not adjacent to any vertex in $\{v_{i+1}, \dots, v_{n-k} \}$ when we come to it, we do nothing.
\end{enumerate}
After these edge reflections are applied, we note that the set of vertices $\{v_1, \dots, v_{n-k} \}$ in the resulting graph is an independent set, and the vertices $c_1, \dots, c_k$ still form an induced cycle.

We next apply the edge reflections $t_{c_1 c_2}^{N[c_1]}$, $t_{c_2 c_3}^{N[c_2]}$, $\dots, t_{c_{k-3}c_{k-2}}^{N[c_{k-3}]}$ in this order, (we do nothing if $k=3$), thus isolating each of the vertices $c_1, \dots , c_{k-3}$. If there are any degree one vertices after performing these operations, we isolate each of them with the edge reflection that only deletes its incident edge. Now after relabeling the vertices $c_{k-2}$, $c_{k-1}$ and $c_k$ with $x$, $y$ and $z$, the resulting graph will appear as in Figure~\ref{fig:cyclic_graph_after_transpositions}, (omitting isolated vertices).

\begin{figure}[h]
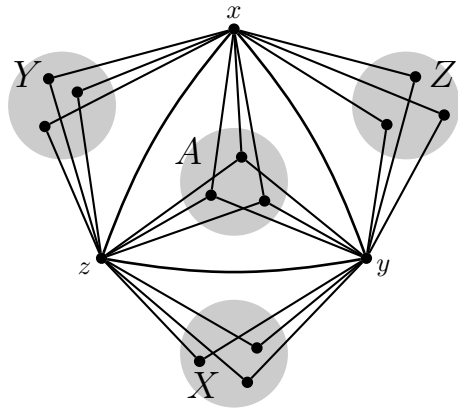

	\begin{center}
		\begin{pic}
			\makenodescircle{v}{3}{0.8}[90][circle, fill, minimum size = \nodesize pt, draw]
			\node (v4) at (-0.12, -0.07) [circle, fill, minimum size = \nodesize pt, draw]{};
			\node (v5) at (0.16,-0.1) [circle, fill, minimum size = \nodesize pt, draw]{};
			\node (vj) at (0.04, 0.13) [circle, fill, minimum size = \nodesize pt, draw]{};
			\node at (-0.24, 0.16) {\large$A$};
			\node (v6) at (0.8, 0.3) [circle, fill, minimum size = \nodesize pt, draw]{};
			\node (v7) at (1.1, 0.35) [circle, fill, minimum size = \nodesize pt, draw]{};
			\node (v8) at (0.95, 0.55) [circle, fill, minimum size = \nodesize pt, draw]{}; 
			\node [below right] at (1, 0.65) {\large$Z$};
			\node (v9) at (-0.97, 0.54) [circle, fill, minimum size = \nodesize pt, draw]{};
			\node (v10) at (-0.82, 0.47) [circle, fill, minimum size = \nodesize pt, draw]{};
			\node (v11) at (-0.99, 0.29) [circle, fill, minimum size = \nodesize pt, draw]{};
			\node [below left] at (-0.98, 0.65) {\large$Y$};
			\node (v13) at (0.12,-0.87) [circle, fill, minimum size = \nodesize pt, draw]{};
			\node (v14) at (-0.18,-0.94) [circle, fill, minimum size = \nodesize pt, draw]{};
			\node (v15) at (0.07, -1.05) [circle, fill, minimum size = \nodesize pt, draw]{};
			\node [above left] at (-0.05, -1.15) {\large$X$};
			
			\drawedges{v1/v6, v2/v4, v1/v4, v3/v4, v3/v6, v1/v5, v2/v5, v3/v5, v1/v7, v1/v8, v1/v9, v3/v7, v3/v8, v2/v9, v1/v10, v1/v11, v2/v10, v2/v11, v2/v13, v2/v14, v3/v13, v3/v14, v2/v15, v3/v15, v1/vj, v2/vj, v3/vj}
			
			\draw[line width = 1pt, out = -130, in = 70] (v1) to (v2); 
			\draw[line width = 1pt, out = -10, in = -170] (v2) to (v3);
			\draw[line width = 1pt, out = 110, in = -50] (v3) to (v1);
			
			\labelnode[90]{v1}{$x$}
			\labelnode[-30]{v3}{$y$}
			\labelnode[210]{v2}{$z$}
			
			\begin{scope}[on background layer]
			\node at  (0,0) [circle, fill, color = gray!40, minimum size = 40pt, draw]{};
			
			\node at  (0.9,0.4) [circle, fill, color = gray!40, minimum size = 40pt, draw]{};

			\node at  (-0.9,0.4) [circle, fill, color = gray!40, minimum size = 40pt, draw]{};
			
			\node at  (0,-0.9) [circle, fill, color = gray!40, minimum size = 40pt, draw]{};
			
			\end{scope}	
		\end{pic}
	\end{center}
	\caption{A connected graph with a cycle after some edge reflections are applied, (omitting isolated vertices).}
	\label{fig:cyclic_graph_after_transpositions}
\end{figure}

As in the figure, we consider the sets $X = (N(y) \cap N(z)) \setminus N[x]$, $Y = (N(x) \cap N(z)) \setminus N[y]$, $Z = (N(x) \cap N(y)) \setminus N[z]$ and $A = N(x) \cap N(y) \cap N(z)$. Supposing that this remaining component has $\ell$ vertices and noting every edge reflection applied thus far isolated at least one vertex, it remains to show we can reach the edgeless graph with $\ell-2$ or fewer edge reflections. There are two cases to deal with.

\textit{Case 1:} At most one of $X$, $Y$ and $Z$ is nonempty. Suppose without loss of generality that $Y = Z = \emptyset$. Then after applying the edge reflection $t_{yz}^{N(y) \cup N(z)}$, the only remaining edges are between the vertex $x$ and each of the vertices of $A$. Those edges can be deleted each one at a time with an edge reflection. That is a total of $\ell - 2 - |X|$ edge reflections. See Figure~\ref{fig:case1}.

\begin{figure}[ht]
\begin{center}
	\begin{pic}
	\makenodescircle{v}{3}{0.4}[90][circle, fill, minimum size = \nodesize pt, draw]
	\drawedges{v1/v2, v2/v3, v3/v1}
	\labelnode[90]{v1}{$x$}
	\labelnode[210]{v2}{$z$}
	\labelnode[-30]{v3}{$y$}
	
	\node at  (0,0) [circle, fill, color = gray!40, minimum size = 17pt, draw]{};
	\node (a) at (0,0) [circle, fill, minimum size = \nodesize pt, draw]{};
	\node[] at (-0.1, 0.1) {$A$};
	
	\node at  (0,-0.4) [circle, fill, color = gray!40, minimum size = 17pt, draw]{};
	\node (X) at (0,-0.4) [circle, fill, minimum size = \nodesize pt, draw]{};
	\node[] at (-0.1, -0.5) {$X$};
	
	\drawedges{a/v1, a/v2, a/v3, v2/X, v3/X}
	\draw[->] (0.6,0) to (0.9,0);
	\end{pic}
\quad
\hspace{-1mm}
	\begin{pic}
	\makenodescircle{v}{3}{0.4}[90][circle, fill, minimum size = \nodesize pt, draw]
	\labelnode[90]{v1}{$x$}
	
	\node at  (0,0) [circle, fill, color = gray!40, minimum size = 17pt, draw]{};
	\node (a) at (0,0) [circle, fill, minimum size = \nodesize pt, draw]{};
	\node[] at (-0.1, 0.1) {$A$};
	
	\node at  (0,-0.4) [circle, fill, color = gray!40, minimum size = 17pt, draw]{};
	\node (x) at (0,-0.4) [circle, fill, minimum size = \nodesize pt, draw]{};
	\drawedges{v1/a}
	
	\draw[->] (0.65,0) to (0.85,0);
	\draw[->] (1.25,0) to (1.45,0);
	\node[] at (1.0, 0.15) {\scriptsize $|A|$ edge reflections};
	\node[] at (1.05, 0) {$\dots$};
	\node[] at (1.7, 0) {\large $\overline{K}_n$};
	\end{pic}
\end{center}
\caption{Case 1 in the proof of Theorem~\ref{thm:tree_iff_n-1}.}
\label{fig:case1}
\end{figure}

\textit{Case 2:} At most one of $X$, $Y$ and $Z$ is empty. Thus $\ell$ is at least $|A| + 5$, and we can reach the empty graph with $|A| + 3 = \ell - (|X| + |Y| + |Z|)$ edge reflections as follows. We apply the edge reflection $t_{yz}^{X \cup A \cup \{y,z \}}$, isolating the vertices in $X$ and removing all edges between $A$ and $\{y,z\}$, then~$t_{xy}^{Z \cup \{ y,z\}}$, isolating $z$ and the vertices in $Y$, and then~$t_{xz}^{Y \cup \{ x,z\}}$, isolating $y$ and the vertices in $Z$. We can then delete the remaining $|A|$ edges, all between $A$ and the vertex $x$, with one edge reflection each. See Figure~\ref{fig:case2}.

\begin{figure}[ht]
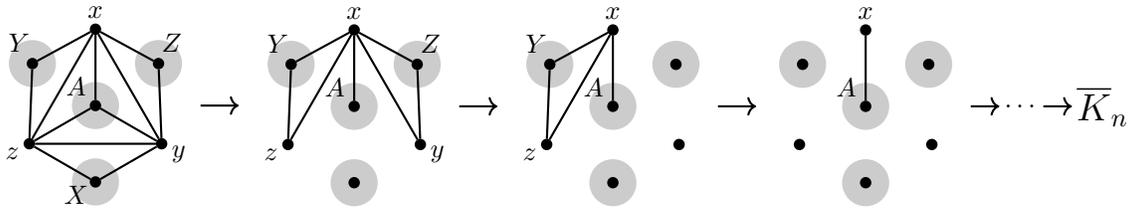

\begin{center}
	\begin{pic}
	\makenodescircle{v}{3}{0.4}[90][circle, fill, minimum size = \nodesize pt, draw]
	\drawedges{v1/v2, v2/v3, v3/v1}
	\labelnode[90]{v1}{$x$}
	\labelnode[210]{v2}{$z$}
	\labelnode[-30]{v3}{$y$}
	\node at  (0,0) [circle, fill, color = gray!40, minimum size = 17pt, draw]{};
	\node (a) at (0,0) [circle, fill, minimum size = \nodesize pt, draw]{};
	\node[] at (-0.1, 0.1) {$A$};
	\node at  (0,-0.4) [circle, fill, color = gray!40, minimum size = 17pt, draw]{};
	\node (X) at (0,-0.4) [circle, fill, minimum size = \nodesize pt, draw]{};
	\node[] at (-0.1, -0.47) {$X$};
	\node at  (0.33,0.22) [circle, fill, color = gray!40, minimum size = 17pt, draw]{};
	\node (Y) at (0.33,0.22) [circle, fill, minimum size = \nodesize pt, draw]{};
	\node[] at (0.4, 0.33) {$Z$};
	\node at  (-0.33,0.22) [circle, fill, color = gray!40, minimum size = 17pt, draw]{};
	\node (Z) at (-0.33,0.22) [circle, fill, minimum size = \nodesize pt, draw]{};
	\node[] at (-0.4, 0.33) {$Y$};
	\drawedges{a/v1, a/v2, a/v3, v2/X, v3/X, v3/Y, v1/Y, v2/Z, v1/Z}
	\draw[->] (0.55,0) to (0.75,0);
	\end{pic}
\quad
\hspace{-3mm}
	\begin{pic}
	\makenodescircle{v}{3}{0.4}[90][circle, fill, minimum size = \nodesize pt, draw]
	\drawedges{v1/v2, v3/v1}
	\labelnode[90]{v1}{$x$}
	\labelnode[210]{v2}{$z$}
	\labelnode[-30]{v3}{$y$}
	\node at  (0,0) [circle, fill, color = gray!40, minimum size = 17pt, draw]{};
	\node (a) at (0,0) [circle, fill, minimum size = \nodesize pt, draw]{};
	\node[] at (-0.1, 0.1) {$A$};
	\node at  (0,-0.4) [circle, fill, color = gray!40, minimum size = 17pt, draw]{};
	\node (X) at (0,-0.4) [circle, fill, minimum size = \nodesize pt, draw]{};
	\node at  (0.33,0.22) [circle, fill, color = gray!40, minimum size = 17pt, draw]{};
	\node (Y) at (0.33,0.22) [circle, fill, minimum size = \nodesize pt, draw]{};
	\node[] at (0.4, 0.33) {$Z$};
	\node at  (-0.33,0.22) [circle, fill, color = gray!40, minimum size = 17pt, draw]{};
	\node (Z) at (-0.33,0.22) [circle, fill, minimum size = \nodesize pt, draw]{};
	\node[] at (-0.4, 0.33) {$Y$};
	\drawedges{a/v1, v3/Y, v1/Y, v2/Z, v1/Z}
	\draw[->] (0.55,0) to (0.75,0);
	\end{pic}
\quad
\hspace{-3mm}
	\begin{pic}
	\makenodescircle{v}{3}{0.4}[90][circle, fill, minimum size = \nodesize pt, draw]
	\drawedges{v1/v2}
	\labelnode[90]{v1}{$x$}
	\labelnode[210]{v2}{$z$}
	\node at  (0,0) [circle, fill, color = gray!40, minimum size = 17pt, draw]{};
	\node (a) at (0,0) [circle, fill, minimum size = \nodesize pt, draw]{};
	\node[] at (-0.1, 0.1) {$A$};
	\node at  (0,-0.4) [circle, fill, color = gray!40, minimum size = 17pt, draw]{};
	\node (X) at (0,-0.4) [circle, fill, minimum size = \nodesize pt, draw]{};
	\node at  (0.33,0.22) [circle, fill, color = gray!40, minimum size = 17pt, draw]{};
	\node (Y) at (0.33,0.22) [circle, fill, minimum size = \nodesize pt, draw]{};
	\node at  (-0.33,0.22) [circle, fill, color = gray!40, minimum size = 17pt, draw]{};
	\node (Z) at (-0.33,0.22) [circle, fill, minimum size = \nodesize pt, draw]{};
	\node[] at (-0.4, 0.33) {$Y$};
	\drawedges{a/v1, v2/Z, v1/Z}
	\draw[->] (0.55,0) to (0.75,0);
	\end{pic}
\quad
\hspace{-3mm}
	\begin{pic}
	\makenodescircle{v}{3}{0.4}[90][circle, fill, minimum size = \nodesize pt, draw]
	\labelnode[90]{v1}{$x$}
	\node at  (0,0) [circle, fill, color = gray!40, minimum size = 17pt, draw]{};
	\node (a) at (0,0) [circle, fill, minimum size = \nodesize pt, draw]{};
	\node[] at (-0.1, 0.1) {$A$};
	\node at  (0,-0.4) [circle, fill, color = gray!40, minimum size = 17pt, draw]{};
	\node (X) at (0,-0.4) [circle, fill, minimum size = \nodesize pt, draw]{};
	\node at  (0.33,0.22) [circle, fill, color = gray!40, minimum size = 17pt, draw]{};
	\node (Y) at (0.33,0.22) [circle, fill, minimum size = \nodesize pt, draw]{};
	\node at  (-0.33,0.22) [circle, fill, color = gray!40, minimum size = 17pt, draw]{};
	\node (Z) at (-0.33,0.22) [circle, fill, minimum size = \nodesize pt, draw]{};
	\drawedges{a/v1}
	\draw[->] (0.55,0) to (0.7,0);
	\draw[->] (0.93,0) to (1.08,0);
	\node[] at (0.82, 0) {$\dots$};
	\node[] at (1.24, 0) {\large $\overline{K}_n$};
	\end{pic}
\end{center}
\caption{Case 2 in the proof of Theorem~\ref{thm:tree_iff_n-1}.}
\label{fig:case2}
\end{figure}

This gives the result.
\end{proof}

Thus, if the permutation $\pi \in S_n$ consists of single cycle but $G_{\pi}$ is not a tree, then the upper bound of $l'(\pi) = n-1$ edge reflections required for $G_{\pi}$ to reach the edgeless graph from Theorem~\ref{thm:edge_reflections_abs_length} is not tight. Recalling our previous discussions, these are exactly the permutations that consist of one cycle that cannot be written as $(c_1, c_2, \dots, c_n)$ with $i \in [n]$ such that $c_1 < c_2 < \dots < c_i > \dots > c_n$.

By following the algorithm in the proof of Theorem~\ref{thm:tree_iff_n-1}, it is clear that in many cases we can empty an $n$-vertex graph containing a cycle in even fewer than $n-2$ edge reflections. That is, suppose that $G$ is an $n$-vertex graph with an induced cycle $C$, and that we followed the algorithm described in this proof so that we are at the point of Figure~\ref{fig:cyclic_graph_after_transpositions}. Consulting both of the cases, it follows that we can empty the graph of its edges it in at least $|X|+|Y|+|Z|$ fewer than $n-2$ edge reflections. Some thought shows that each of the vertices in $X$, $Y$, and $Z$ is the remnant of a component of the graph $G[V(G)-C]$ with an even number of edges between it and the cycle $C$ in~$G$. The key observation is that for every edge reflection applied up to that point, the number of edges between a component of $G[V(G)-C]$ and $C$ can only decrease by an even amount. We observe that there are also components of $G[V(G)-C]$ such that all of their vertices became isolated prior to this point, and for each of those, we also `saved' an edge reflection. These observations yield the following.

\begin{corollary}
	Suppose $G = (V,E)$ is a graph on $n$ vertices with an induced cycle $C$. Let $k$ be the number of components $K$ of $G[V-C]$ such that there are an even number of edges between $K$ and $C$ in $G$. Then $G$ can be transformed into the edgeless graph in $(n - 2) - k$ or fewer edge reflections. 
\end{corollary}

\section{Other Parameters and Further Directions}

There are many questions to explore about these edge reflections and the reflection graph~$\Gamma_n$. In this section, we outline some of these problems, hopefully giving rise to interesting combinatorial questions. We begin with a proposition.

\begin{proposition}
	Suppose that $G$ is a graph and $\Pi$ is a partition of the edges of $G$ into blocks that induce nested triangle graphs. Further, let $n_k$ denote the number of blocks of $\Pi$ that induce the nested triangle graph $N_k$ for each $k \geq 0$. Then the graph $G$ can be transformed into the edgeless graph with
	\begin{equation} \tag{$\ddag$} \label{eqn:partition_bound}
		\left| \Pi \right| = m - 2 \sum_{k \geq 0} k \cdot n_k
	\end{equation}
	edge reflections.
\end{proposition}

\begin{proof}
	For each block of $\Pi$, simply apply the edge reflection that removes all of its edges. It is not difficult to see that both sides of (\ref{eqn:partition_bound}) are equivalent. 
\end{proof}

It is thus interesting to ask for which graphs does there exist a partition of its edges into nested triangles such that the bound given by this proposition is optimal. For example, the bound is optimal for all forests. We include the following conjecture, noting that a graph is \emph{chordal} if all of its induced cycles are triangles.

\begin{conjecture}
	For all chordal graphs $G$, the minimum number of edge reflections required to obtain the empty graph is given by some partition in the previous theorem.
\end{conjecture}

We mention that there are graphs that are not chordal such that there exists a partition of its edges into nested triangles with (\ref{eqn:partition_bound}) giving the optimal bound. See, for example, the square with spikes in Figure~\ref{fig:square_with_spikes}.

\begin{figure}[ht]
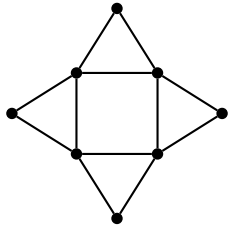

\begin{center}
	\begin{pic}
	\makenodescircle{v}{4}{0.3}[45][circle, fill, minimum size = \nodesize pt, draw]
	\drawedges{v1/v2, v2/v3, v3/v4, v4/v1}

	\node (a) at (0,0.55) [circle, fill, minimum size = \nodesize pt, draw]{};
	\node (b) at (-0.55,0) [circle, fill, minimum size = \nodesize pt, draw]{};
	\node (c) at (0,-0.55) [circle, fill, minimum size = \nodesize pt, draw]{};
	\node (d) at (0.55,0) [circle, fill, minimum size = \nodesize pt, draw]{};
	
	\drawedges{a/v1, a/v2, b/v2, b/v3, c/v3, c/v4, d/v4, d/v1}
	\end{pic}
\end{center}
\caption{The square with spikes graph.}
\label{fig:square_with_spikes}
\end{figure}

As a very general problem, we propose the following.

\begin{problem}
	Investigate bounds on the number of edge reflections required to reach the edgeless graph in terms of other graph parameters. For example, consider parameters such as diameter, connectivity, girth, etc.
\end{problem}

%% file: bio.tex
Sean Patrick Mandrick was born in 1999 in Rochester, New York, where he was also raised. He attended Colgate University from 2017 to 2020, earning a BA in mathematics with a minor in computer science. He then studied mathematics at the University of Florida from 2021 to 2025, earning an MS in 2023 and a PhD in 2025.